\documentclass[oneside, a4paper]{amsart}
\usepackage[a4paper, width=15cm, height=24cm]{geometry}
\usepackage[onehalfspacing]{setspace}
\usepackage{graphicx}
\usepackage{tikz}
\usepackage{amsfonts, amsmath, amssymb, amsthm}
\usepackage[T1]{fontenc}
\usepackage{lmodern,dsfont}
\usepackage{overpic}
\usepackage[english]{babel}
\usepackage{bm}
\usepackage{MnSymbol}
\usepackage{hyperref}
\newcommand{\inner}[1]{\llangle #1 \rrangle}
\newcommand{\innerM}[1]{\langle #1 \rangle}
\newcommand{\spann}[1]{\text{span} \{ #1 \}}
\newcommand{\spannR}[1]{\text{span}_{\mathbb{R}} \{ #1 \}}
\newcommand{\f}{\mathfrak{f}}
\newcommand{\dir}{\mathfrak{a}}
\newcommand{\nor}{\mathfrak{n}}
\newcommand{\tang}{\mathfrak{t}}
\newcommand{\frakt}{\mathfrak{t}}
\newcommand{\q}{\mathfrak{q}}
\newcommand{\p}{\mathfrak{p}}
\newcommand{\Light}{\mathbb{P}(\mathcal{L})}
\newcommand{\Q}{\mathcal{Q}}
\newcommand{\ffminus}{\mathfrak{f}_{\bar{1}}}

\newcommand{\ffplus}{\mathfrak{f}_{1}}
\newcommand{\fminus}{f_{\bar{1}}}
\newcommand{\fplus}{f_{1}}
\newcommand{\C}{\mathcal{C}_\mathcal{Q}}
\newcommand{\CE}{\mathcal{C}_{\mathbb{E}^2}}
\newcommand{\CS}{\mathcal{C}_{\mathbb{S}^2}}
\newcommand{\CH}{\mathcal{C}_{\mathbb{H}^2}}
\newcommand{\QE}{\mathbb{E}^2}
\newcommand{\QS}{\mathbb{S}^2}
\newcommand{\QH}{\mathbb{H}^2}
\newcommand{\dDelta}{\zeta}
\newcommand{\EucDelta}{\eta}
\newcommand{\ninv}{\theta}
\newcommand{\asso}{\mathcal{T}}
\newcommand{\matC}{\mathbb{C}^{2\times2}}
\newcommand{\G}{\mathcal{G}}
\newcommand{\V}{\mathcal{V}}
\newcommand{\E}{\mathcal{E}}
\newcommand{\R}{\mathbb{R}}

\DeclareMathOperator{\tr}{tr}
\DeclareMathOperator{\id}{id}
\newtheoremstyle{dotless}{6pt}{18pt}{}{}{\bfseries}{.}{\newline}{}
\theoremstyle{dotless}
\newtheorem{thm}{Theorem}
\newtheorem{defi}[thm]{Definition}
\newtheorem{lem}[thm]{Lemma}
\newtheorem{ex}[thm]{Example}

\newtheorem{rem}[thm]{Remark}
\newtheorem{prop}[thm]{Proposition}
\newtheorem{propdefi}[thm]{Proposition and Definition}
\newtheorem{cor}[thm]{Corollary}
\title[Discrete curve theory in space forms]{Discrete curve theory in space forms: planar elastic and area-constrained elastic curves}
\author{Tim Hoffmann, Jannik Steinmeier and Gudrun Szewieczek}
\begin{document}
\bibliographystyle{plainurl}
%%%%%%%%%%%%%%%%%%%%%%%%%%%%
\maketitle
%%%%%%%%%%%%%%%%%%%%%%%%%%%%
\color{black}
\begin{center}
\begin{minipage}{13cm}\small
\textbf{Abstract.} We propose a notion of discrete elastic and area-constrained elastic curves in 2-dimensional space forms. Our definition extends the well-known discrete Euclidean curvature equation to space forms and reflects various geometric properties known from their smooth counterparts. Special emphasis is paid to discrete flows built from B\"acklund transformations in the respective space forms. The invariants of the flows form a hierarchy of curves and we show that discrete elastic and constrained elastic curves can be characterized as elements of this hierarchy. 

This work also includes an introductory chapter on discrete curve theory in space forms, where we find discrete Frenet-type formulas and describe an associated family related to a fundamental theorem. 
\end{minipage}
%%%%%%%%%%%%%%%%%%%%%%%
%%%%%%%%%%%%%%%%%%%%%%%
\\ \ \\ \ \\\begin{minipage}{13cm}\small
\end{minipage}
%%%%%%%%%%%%%%%%%%%%%%%
%%%%%%%%%%%%%%%%%%%%%%%
\\ \ \\\begin{minipage}{13cm}\small
\end{minipage}
\end{center}
%%%%%%%%%%%%%%%%%%%%%%%
%%%%%%%%%%%%%%%%%%%%%%%
\section{Introduction}
The study of smooth planar elastic curves, initiated by Bernoulli and Euler \cite{euler_elastic}, was originally driven by a physical problem: what shape can an elastic rod take? Using a variational approach, solutions to this problem arise as critical values of the bending energy, i.\,e., the averaged squared curvature of the curve, while constraining the length. 

Over time, elastic curves have been investigated from various differential geometric perspectives, showing that they admit a number of interesting properties \cite{chopembszew, euler_elastic, langer_singer, tjaden}. Remarkably, it has turned out that the class of smooth elastic curves belongs to a hierarchy of commuting energy functionals~\cite{ LangerPerlineHierarchy}. This point of view naturally leads to the class of area-constrained elastic curves, which are critical values of the bending energy under constrained length and oriented area \cite{PhysRevE.65.031801, pinkall_willmore, pinkall_book_dg}. 
%Furthermore, 
The symplectic gradients of these energy functionals constitute a hierarchy of commuting flows of space curves \cite{CKPPcommuting} which serves as geometric approach to the non-linear Schr\"odinger hierarchy. Every other flow is contained in a plane (if applied to a planar curve) and this sub-hierarchy is known as the modified Korteweg-de Vries (mKdV) hierarchy \cite{goldstein_KdV,matsutani2016euler}. Elastic curves are the invariant curves of the Hashimoto flow, while constrained elastic curves are the invariants of the mKdV flow.

Due to their global nature, constrained elastic curves are also intimately connected with (global) surface theory. The most prominent example is provided by Wente's CMC-torus on which the planar elastic figure eight arises as planar curvature line~\cite{Abresch1987}. Considering constrained elastic curves in space forms~\cite{ Heller_elastic, langer_singer} has then revealed further beautiful relations to integrable surface theory. For example, those can be used to generate examples of constrained Willmore tori~\cite{pinkall_willmore, Heller_elastic} and of special isothermic channel surfaces of type~$d$~\cite{iso_type_d}. Moreover, constrained elastic curves in space forms and certain transformations of them foliate isothermic (and, more generally, Lie applicable) surfaces with a family of spherical curvature lines \cite{bobenko2023iso, chopembszew}.

\bigskip

A discrete notion of elastic curves in Euclidean space has been introduced in \cite{lagrangeTop}. This integrable discretization relies on the variational approach. Further studies, based on discrete Frenet formulas, can be found in \cite{fairing_elastica}. The latter also demonstrates the relevance of elastic curves in applications by using them for fairing planar curve (see~\cite{Brander_elastic} for a smooth version of this algorithm). As indicated in~\cite{elastic_hotblade, hafner2021tdsopec}, elastic curves are also useful in the design of curved surfaces, and, as demonstrated in \cite{discrete_elastic_rods}, for physical simulation of elastic material.

Discrete elastic curves are known to be invariant under a semi-discrete Hashimoto flow \cite{lagrangeTop}. Similarly, discrete area-constrained elastic curves in Euclidean space, characterized as curves that are invariant under a semi-discrete mKdV-flow, have been introduced in \cite{HKtodaLattice}. 

One can also discretize the flow direction. In \cite{hoffmannSmokeRingFlow, pinkallSmokeRingFlow} a discrete Hashimoto flow for space curves is built from pairs of B\"acklund transformations. The invariant curves of this discrete flow are, again, discrete elastica. More generally, invariant curves of sequences of B\"acklund transformations give rise to a hierarchy of curves \cite{skewParallelogramNets}. In the plane this type of transformation is also known as the discrete bicycle transformation \cite{discrete_bicycle} and constitutes a special case of a cross-ratio evolution \cite{affolter2023integrable,AFIT_dynamics,HMNP_periodicConformalMaps} (with positive cross-ratio). As such, it is also closely related to the Darboux transformation of isothermic nets \cite{BP_discreteIsothermicSurfaces,CHO2023102065}. B\"acklund and Darboux transformations of discrete curves are also known to provide a (semi-)discrete potential mKdV equation closely related to the discrete sine-Gordon equation \cite{discretemKdV,Inoguchi_2014}. 
 
\bigskip

Generalizing these previous works, in this paper we propose a discrete definition for elastic and area-constrained elastic curves in 2-dimensional space forms. To the best of our knowledge this is the first time that these discrete curves are studied in ambient spaces of constant sectional curvature.

Based on a discrete version of the well-known curvature equation for these classes of curves, in Section~\ref{sec:constr_elastic} we prove a variety of properties expected from the smooth theory. In more detail, we show existence of an associated straight or circular directrix which implies some proportionality for the curvature~(see Proposition~\ref{prop:proportionality}). 
 
Section~\ref{sec:flows} is devoted to the study of \emph{$n$-invariant} curves which are invariant under an integer number of B\"acklund transformations in the plane and this approach directly extends to non-Euclidean geometry. These curves constitute a hierarchy of planar discrete curves and we find elastic and constrained elastic curves as elements of this hierarchy (see Theorem~\ref{thm:invariantCurvesConstElastica}). Our observations highlight the close relation between B\"acklund transformations and the hierarchy of flows.

\bigskip

In order to address the various geometric characteristics of (area-constrained) elastic curves, the technical framework is adapted to the specific problem at hand, and alternative models for space form geometry are employed. 
While the light cone model provides a direct demonstration of the existence of directrixes and related curvature properties, a matrix model allows us to study curves, their associated family and their B\"acklund transformations in terms of quaternionic polynomials which come with a well-known factorization theory \cite{gordon1965zeros,izosimov2023recutting, niven1941equations}.

Furthermore, these two approaches enable us to establish a direct relationship between the developed theory on discrete constrained elastic curves and further recent research in this field. 
For instance, \cite{skewParallelogramNets} (see also \cite{skewParNetsThesis}) aims to unify polynomial integrable systems studied in discrete differential geometry by factorization into a common linear system, called \emph{skew parallelogram net}. We apply techniques  (like the associated family and B\"acklund transformations) that exist for this more general setting and show that constrained elastic curves are described by such a polynomial system as well. 

Moreover, in \cite{discrete_iso_spherical} it is proven that the discrete constrained elastic curves introduced in this paper, build the foundations of discrete isothermic nets with spherical curvature lines. This relies on the fact that constrained elastic curves can be extended to specific discrete holomorphic maps that admit lifted-foldings to these special isothermic nets.

\bigskip
    
For the convenience of the reader, we provide a detailed introduction to the used models, including conversion formulas between them in Section~\ref{sec:discrete_curve_theory}. In those two frameworks, we then develop some basics of discrete curve theory in space forms. In particular, we provide discrete Frenet-type equations and relate the fundamental theorem of curves to the construction of associated families.

\bigskip

\noindent\textbf{Acknowledgements.} We thank Andrew Sageman-Furnas and Jack Reever for stimulating discussions on this topic. This research was partially supported by the DFG Collaborative Research Center TRR 109 "Discretization in Geometry and Dynamics". 
%%%%%%%%%%%%%%%%%%%%%%%%%%%%%%%%%%%%%%%%%%%%%%%
%
%
%%%%%%%%%%%%%%%%%%%%%%%%%%%%%%%%%%%%%%%%%%%%%%%
\section{Discrete Curve Theory in space forms}\label{sec:discrete_curve_theory}
\noindent We study discrete curves in 2-dimensional connected, geodesically complete Riemannian manifolds~$\Q$ of constant sectional curvature~$\kappa_Q$, so-called \emph{space forms}. Depending on the sign of the sectional curvature, we obtain three types which can be represented by the following standard models:
\begin{itemize}
\item \emph{Euclidean space forms} $(\kappa_Q=0)$ modeled as the Euclidean plane $\QE$;
\item \emph{hyperbolic space forms} $(\kappa_Q<0)$ represented either by the Poincar\'e half-plane, the disc model or the unit hyperboloid 
\begin{equation*}
    \QH=\{(x,y,z)\in\R^3\,|\,-x^2-y^2+z^2=1\}    
\end{equation*}
together with the Minkowski bilinear form induced by $(x,y,z)\mapsto -x^2-y^2+z^2$.
\item \emph{spherical space forms} $(\kappa_Q>0)$ canonically considered as the unit sphere 
\begin{equation*}
    \QS=\{(x,y,z)\in\R^3\,|\,x^2+y^2+z^2=1\}
\end{equation*}
equipped with the standard scalar product.
\end{itemize}

\bigskip

The aim of this section is to provide a framework for discrete curve theory in 2-dimensional space forms, which extends established concepts for discrete curves in Euclidean space. Note that throughout the paper we will restrict to discrete curves parametrized by arc-length.

Our intention is to define the basic concepts, such as tangent line congruences and curvature circles, in the most geometric manner possible. These will be followed by a detailed description of the geometric objects in the light cone model as well as in a matrix model.

\subsection{Various models for space forms}

\subsubsection{Light cone model} In this model all space form geometries are considered as subgeometries of Lie sphere geometry, the contact geometry for oriented circles. Here we intend to give a brief geometric introduction to this topic. For more details the reader is referred
to the exhaustive literature in this area, for example, to the surveys \cite{blaschke} and \cite{book_cecil}.

We work in a 5-dimensional vector space~$\mathbb{R}^{3,2}$ equipped with an inner product of signature~$(3,2)$ denoted by $\inner{. ,.}$. 

The central object in this model is the projective light-cone
\begin{equation*}
\Light:= \{ \mathbb{R}\mathfrak{v} \subset \mathbb{R}^{3,2} \ | \ \inner{\mathfrak{v}, \mathfrak{v}}=0, \mathfrak{v} \neq 0 \}.
\end{equation*}
Elements in $\Light$ are then in 1-to-1 correspondence with the set of points, oriented circles and oriented lines in $\mathbb{R}^2$ via the following identification table (cf.\,\cite{book_cecil})
\\\begin{center}
\begin{tabular}{p{5.5cm} | p{7.5cm}}
\textbf{Geometric objects in $\mathbb{R}^2 \cup \infty$} & \textbf{Vectors in $\Light$}
\\\hline $\infty$ & $\mathbb{R}\q_0$ with $\q_0:=( 0,0, 1, -1, 0 )$
\\\hline point $(x,y) \in \mathbb{R}^2$ & $\mathbb{R}(x ,y, \frac{1}{2}(1-x^2-y^2), \frac{1}{2}(1+x^2+y^2), 0)$
\\\hline oriented circle with radius $r \in \mathbb{R}^\times$ and center $(x,y) \in \mathbb{R}^2$ & $\mathbb{R}(x ,y, \frac{1}{2}(1-x^2-y^2+r^2), \frac{1}{2}(1+x^2+y^2-r^2), r)$
\\\hline oriented line with normal distance $d$ and unit normal vector $(x,y) \in \mathbb{R}^2$ & $\mathbb{R}(x,y ,-d,d,1)$
\end{tabular}
\end{center} 
\ \\As common in M\"obius geometry, we usually consider oriented lines also as circles with infinite radius.

\bigskip

Moreover, throughout this work we will use the following notation convention: homogeneous coordinates of elements in the projective space $\mathbb{P}(\mathbb{R}^{3,2})$ will be denoted by the
corresponding black letter; if statements hold for arbitrary
homogeneous coordinates we will use this convention without 
explicitly mentioning it.

\bigskip

In addition to this table, we further define a so-called \emph{point-sphere complex}~$\p \in \mathbb{R}^{3,2}$ as 
\begin{equation*}
\p:=(0,0,0,0,1).
\end{equation*} 

An element $v \in \Light$ then represents a point in~$\mathbb{R}^2$ if and only if $\inner{\mathfrak{v}, \p}=0$. Similarly, the vector~$\mathfrak{q}_0$ is used to identify Euclidean lines: $v \in \Light$ is identified with an oriented line in $\mathbb{R}^2$ if and only if $\inner{\mathfrak{v}, \q_0 }=0$.

\bigskip

Various geometric relations between objects in $\mathbb{R}^2$ can now be expressed via simple formulas for the corresponding elements in the light cone. In what follows, we list the most important relations for this text (for detailed proofs see for example \cite{blaschke}). 

Two geometric objects $u,v$ in $\mathbb{R}^2$ are in \emph{oriented contact}, that is, two circles are tangent with corresponding orientation or, a point lies on a circle, if and only if $\inner{\mathfrak{u}, \mathfrak{v}}=0$. The (unoriented) angle $\varphi$ between two circles $u,v
\in \Light$ is given by
\begin{equation}\label{equ_intersection_angle}
 \cos \varphi
  = 1 + \frac{
   \inner{\mathfrak{u},\mathfrak{v}}}
   {\inner{\mathfrak{u},\mathfrak{p}}
   \inner{\mathfrak{v},\mathfrak{p}}}.
\end{equation} 
In particular, two oriented circles $u,v \in \Light$ intersect orthogonally if and only if $\inner{\mathfrak{u},\mathfrak{v}+\inner{\mathfrak{v},\mathfrak{p}}\mathfrak{p}}=0$.

\bigskip

Any element $a \in \mathbb{P}(\mathbb{R}^{3,2})$ defines
a \emph{linear circle complex}
$\mathbb{P}(\mathcal{L}\cap\{a\}^\perp)$,
that is, a 2-dimensional family of circles. Depending on the signature of the vector $a$, those families admit different geometric interpretations. To see those, we introduce the (possibly complex) light-like vector 
\begin{equation*}
\mathfrak{a}^\star:=\mathfrak{a} + \lambda \p, \ \text{ where } \lambda := \inner{\mathfrak{a}, \p} - \sqrt{\inner{\mathfrak{a}, \p}^2 + \inner{\mathfrak{a}, \mathfrak{a}} } 
\end{equation*}
and observe that $\frac{\inner{\mathfrak{s}, \mathfrak{a}^\star}}{\inner{\mathfrak{s}, \p}\inner{\mathfrak{a}^\star, \p}}\equiv const.$ for all circles~$s$ in the linear circle complex determined by~$a$. We obtain three types: 
\begin{itemize}
\item if $\inner{ \mathfrak{a}, \mathfrak{a}} = 0$, then the 2-dimensional family consists of all elements in $\Light$ that are in oriented contact with the circle represented by~$a^\star=a$;
\item if $\inner{\mathfrak{a}, \mathfrak{a}} > 0$, then the vector $a^\star \in \Light$ represents a circle. Due to (\ref{equ_intersection_angle}), all circles in the linear circle complex intersect this circle $a^\star$ at a constant angle; 
\item if $\inner{\mathfrak{a}, \mathfrak{a}} < 0$, there are two cases: either $a^\star$ defines a real vector in $\Light$, then the circles in the linear complex intersect $a^\star$ at a constant imaginary angle; or $a^\star$ is a circle with imaginary radius, where the geometric interpretation is obscure. The latter case includes for example linear circle complexes consisting of all circles with the same constant radius.
\end{itemize}
We call the (possibly complex) light-like vector $a^\star$ the \emph{directrix of the linear circle complex}.

\bigskip

Space forms in the light-cone model are recovered by the choice of \emph{space form vectors $\mathfrak{q} \in \mathbb{R}^{3,2}$} with $\inner{\p, \q}=0$. Geometric configurations composed of circles in~$\mathbb{R}^2$ can then be interpreted in a space form $\mathcal{Q}$ of constant sectional curvature~$\kappa_\mathcal{Q}=-\inner{\q, \q}$ as follows 
\begin{itemize}
\item $\inner{\q, \q}=0$: if $\q \in \spann{\q_0}$ represents infinity, then the configuration is directly interpreted as \textbf{Euclidean space}; else, if $\q$ represents a point, then we obtain an actual Euclidean space form after applying a M\"obius transformation that maps $\q$ to $\infty$. 
\\[-6pt]\item $\inner{\q, \q}>0$: configurations are interpreted in the Poincare disc or half-plane modeling \textbf{hyperbolic space}. The space form vector $\q$ determines the respective boundary of the model via~$\mathfrak{q}\pm\sqrt{\inner{\q,\q}}\p$;
\\[-6pt]\item $\inner{\q, \q}<0$: planar configurations in this case are related to \textbf{spherical geometries} modelled on $S^2 \subset \mathbb{R}^3$ via an appropriate stereographic projection. 
\end{itemize} 

\bigskip

For a given space form $\Q$, the \emph{space form reflections} are given by reflections in linear circle complexes $a \in \mathbb{P}(\mathbb{R}^{3,2}) \setminus \Light$ with $\inner{\mathfrak{a}, \mathfrak{p}}=\inner{\mathfrak{a}, \mathfrak{q}}=0$. Those reflections $\sigma_a$ are given explicitly via
\begin{equation}\label{equ_formula_inversion}
 \sigma_a:\mathbb{R}^{3,2} \rightarrow \mathbb{R}^{3,2}, \ \ 
 \mathfrak{r} \mapsto \sigma_a(\mathfrak{r}) :
  = \mathfrak{r}-\frac{2\inner{\mathfrak{r}, \mathfrak{a}}}
   {\inner{\mathfrak{a}, \mathfrak{a}}}\mathfrak{a}.
\end{equation}  
All \emph{space form motions (isometries) of $\Q$} are then obtained as compositions of any number of space form reflections.

\bigskip

The \emph{geodesic curvature $\kappa$ of a circle $s$} with respect to the space form $\mathcal{Q}$ is given by 
\begin{equation}\label{equ_formula_curvature}
\kappa=\frac{\inner{\mathfrak{s}, \mathfrak{q}}}{\inner{\mathfrak{s}, \mathfrak{p}}}.
\end{equation}
For a Euclidean space form~$\Q$ determined by~$\q_0$, this follows directly from the table: let $s$ be a circle with radius~$r$, then its geodesic curvature is given by~$\kappa(s)=\frac{1}{r}$.   

Proofs for general space forms can be found in~\cite{blaschke} and are obtained by direct computations. For example, consider the classical Poincar\'e half-plane determined by the space form vector $\q:=(0,1,0,0,0)$ and let~$s$ be a Euclidean circle with center~$(x,y)$ and radius~$r$ that lies completely in the upper-half plane. It is well-known, that the hyperbolic radius $r_\mathbb{H}$ of $s$ is $r_\mathbb{H}=\frac{1}{2}\ln \frac{y+r}{y-r}$ and therefore the geodesic curvature of $s$ becomes
\begin{equation*}
\kappa_\mathbb{H}=\frac{1}{\tanh r_\mathbb{H}}=\frac{y}{r}=\frac{\inner{\mathfrak{s}, \mathfrak{q}}}{\inner{\mathfrak{s}, \mathfrak{p}}}.
\end{equation*} 

\bigskip

\emph{Geodesics} in space forms, that is, curves with vanishing geodesic curvature, are specific lines and circles. By (\ref{equ_formula_curvature}), it follows that a circle $s \in \Light$ represents a geodesic in the space form~$\mathcal{Q}$ if and only if $\inner{\mathfrak{s}, \q}=0$. 

\subsubsection{Matrix model}
Here, we consider standard models for each space form as subsets of the complex matrix algebra~
$\matC$. We represent points in the Euclidean plane $\QE$ and in spherical space (modelled by the unit sphere $\QS$) by quaternions while we represent points in hyperbolic space (modelled by the unit hyperboloid $\QH$) by split-quaternions. 
Thus, we only consider spherical space with $\varepsilon:=\kappa_\Q=1$ and hyperbolic space with $\varepsilon:=\kappa_\Q=-1$. Other curvatures can be obtained by rescaling the metric. 
Both the quaternions and split-quaternions arise as real sub-algebras of $\matC$. For this, consider the basis of $\matC$ given by the Pauli matrices
\begin{align*}
\sigma_0=\begin{pmatrix}1 & 0 \\ 0 & 1 \end{pmatrix},\quad
\sigma_1=\begin{pmatrix}0 & 1 \\ 1 & 0 \end{pmatrix},\quad
\sigma_2=\begin{pmatrix}0 & -i \\ i & 0 \end{pmatrix},\quad
\sigma_3=\begin{pmatrix}1 & 0 \\ 0 & -1 \end{pmatrix}\in\matC.
\end{align*}

The real sub-algebra generated by $\bm 1=\sigma_0,\bm i=-i\sigma_1,\bm j=-i\sigma_2,\bm k=-i\sigma_3$ is isomorphic to the quaternions. 
We identify the Euclidean plane with a plane in the quaternions via 
\begin{align*}
\QE\ni(x,y)\leftrightarrow x\bm i+y\bm j=\begin{pmatrix}0&-y-i x\\y-i x& 0\end{pmatrix}\in\spannR{\bm i,\bm j}
\end{align*}
and the unit sphere with imaginary unit quaternions via
\begin{align*}
\QS\ni(x,y,z)\leftrightarrow x\bm i+y\bm j+z\bm k=\begin{pmatrix}-i z&-y-i x\\y-i x& iz\end{pmatrix}\in\{F\in\spannR{\bm i,\bm j,\bm k}\,|\,\det F=1\}.
\end{align*}
Now, the real sub-algebra generated by $\bm 1=\sigma_0,\sigma_1,\sigma_2,\bm k=-i\sigma_3$ is isomorphic to the split-quaternions and we identify the unit hyperboloid with imaginary unit split-quaternions via 
\begin{align*}
\QH\ni(x,y,z)\leftrightarrow x\sigma_1+y\sigma_2+z\bm k=\begin{pmatrix}-i z&x-iy\\x+i y& iz\end{pmatrix}\in\{F\in\spannR{\sigma_1,\sigma_2,\bm k}\,|\,\det F=1\}.
\end{align*}
To each matrix $F=\begin{pmatrix}a & b \\ c & d\end{pmatrix}\in\matC$ we denote the adjugate matrix by 
\begin{align*}
F^*=(\det F) F^{-1}=\begin{pmatrix}
d & -b \\ -c & a\end{pmatrix}.    
\end{align*}
The complex symmetric bilinearform $\innerM{F,G}:=\frac12\tr FG^*$ induces the usual Euclidean scalar product of signature $(4,0)$ on the quaternions while it induces a real inner product of signature $(2,2)$ on the split-quaternions. 
The metric in each space form arises from this scalar product:
\\\begin{center}
\begin{tabular}{p{2cm} | p{3cm} | p{3cm} | p{3.5cm}}
\textbf{Space form} & \textbf{Points $f,g$} & \textbf{Signature of $\innerM{\cdot,\cdot}$} & \textbf{distance of $f$ and $g$}
\\\hline Euclidean & $F,G\in\QE$ & $(2,0)$ & $\sqrt{\innerM{G-F,G-F}}$ 
\\\hline Spherical & $F,G\in\QS$ & $(3,0)$ & $\arccos\innerM{F,G}$ 
\\\hline Hyperbolic & $F,G\in\QH$ & $(1,2)$ & $\operatorname{arcosh}\innerM{F,G}$ 
\end{tabular}
\end{center}
\ \\Note that the hyperboloid $\QH$ has two sheets which can be distinguished by the sign of the $\bm k$-component of any $F\in\QH$. The distance between two points $F,G$ is only well-defined if both points lie on the same sheet and, hence, fulfill $\innerM{F,G}\geq1$. 

\bigskip

A geodesic in Euclidean space $\QE$ is the solution set of the equation $\innerM{N,F}=d$ for a unit imaginary quaternion $N$ and normal distance $d$. The geodesic through two distinct points $F$ and $G$ is given by 
\begin{align*}
    N=\frac{\bm k(G-F)}{\sqrt{\det(G-F)}}\qquad\text{and}\qquad d=\frac{\innerM{\bm k G,F}}{\sqrt{\det(G-F)}}. 
\end{align*}
Similarly, a geodesic in a non-Euclidean space form is the solution set of the equation $\innerM{N,F}=0$ for a trace-free normalized (split-)quaternion $N$. The geodesic through two distinct points $F,G$ is given by 
\begin{align*}
    N=\frac{\overrightarrow{FG}}{\sqrt{\varepsilon\det \overrightarrow{FG}}},
\end{align*}
where $\varepsilon:=\kappa_\Q=\pm1$ and where $\overrightarrow{M}=M-\frac12\tr M\bm 1$ denotes the trace-free component of a matrix. 
Note that in hyperbolic case we have $|\innerM{F,G}|>1$ and $\det\overrightarrow{FG}<0$ and, thus, we normalize to $\det N=-1$. 

\bigskip

Isometries in Euclidean space can be expressed by maps of the form
\begin{align*}
F\mapsto EFE^{-1}+E_T,
\end{align*}
where either $E,E_T\in \spannR{\bm i,\bm j}$ (orientation reversing) or $E\in \spann{\bm1,\bm k}$ and $E_T\in \spannR{\bm i,\bm j}$ (orientation preserving). 
Isometries in non-Euclidean space can be expressed by maps of the form
\begin{align*}
F\mapsto -EFE^{-1} \text{ (orientation reversing)},
\end{align*}
or
\begin{align*}
F\mapsto EFE^{-1} \text{ (orientation preserving)},
\end{align*}
where $E$ is a quaternion in the spherical case and an invertible split-quaternion in the hyperbolic case.

\subsubsection{Conversion of the models}
\label{sec:convModel}

The three standard models $\QE,\QS,\QH$(Poincar\'e disc) are naturally obtained from the light cone model by a choice of a specific space form vector $\q$:
\\\begin{center}
\begin{tabular}{p{2cm} | p{2cm}}
\textbf{Space form} & \textbf{$\q$}
\\\hline Euclidean & $(0,0,-1,1,0)$ 
\\\hline Spherical & $(0,0,0,1,0)$
\\\hline Hyperbolic & $(0,0,-1,0,0)$
\end{tabular}
\end{center}
\ \\Now, the correspondence between the representation of points in both models is given by the following table. Each vector in the light cone is given by homogeneous coordinates with $\inner{\f,\q}=-1$.
\\\begin{center}
\begin{tabular}{p{2cm} | p{3cm} | p{6.5cm}}
\textbf{Space form} & \textbf{Point as matrix} & \textbf{Point in $\Light$}
\\\hline Euclidean & $F=x\bm i+y\bm j$ & $f=\mathbb{R}(x,y,\frac12(1-x^2-y^2),\frac12(1+x^2+y^2),0)$
\\\hline Spherical & $F=x\bm i+y\bm j+z\bm k$ & $f=\mathbb{R}(x,y,z,1,0)$
\\\hline Hyperbolic & $F=x\sigma_1+y\sigma_2+z\bm k$ & $f=\mathbb{R}(x,y,1,z,0)$
\end{tabular}
\end{center}
\ \\This correspondence relates the inner products of two vectors in the respective models. In the Euclidean case we have
\begin{align*}
   \inner{\f,\mathfrak{g}}=-\frac{1}{2}\|F-G\|^2=-\frac{1}{2}\innerM{G-F,G-F} 
\end{align*}
and in the non-Euclidean cases we have (with $\varepsilon=\kappa_\Q=\pm1$)
\begin{align*}
   \inner{\f,\mathfrak{g}}=\varepsilon(\innerM{F,G}-1)=-\varepsilon\frac{1}{2}\innerM{G-F,G-F}  .
\end{align*}

For geodesics the correspondence is given by
\\\begin{center}
\begin{tabular}{p{3cm} | p{5.5cm} | p{5.5cm}}
\textbf{Space form} & \textbf{Geodesic in the matrix model} & \textbf{Geodesic in $\Light$}
\\\hline Euclidean & $N=a\bm i+b\bm j,\quad d\in\mathbb{R}$ & $\mathfrak t=(a,b,-d,d,1)$
\\\hline Spherical & $N=a\bm i+b\bm j+c\bm k$ & $\mathfrak t=(a,b,c,0,1)$
\\\hline Hyperbolic & $N=a \sigma_1+b \sigma_2+c\bm k$ & $\mathfrak t=(a,b,0,c,1)$
\end{tabular}
\end{center}
\ \\The choice of sign in the last component is a convention. Changing the sign corresponds to changing the orientation of the geodesic.

Our conversion for the inner product between a point and a geodesic reads $\inner{\f,\frakt}=\innerM{N,F}-d$ in Euclidean space and $\inner{\f,\frakt}=\varepsilon\innerM{N,F}$ in non-Euclidean space.

%%%%%%%%%%%%%%%%%%%%%%%%%%%%%%%%%%%%%%%%%%%%%%%%%%%%%%%%
%
%
%%%%%%%%%%%%%%%%%%%%%%%%%%%%%%%%%%%%%%%%%%%%%%%%%%%%%%%%
\subsection{Basic differential geometric notions} 
Let $\G=(\V,\E)$ be a 1-dimensional graph with vertices~$\V$ and directed combinatorical edges~$\E$. A \emph{discrete curve} is a map $f:\V\to\Q$ that assigns to each vertex of the graph a point in some space form $\Q$. As illustrated in Figure~\ref{fig:notation}, we will use the shift notation $f_{\bar{1}}, f_0, f_1$ for three consecutive curve points 
that are then connected by the two combinatorial edges $e_{\bar{1}0}$ and $e_{01}$.

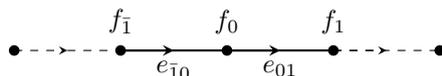
\begin{figure}
\centering
\begin{tikzpicture}[scale=1.4]
\draw [thick]   (-1,0) -- (1,0);
\draw [thick,-stealth]   (0,0) -- (.5,0);
\draw [thick,-stealth]   (-1,0) -- (-0.5,0);
\draw [dashed]  (-1.5,0) -- (-1,0);
\draw [dashed,-stealth]  (-2,0) -- (-1.5,0);
\draw [dashed]  (1,0) -- (2,0);
\draw [dashed,-stealth]  (1,0) -- (1.5,0);
      \fill[black] (-2,0) circle (0.5mm) node[above=1mm] {}  
                   (-1,0) circle (0.5mm) node[above=1mm] {$f_{\bar{1}}$} 
                   (0,0) circle (0.5mm) node[above=1mm] {$f_0$}
                   (1,0) circle (0.5mm) node[above=1mm] {$f_{1}$}
                   (2,0) circle (0.5mm) node[above=1mm] {}
                   (0.5,0) circle(0.0mm) node[below=0.2mm] {$e_{01}$}
                   (-0.5,0) circle(0.0mm) node[below=0.2mm] {$e_{\bar{1}0}$};
\end{tikzpicture}
\caption{Notation convention used for discrete curves.}
    \label{fig:notation}
\end{figure}

When we need a distinguish initial vertex, we specify one vertex $t_0\in \V$ and write $f(t_0)$ for the initial point of the curve.

\begin{defi}
A discrete curve $f:\V \to \Q$ is said to be parametrized by \emph{constant arc-length}, if the distance (with respect to the metric of $\Q$) between any two consecutive points of $f$ is constant.
\end{defi}

Throughout the paper, we consider \emph{regularly} arc-length parametrized discrete curves in a distinguished space form~$\Q$: that is, any three consecutive curve points are pairwise distinct and, additionally,
\begin{itemize}
\item if $\Q$ is a hyperbolic space form, all curve points lie inside or outside the hyperbolic boundary;
\item if $\Q$ is a spherical space form, then two consecutive points are not antipodal.
\end{itemize}

\begin{defi}
Suppose that $f: \V \to Q$ is a discrete regular curve in the space form~$Q$. The unique distance-minimizing geodesic~$g_{01}$ in~$Q$ that connects two consecutive curve points $f_0$ and $f_1$ is called the \emph{geodesic edge} between $f_0$ and $f_1$. 

Moreover, we say that an orientation for the geodesic edges is \emph{consistent} if, for any two neighbouring oriented geodesic edges $g_{\bar{1}0}$ and $g_{01}$, there exists a circle in oriented contact with them at the curve points $f_{\bar{1}}$ and $f_{1}$.
\end{defi}

Note that a choice of orientation for one geodesic edge uniquely determines a global orientation for the discrete curve. As we will see below, those oriented geodesic edges encode first order information about the geometry of the curve (see Figure~\ref{fig:double_circles}).

Thus, in what follows, we will consider discrete curves together with their consistently oriented geodesic edges.

\begin{defi}
The space of all regular and consistently oriented discrete curves that have constant arc-length in the space form~$\Q$ is denoted by $\mathcal{C}_\Q$. 
\end{defi}

In some models it is more convenient to work with the complete geodesics rather than with the geodesic edges. Thus, in analogy with the smooth case, we define a tangent line congruence for discrete curves:

\begin{defi}
For a discrete curve~$f \in \mathcal{C}_\Q$, the map 
\begin{equation*}
t: \E \to E_Q:= \{ \text{complete oriented geodesics in Q} \}
\end{equation*}
is called the \emph{tangent line congruence of $f$} if it extends the oriented geodesic edges.
\end{defi}

%%%%%%%%%%%%%%%%%%%%%%%%%%%%%%%%%%%%%%%%%%%%%
\begin{figure}
    \begin{minipage}{4.5cm}
    \hspace*{-0.6cm}\includegraphics[scale=0.6]{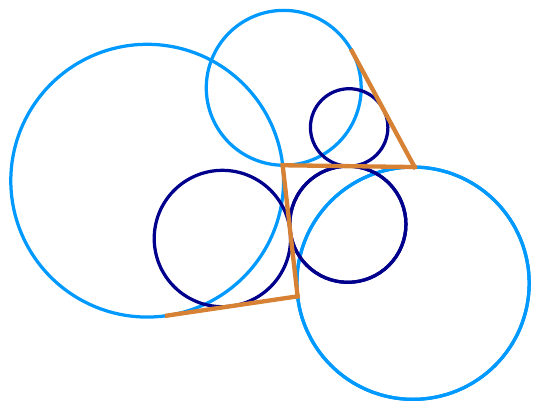}
    \end{minipage}
    \begin{minipage}{4.5cm}
    \hspace*{0.6cm}\includegraphics[scale=0.58]{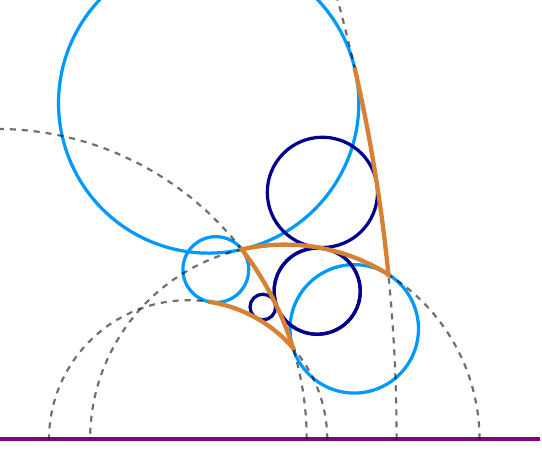}
    \end{minipage}
    \begin{minipage}{4.5cm}
    \hspace*{1.55cm}\includegraphics[scale=0.3]{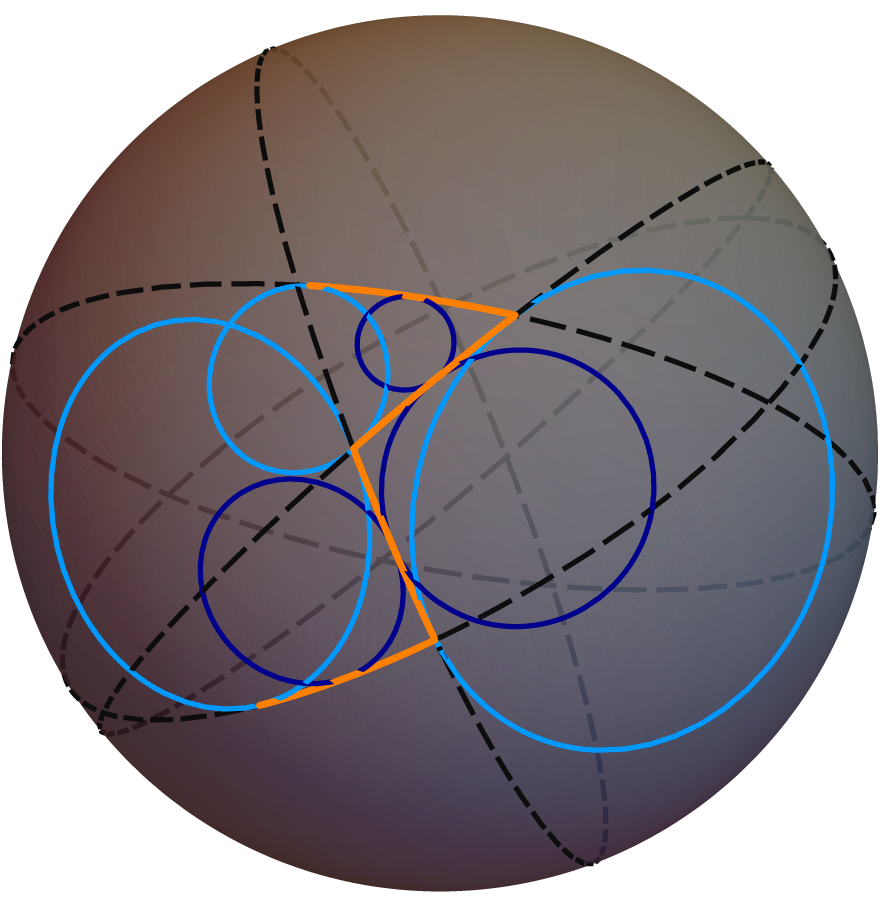}
    \end{minipage}
    \caption{Arc-length parametrized discrete curves in Euclidean, hyperbolic and elliptic space with their geodesic edges (orange), curvature circles (dark blue) and double-curvature circles (light blue).}
    \label{fig:double_circles}
\end{figure}
%%%%%%%%%%%%%%%%%%%%%%%%%%%%%%%%%%%%%%%%%%%%%
%%%%%%%%%%%%%%%%%%%%%%%%%%%%%%%%%%%%%%%%%%%%%
We are now prepared to introduce a notion of curvature for planar discrete curves in space forms. Our definition naturally generalizes the notion of curvature circle as discussed in~\cite{hoffmann_lecturenotes} for discrete arc-length parametrized curves in Euclidean space: suppose that $f \in \mathcal{C}_{\mathbb{E}^2}$ has constant Euclidean arc-length given by $|f_0-f_1|^2 \equiv \EucDelta^2 \in \mathbb{R}^\times$. Then the Euclidean curvature at the curve point~$f_0$ is defined as $\kappa_0:= \frac{1}{r_0}$, where $r_0$ is the radius of the circle $c_0$ tangent to the oriented edges $g_{\bar{1}0}$ and $g_{01}$ at the midpoints of these two edges (see Figure~\ref{fig:double_circles} Left). 

We observe that the Euclidean curvature of the circle $c_0$ is exactly twice the Euclidean curvature of the circle tangent to the oriented edges $g_{\bar{1}0}$ and $g_{01}$ at the curve points $f_{\bar{1}}$ and $f_1$. As illustrated in Figure~\ref{fig:double_circles}, this fact inspires our curvature notion for discrete curves in space forms:

\begin{defi}
Suppose that $f \in \mathcal{C}_\Q$ is a discrete curve in the space form $\Q$. Its \emph{geodesic curvature} $\kappa_0$ at the curve point $f_0$ is defined as two times the geodesic curvature of the \emph{double-curvature circle} which is the circle tangent to the geodesic edges $g_{\bar{1}0}$ and $g_{01}$ at the curve points $f_{\bar{1}}$ and $f_1$.
\end{defi}

In what follows, we point out how those basic geometric objects for discrete curves in space forms are represented in the two models under consideration.

\subsubsection{Light-cone model} To describe this setup in the light-cone model, we identify the curve points of a discrete curve~$f \in \C$ with vectors in~$\Light$ and fix a \emph{space form lift}~$\f \in f$ such that $\inner{\f_0,\q}=-1$. Furthermore, for the tangent line congruence, we choose the canonical lift $\tang \in t$ with $\inner{\tang_{01}, \p}=-1$.

Since $f \in \C$ has constant arc-length, along the curve we obtain canonical space form reflections~$\sigma_{a_0}$ with respect to the circle complexes
\begin{equation*}
    \mathfrak{a}_0:=\tang_{\bar{1}0} - \tang_{01}.
\end{equation*}
The space form motion~$\sigma_{a_0}$ has $\f_0$ as fixed point and interchanges the curve points~$\f_{\bar{1}}$ and~$\f_1$.

Therefore, we conclude that the quantity~$\inner{\f_0, \f_1}=\inner{\f_{\bar{1}}, \f_0}$ is constant along any discrete arc-length parametrized curve. 

\begin{defi}
The subspace of $\mathcal{C}_\Q$ that contains all curves with conserved quantity $\eta^2:=-2\inner{\f_0,\f_1}\equiv const$ will be denoted by~$\mathcal{C}_\Q(\eta)$.
\end{defi}

Finally, we aim to give a simple formula for the curvature of $f \in \C$ in this model. For the special choices of lifts from above, the circle~$c_0$ that is tangent to the geodesics $t_{\bar{01}}$ and $t_{01}$ and passes through the curve points $\fminus$ and $\fplus$ is described by
\begin{equation}\label{eq:curvature_circle}
\mathfrak{c}_0:= \mathfrak{t}_{\bar{1}0} - \frac{\inner{\mathfrak{t}_{\bar{1}0}, \ffplus}}{\inner{\ffminus, \ffplus}}\ffminus = \mathfrak{t}_{01} - \frac{\inner{\mathfrak{t}_{01}, \ffminus}}{\inner{\ffplus, \ffminus}}\ffplus.
\end{equation}
Therefore, since the geodesic curvature $\kappa_0$ in $\Q$ at the curve point $f_0$ is defined as two times the geodesic curvature of $c_0$, it follows from (\ref{equ_formula_curvature}) that
\begin{equation}\label{eq:curvature_fraction}
\kappa_0 = \frac{\inner{\mathfrak{c}_0, \mathfrak{q}}}{\inner{\mathfrak{c}_0, \mathfrak{p}}} = -\frac{2\inner{\mathfrak{t}_{\bar{1}0}, \ffplus}}{\inner{\ffminus, \ffplus}}=-\frac{2\inner{\mathfrak{t}_{01}, \ffminus}}{\inner{\ffplus, \ffminus}}.
\end{equation}

\subsubsection{Matrix model}

Here, we consider curves in $\mathcal{C}_\Q(\eta)$ for the standard models $\Q=\QE,\QS,\QH$ and identify every curve point $f$ with a complex matrix $F$ as explained above. 

The first order information is encoded in an edge based \emph{transport matrix} $u:\E\to\matC$ which we have to define separately for Euclidean and non-Euclidean space forms: 
In the Euclidean plane we define
\begin{align*}
u_{01}:=F_1-F_0.
\end{align*}
Since the curve has constant arc-length we have that $\sqrt{\det u_{01}}=\eta>0$ and $\tr u_{01}=0$ are constant. 
The geodesic connecting neighboring points is given by the equation $\innerM{N_{01},G}=d_{01}$ for $N_{01}=\frac{1}{\eta}\bm k u_{01}$ and a corresponding normal distance $d_{01}$. 

For a spherical or hyperbolic space form we define
\begin{align*}
u_{01}:=F_1F_0^{-1}=-F_1F_0.
\end{align*}
Since the curve has constant arc-length we have that 
\begin{align*}
    \frac12\tr u_{01}=\innerM{ F_0,F_1}=1+\varepsilon\inner{\f_0,\f_1}=1-\varepsilon\frac{\eta^2}{2}\begin{cases}\in (-1,1)&\text{for }\Q=\QS\\\in (1,\infty)&\text{for }\Q=\QH\end{cases}
\end{align*}
is constant. We introduce the conserved quantity
\begin{align*}
    \dDelta:=\sqrt{\varepsilon\det \vec u_{01}}=\sqrt{\varepsilon(1-\innerM{ F_0,F_1}^2)}=\eta\sqrt{1-\varepsilon\frac{\eta^2}{4}}
\end{align*}
which is the $\sin$ (or $\sinh$) of the distance between consecutive points. 
Then, the geodesic connecting neighboring points is given by the equation $\innerM{N_{01},G}=0$ where $N_{01}=\frac{1}{\dDelta}\vec u_{01}$. 

Both quantities $\eta$ and $\dDelta$ encode the arc-length. In the Euclidean case we set $\dDelta=\eta$ and their relation $\dDelta^2=\eta^2(1-\varepsilon\frac{\eta^2}{4})$ still holds (with $\varepsilon=\kappa_\Q=0$). 

Note that in this model the curve direction induces a consistent orientation on the geodesic edges. Thus, the tangent line congruence is determined by $N_{01}=\frac{1}{\dDelta}\vec u_{01}$ (and $d_{01}$ in the Euclidean case). This choice of orientation translates to the light cone model by a consistent choice of sign as pointed out in Section~\ref{sec:convModel}. 

As motivated in \cite{hoffmann_lecturenotes} for the Euclidean case, we will introduce the harmonic mean
\begin{align}
    T_0:=2(\vec{u}_{\bar10}^{-1}+\vec{u}_{01}^{-1})^{-1}
\end{align}
as a vertex tangent vector $T:\V\to\matC$. A quick computation shows that
\begin{align*}
    \vec u_{01}^{-1}T_0=2(\vec u_{\bar10}^{-1}\vec u_{01}+1)^{-1}=2(-\frac{1}{\dDelta^2}\vec u_{\bar10}\vec u_{01}+1)^{-1}=2(\vec u_{\bar10}\vec u_{01}^{-1}+1)^{-1}=T_0\vec u_{\bar10}^{-1}.
\end{align*}
Thus, we can define a second vertex based map $H:\V\to\matC$ by
\begin{align}
\label{eq:HDef}
H_0:=\vec u_{01}^{-1}T_0=T_0\vec u_{\bar10}^{-1}.%2\vec{u}_{01}(\vec{u}_{\bar10}+\vec{u}_{01})^{-1}.
\end{align}
%which encodes second order information. 
Both $T$ and $H$ transport $u$ by
\begin{align}
    u_{01}=T_0u_{\bar10}T_0^{-1}=H_0u_{\bar10}H_0^{-1}.
\end{align}
The vector $H_0$ encodes the curvature  $\kappa_0$ at each curve point $F_0$:
\begin{prop}
\label{prop:Hcurv}
In Euclidean space we have $H_0=\bm 1+\frac{\EucDelta}{2}\kappa_0 \bm k$. In non-Euclidean space we have $H_0=\bm 1+ \frac{\dDelta}{2}\kappa_0 F_0$. 
\end{prop}

\begin{proof}
We obtain the curvature using the light cone model: From \eqref{eq:curvature_circle} and \eqref{eq:curvature_fraction} we obtain the relation $\mathfrak{t}_{\bar{1}0}-\mathfrak{t}_{01} = \frac{\kappa_0}{2} (\ffplus-\ffminus)$. 
This translates to $N_{\bar{1}0}-N_{01} = \frac{\kappa_0}{2} (F_1-F_{\bar1})$ in the matrix model by applying the conversion tables in Section \ref{sec:convModel}. Here, again, $N$ denotes the normalized matrix describing the geodesic edge.

In Euclidean geometry we have $N_{01}=\frac{\bm k u_{01}}{\EucDelta}$ and, therefore,
\begin{align*}
\bm 1+\EucDelta\bm k\frac{\kappa_0}{2}&=\bm 1+\EucDelta\bm k(N_{\bar10}-N_{01})(F_1-F_{\bar1})^{-1}=\bm 1-(u_{\bar10}-u_{01})(u_{\bar10}+u_{01})^{-1}\\
&=2u_{01}(u_{\bar10}+u_{01})^{-1}=H_0.
\end{align*}

In non-Euclidean geometry we have $N_{01}=\frac{\vec u_{01}}{\dDelta}$ and, therefore,
\begin{align*}
\bm 1+\dDelta \frac{\kappa_0}{2}F_0 &=\bm 1+\dDelta(N_{\bar10}-N_{01})(F_1-F_{\bar1})^{-1}F_0=\bm 1+(u_{\bar10}-u_{01})(F_0^*F_1-F_0^*F_{\bar1})^{-1}\\
&=( u_{01}^*-u_{\bar10}+u_{\bar10}-u_{01})(u_{01}^*-u_{\bar10})^{-1}=-2\vec u_{01}(-\vec u_{01}-\vec u_{\bar10})^{-1}=H_0.
\end{align*}
\end{proof}

%%%%%%%%%%%%%%%%%%%%%%%%%%%%%%%%%%%%%%%%%%%%%%%%%%%
%
%
%%%%%%%%%%%%%%%%%%%%%%%%%%%%%%%%%%%%%%%%%%%%%%%%%%%
\subsection{Constructions for discrete curves in space forms}
In this subsection we develop discrete counterparts to some differential geometric concepts such as a fundamental theorem and an associated family for arc-length parametrized curves in 2-dimensional space forms. Discrete Euclidean versions of these results can be found in various surveys on discrete curve theory, for example, in \cite{hoffmann_lecturenotes, Inoguchi_2014, sauer_differenzengeo}. 
\subsubsection{Frenet-type equations for discrete curves in 2-dimensional space forms}
%%%%%%%%%%%%%%%%%%%%%%%%%%%%%%%%%%%%%%
Additionally to the tangent line congruence,  we associate edge-based normals to discrete arc-length parametrized curves in space forms. In the light-cone model, we define them as follows: 
\begin{defi}
   For a discrete curve~$f \in \mathcal{C}_\Q(\eta)$, the \emph{edge-normals}~$\nor_{01}$ are defined by $\nor_{01}:=\frac{1}{\eta}(\f_0 - \f_1)+ \p$, where $\f \in f$ denotes the space form lift with $\inner{\f_0, \q}=-1$. 
\end{defi}
Geometrically, the edge-normal~$\nor_{01}$ is the geodesic that intersects the geodesic edge~$t_{01}$ orthogonally in the midpoint of the two curve points~$f_0$ and $f_1$. The latter follows from the fact that the space form isometry $\sigma_{\f_0 - \f_1}$ reverses the orientation of $\nor_{01}$ and has $\tang_{01}$ as fixed point. 

\bigskip

This notion of edge-normals provides discrete 
Frenet-type formulas in space forms that generalize the Frenet-Serret formulas known for discrete curves in Euclidean space, see~\cite{Inoguchi_2014, sauer_differenzengeo, fairing_elastica}. 
\begin{prop}
Any discrete curve $f \in \mathcal{C}_\mathcal{Q}(\eta)$ satisfies the following two Frenet-type equations:
\begin{align}
\label{eq:frenet_eq_1} \frac{1}{\eta}\{ \tang_{\bar{1}0} - \tang_{01} \} &= -\frac{\kappa_0 }{2} \{ \nor_{\bar{1}0} + \nor_{01} - 2\p \} = -\frac{\kappa_0 }{2 \eta} \{ \f_{\bar{1}} - \f_1 \},
\\ \frac{1}{\eta}\{\nor_{\bar{1}0} - \nor_{01} - \eta ( \inner{\q,\q}\f_0  + \q) \} &=  \frac{\kappa_0 }{2}(1+\frac{\inner{\q,\q}}{4}\eta^2) \{ \tang_{\bar{1}0} + \tang_{01} - 2\p \}\label{eq:frenet_eq_2}
\end{align}
\end{prop}

\begin{proof}
Formula~\eqref{eq:frenet_eq_1} follows directly from the description of the double-curvature circle congruence
\begin{equation}\label{eq:double_circles}
\mathfrak{c}_0=\frac{\kappa_0}{2}\f_{\bar{1}} + \tang_{\bar{1}0}=\frac{\kappa_0}{2}\f_1 + \tang_{01}.
\end{equation}
To prove the second Frenet-type formula~\eqref{eq:frenet_eq_2}, we observe that the vector fields
\begin{align*}
\mathfrak{v}_0:=\frac{1}{\eta} \{ \nor_{01} - \nor_{\bar{1}0} + \eta(\f_0 \inner{\q,\q} + \q) \}
\\\tilde{\mathfrak{v}}_0:=- \frac{\kappa_0}{2}(1+\frac{\inner{\q,\q}}{4}\eta^2) \{ \tang_{\bar{1}0} + \tang_{01} - 2\p \}
\end{align*}
are both orthogonal to $\f_{\bar{1}}, \f_{\bar{1}} - \f_1, \p$ and $\q$ and therefore parallel. Moreover, by taking inner product with the circle congruence~$\mathfrak{c}_0$, we conclude that $\mathfrak{v}_0$ and $\tilde{\mathfrak{v}}_0$ indeed coincide. 
\end{proof}

\begin{rem}
    We can apply our conversion of the models to obtain Frenet-type equations in the matrix model as well. With $\dDelta=\eta\sqrt{1-\varepsilon\frac{\eta^2}{4}}$ they become

\begin{align*}
    \frac{1}{\eta\zeta}(\vec u_{\bar10}-\vec u_{01})&=-\frac{\kappa_0}{2}(\hat N_{\bar10}+\hat N_{01})\\
    \frac{\eta}{\zeta}(\hat N_{\bar10}-\hat N_{01}+\varepsilon\eta F_0)&=\frac{\kappa_0}{2}(\vec u_{\bar10}+\vec u_{01})
\end{align*}
where $\hat N_{01}=\frac{1}{\EucDelta}\bm k u_{01}$ in the Euclidean case and $\hat N_{01}=\frac{1}{\eta}(F_{0}-F_1)$ in the non-Euclidean case.
\end{rem}

%
%%%%%%%%%%%%%%%%%%%%%%%%%%%%%%%%%%%%%%
\subsubsection{Associated family of curves}
\label{sec:assoFamily}
As in the smooth case, a discrete curve with constant arc-length in a space form is, up to space form motions, uniquely described by its curvature function and we can formulate a fundamental theorem.  

\begin{thm}
To every space form $Q$, an admissible arc-length given by~$\eta$ and each map~$\kappa: \V \to  \mathbb{R}$ there exists a curve~$f\in\C (\eta)$ that has~$\kappa$ as curvature function. This discrete curve~$f$ is unique up to motions in $Q$. 
\end{thm}

\begin{figure}[h!]
  \centering
  \includegraphics[width=0.15\textwidth]{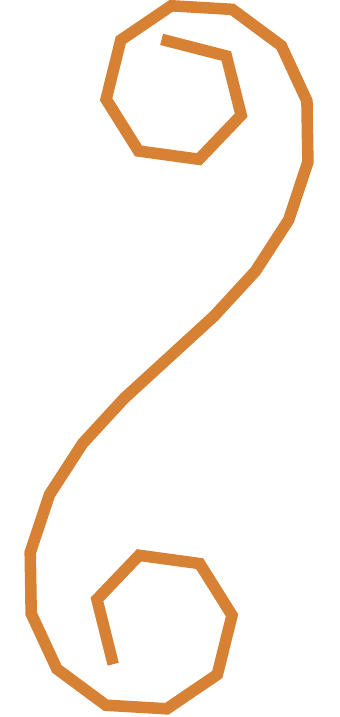}
  \hspace{20mm}
  \includegraphics[width=0.2\textwidth]{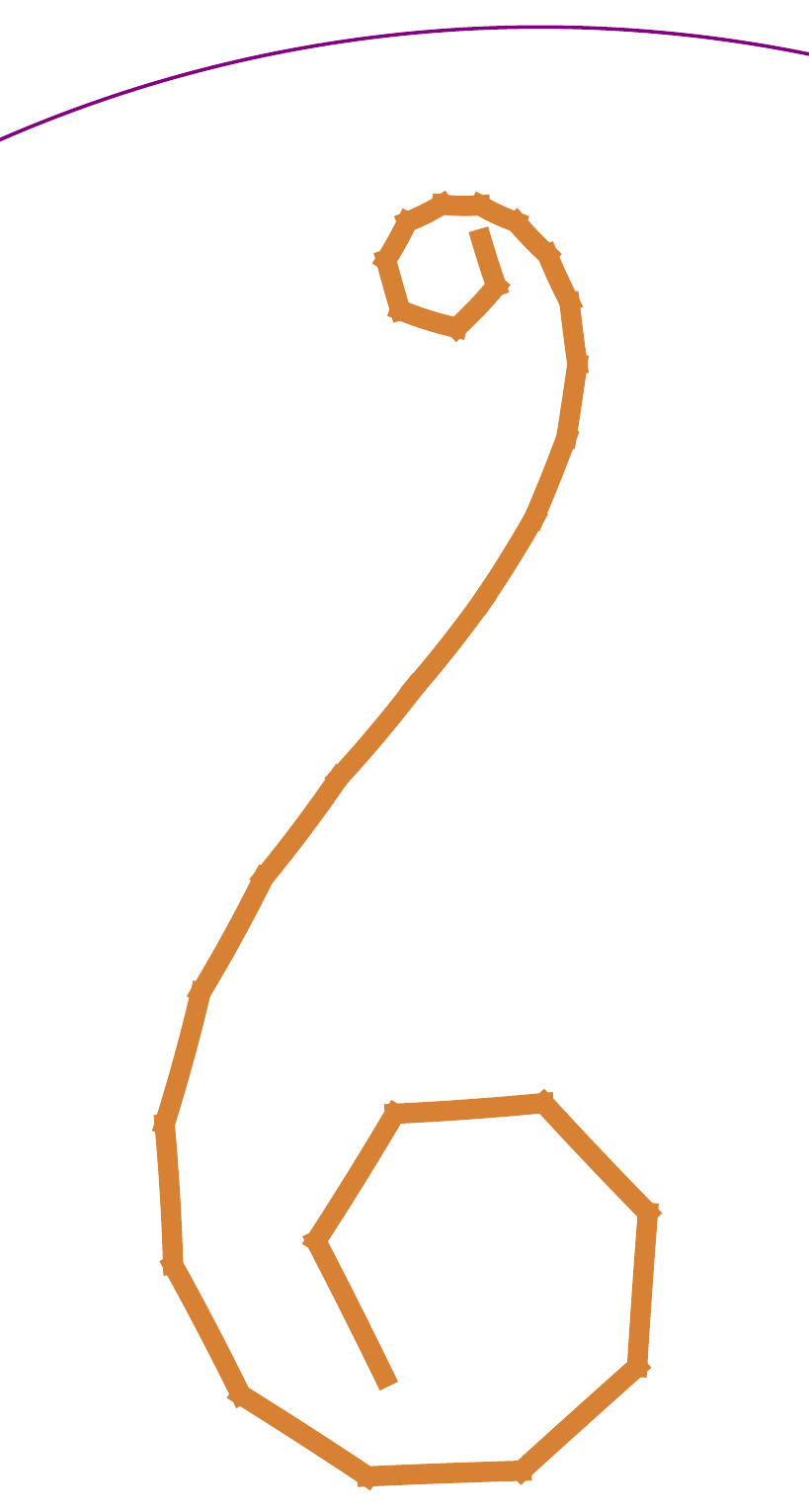}
  \hspace{20mm}
  \includegraphics[width=0.28\textwidth]{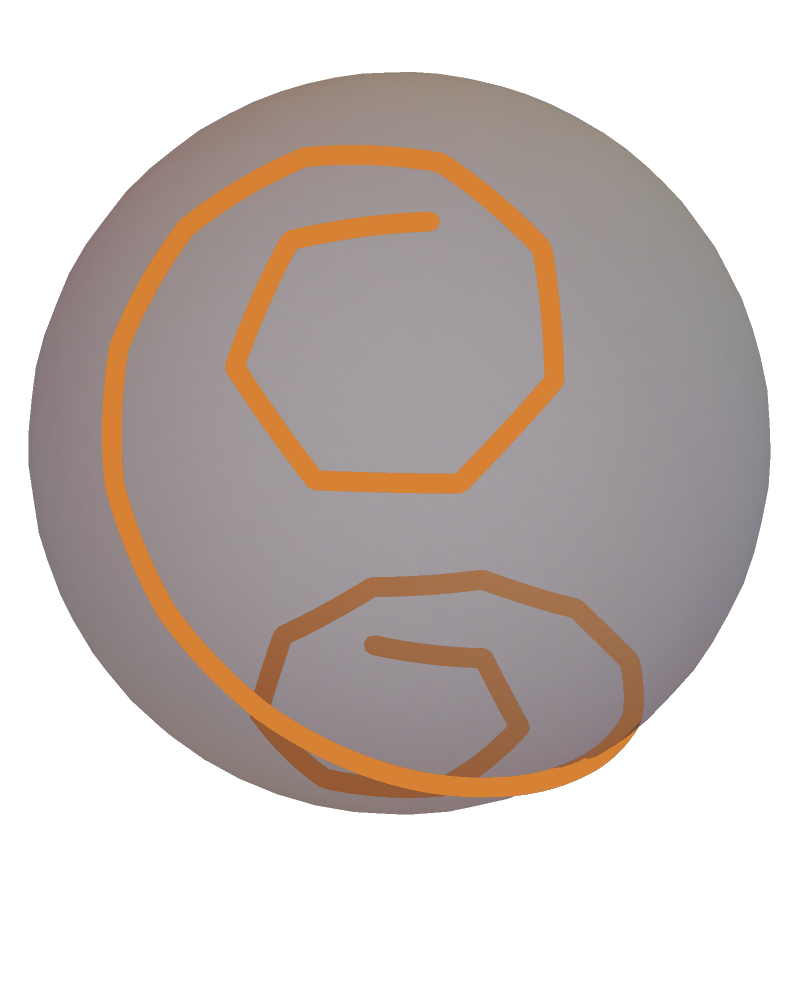}
  \caption{Curves obtained from $\kappa(t_i)=a i$ for some constant $a\in\mathbb{R}$ form discrete Euler spirals/clothoids in Euclidean space $\QE$, hyperbolic space $\QH$ (the Poincar\'e disc) and spherical space~$\QS$.}
  \label{fig:clothoid}
\end{figure}

Geometrically, it is clear that the curve can be iteratively constructed. We exemplify this by explaining the construction of the curve in the matrix model. 
In Euclidean space $H_0=\bm 1+\frac{\eta}{2}\kappa_0\bm k$ is already determined from the curvature and after prescribing two initial points with distance $\eta$ we construct the curve by iterating $u_{01}=H_0u_{\bar10}H_0^{-1}$ and $F_1=F_0+u_{01}$. 
Similarly, in spherical and hyperbolic space we also obtain the curve from two initial points with distance given by $\eta$ by iterating $H_0=\bm 1+\frac{\dDelta}{2}\kappa_0 F_0$, $u_{01}=H_0u_{\bar10}H_0^{-1}$ and $F_1=u_{01}F_0$.

\bigskip

To a given curve we will now compute associated curves with the same curvature function (up to global scaling) in other space forms. This is an application of the associated family presented in \cite{skewParallelogramNets} and, therefore, closely related to other associated families of curves and surfaces \cite{discretizationOfSurfacesIntegrable}. 
The associated family provides another way of integrating the curvature function to a non-Euclidean space form by first integrating it into Euclidean space and then applying the associated family. 
We will also show that applying the associated family is reversible. Thus, we can go back and forth between non-Euclidean and Euclidean space. In Section \ref{sec:flows}, we will apply this to generalize some Euclidean statements to space forms. 

Here, we consider the curves $f\in\C$ for $\Q=\QE,\QS,\QH$ in the matrix model and use that all space forms are contained in the common complex algebra $\matC$. 
We will also fix the initial point to $F(t_0)=\bm0$ for curves in Euclidean space and to $F(t_0)=\bm k$ for curves in non-Euclidean space. This is always achievable by applying an isometry. Now, the curve $F$ is completely determined by its transport matrix $u$. 

To a curve $F\in\C$ and each $\lambda\in\mathbb{C}$ (such that $\bm 1\pm\lambda u_{01}$ is invertible) we assign a map $\Phi^\lambda:\V\to\matC$ defined by $\Phi^\lambda(t_0)=\bm 1$ and $\Phi_1^\lambda=(\bm 1+\lambda u_{01})\Phi^\lambda_0$. 
Also, we define an edge based map $u^\lambda:\E\to\matC$ by
\begin{align}
\label{eq:assoEdge}
u^\lambda_{01}:=(\Phi_1^\lambda)^{-1}(\bm 1-\lambda u_{01})\Phi^\lambda_0.
\end{align}

\begin{propdefi}
\label{associatedFamily}
For $f\in\CE(\EucDelta)$ the transport matrix $u^\lambda$ 
defines a curve
\begin{enumerate}
\item in $\CS$ for each $\lambda\in\mathbb{R}\setminus\{0\}$,
\item in $\CH$ for each $\lambda\in i\mathbb{R}$ with $|\lambda|\in(0,\frac{1}{\eta})$
\end{enumerate}
which we denote by $\asso^\lambda f$. Its points can be computed by $\asso^\lambda F=(\Phi^\lambda)^{-1}\bm k\Phi^\lambda$. 
Similarly, the transport matrix
\begin{enumerate}
\item $u^1$ for a curve $f\in \CS$, or
\item $iu^1$ for a curve $f\in \CH$
\end{enumerate}
defines a curve in $\CE$ which we denote by $\asso^1 f$. 
We call the set of curves obtained by these operations from a curve $f\in \C$ the \emph{associated family} of $f$.
\end{propdefi}

\begin{proof}
First, let us state the general equation
\begin{align}
\label{eq:assoEdgeSandwich}
u^\lambda_{01}=(\Phi_0^\lambda)^{-1}(\bm 1+\lambda u_{01})^{-1}(\bm 1-\lambda u_{01})\Phi_0^\lambda=(\Phi_1^\lambda)^{-1}(\bm 1+\lambda u_{01})^{-1}(\bm 1-\lambda u_{01})\Phi_1^\lambda.
\end{align}

Now, let $f\in\CE(\EucDelta)$ be a Euclidean curve. First, we prove the equation $\asso^\lambda F=(\Phi^\lambda)^{-1}\bm k\Phi^\lambda$ iteratively: If it holds at some point $F_0$ then it also holds at the next point $F_1$ since
\begin{align*}
\asso^\lambda F_1&=\asso^\lambda u_{01}\asso^\lambda F_0=(\Phi_1^\lambda)^{-1}(\bm 1-\lambda u_{01})\Phi^\lambda_0(\Phi_0^\lambda)^{-1}\bm k\Phi^\lambda_0=(\Phi_1^\lambda)^{-1}\bm k(\bm 1+\lambda u_{01})\Phi^\lambda_0=(\Phi_1^\lambda)^{-1}\bm k\Phi^\lambda_1.
\end{align*}
Analogously, one can iterate backwards. The equation holds at $t_0$ and, hence, it holds everywhere. 
For $\lambda\in\mathbb{R}$ each $\Phi^\lambda$ is a quaternion and, therefore, $\asso^\lambda f\in\QS$ is a spherical curve. Similarly, for $\lambda\in i\mathbb{R}$ each $\Phi^\lambda$ is a split-quaternion and $\asso^\lambda f\in\QH$ is a hyperbolic curve. 
From \eqref{eq:assoEdgeSandwich} and
\begin{align}
\label{eq:uLamForm1}
(\bm 1+\lambda u_{01})^{-1}(\bm 1-\lambda u_{01})=\frac{(\bm 1-\lambda u_{01})^2}{1+\lambda^2\EucDelta^2}=\frac{(1-\lambda^2\EucDelta^2)\bm1-2\lambda u_{01}}{1+\lambda^2\EucDelta^2}
\end{align}
we conclude that the arc-length of $\asso^\lambda f$ is given (with $\varepsilon=\pm1$ depending on the space form of $\asso^\lambda f$) by
\begin{align}
\label{eq:TDelta1}
%\asso^\lambda r??=\frac{1-\lambda^2\EucDelta^2}{1+\lambda^2\EucDelta^2},\qquad 
\asso^\lambda \zeta=\sqrt{\varepsilon\det \vec u^\lambda _{01}}=\frac{2|\lambda|\EucDelta}{1+\lambda^2\EucDelta^2}\qquad
\text{and}\qquad\asso^\lambda\eta^2=2\varepsilon(1-\frac12\tr u^\lambda_{01})=\frac{4\varepsilon\lambda^2\eta^2}{1+\lambda^2\eta^2}
\end{align}
which is constant along the curve. 
Regularity of $\asso^\lambda f$ follows from regularity of $f$ and the stated restrictions on $\lambda$. For instance, for imaginary $\lambda$ the condition $|\lambda|\in(0,\frac{1}{\eta})$ ensures that neighboring points are distinct and lie on the same sheet of the hyperboloid. 
We conclude $\asso^\lambda f\in\C$ for $\Q=\QS$ and $\Q=\QH$ respectively.

Now, let $f\in\C(\eta)$ for $\Q=\QS$ or $\Q=\QH$. We will show that the transport matrix is a quaternion in $\spannR{\bm i,\bm j}$. We have $(\bm 1\pm u_{01})F_0=F_0\pm F_1=\pm F_1(\bm 1\pm u_{01})$ and, therefore, $\Phi^1F(t_0)=F\Phi^1$. With $F(t_0)=\bm k$ we have
\begin{align*}
u^1_{01}\bm k=(\Phi_1^1)^{-1}(\bm 1- u_{01})F_0\Phi^1_0=-(\Phi_1^1)^{-1}F_1(\bm 1- u_{01})\Phi^1_0=-\bm k u^1_{01}.
\end{align*}
If $\Q=\QS$ then $u^1_{01}$ is a quaternion anti-commuting with $\bm k$ and, thus, we have $u^1_{01}\in\spannR{\bm i,\bm j}$. If $\Q=\QH$ then $u^1_{01}$ is a split-quaternion anti-commuting with $\bm k$ and, thus, we have $iu^1_{01}\in\spannR{\bm i,\bm j}$. In both cases $\asso^1 f$ is a well-defined Euclidean curve and from \eqref{eq:assoEdgeSandwich} and  (with $\varepsilon=\pm1$ depending on the space form of $f$)
\begin{align}
\label{eq:uLamForm2}
(\bm 1+u_{01})^{-1}(\bm 1-u_{01})=\frac{(\bm 1+ u^*_{01})(\bm1-u_{01})}{(1+\frac12\tr u_{01})^2+\det\vec u_{01}}=\frac{u_{01}^*-u_{01}}{2+\tr u_{01}}=\frac{-2\vec u_{01}}{4-\varepsilon\eta^2}
\end{align}
we conclude that the arc-length of $\asso^1f$ is given by
\begin{align}
\label{eq:TDelta2}
\asso^1\EucDelta=\sqrt{\varepsilon\det u^1 _{01}}=\frac{2\dDelta}{4-\varepsilon\eta^2}=\frac{\eta^2}{2{\dDelta}}
\end{align}
which is constant. Regularity follows from regularity of $f$ and we conclude that $\asso^1f\in\CE$.
\end{proof}

The associated family preserves the curvature (up to a constant):

\begin{prop}
\label{prop:familyCurvature}
The curvature of $\asso^\lambda f$ is given by $\asso^\lambda\kappa=c\kappa$ for the constant $c=\frac{\dDelta}{\asso^\lambda \dDelta}$ where
\begin{enumerate}
\item $c=\frac{1+\lambda^2\EucDelta^2}{2|\lambda|}$ for $f\in\CE(\EucDelta)$ and
\item $c=2(1+\varepsilon\frac{\eta^2}{4})$ for $f\in\CS(\eta)$ or $f\in\CH(\eta)$ with $\lambda=1$.
\end{enumerate}
\end{prop}

\begin{proof}
From \eqref{eq:assoEdgeSandwich},\eqref{eq:uLamForm1} and \eqref{eq:uLamForm2} we see that $\asso^\lambda u_{01}=(\Phi_0^\lambda)^{-1}(x+y\vec u_{01})\Phi_0^\lambda=(\Phi_1^\lambda)^{-1}(x+y\vec u_{01})\Phi_1^\lambda$ holds in all cases for some constant scalar quantities $x,y$. From this, we obtain $\asso^\lambda H=(\Phi^\lambda)^{-1}H\Phi^\lambda$. We use the identities in Proposition \ref{prop:Hcurv} for $H$ and $\asso^\lambda H$ and in the Euclidean case we write $\dDelta=\EucDelta$:
For $f\in\CE$ this yields
\begin{align*}
\bm 1+ \frac{\asso^\lambda\dDelta}{2}\asso^\lambda \kappa \asso^\lambda F=\asso^\lambda H=(\Phi^\lambda)^{-1}H\Phi^\lambda=\bm 1+\frac{\dDelta}{2}\kappa \asso^\lambda F
\end{align*}
while for $f\in\CS$ or $f\in\CH$ this yields (using $F \Phi^1=\Phi^1\bm k$)
\begin{align*}
\bm 1+\frac{\asso^1\dDelta}{2}\asso^1\kappa \bm k=\asso^1 H=(\Phi^1)^{-1}H\Phi^1=\bm 1+ \frac{\dDelta}{2}\kappa \bm k.
\end{align*}
In both cases we have $\asso^\lambda\kappa=\frac{\dDelta}{\asso^\lambda \dDelta}\kappa$ and with \eqref{eq:TDelta1},\eqref{eq:TDelta2} the claim follows.
\end{proof}

By rescaling the metric of a space form by a factor one scales the curvature with the inverse factor. Therefore, if we scale the metric of the space form containing $\asso^\lambda f$ by the factor $c$ we find that $f$ and $\asso^\lambda f$ have the same curvature.  

For suitable $\lambda_1,\lambda_2$ applying the associated family to a curve $f$ twice yields a curve $\asso^{\lambda_2}\asso^{\lambda_1}f$ in the original space form of $f$. 

\begin{prop}
\label{prop:FamilyIdentity}
The associated family is reversible in the following sense:
\begin{enumerate}
\item $\asso^1\circ \asso^1=\id$ on $\CS$ and
\item $\asso^{-i}\circ \asso^1=\id$ on $\CH$.
\end{enumerate}
\end{prop}

\begin{proof}
Let $f\in\C$ be a curve in a non-Euclidean space form and set $\lambda=1$ in the spherical case and $\lambda=-i$ in the hyperbolic case. Then, in both cases $\lambda \asso^1u_{01}= u^1_{01}$. Therefore,
\begin{align*}
\bm1+\lambda \asso^1 u_{01}=\bm1+ u^1_{01}&=(\Phi_1^1)^{-1}(\bm 1+u_{01}+\bm 1-u_{01})\Phi^1_0=2(\Phi_1^1)^{-1}\Phi^1_0,\\
\bm1-\lambda \asso^1 u_{01}=\bm1- u^1_{01}&=(\Phi_1^1)^{-1}(\bm 1+u_{01}-\bm 1+u_{01})\Phi^1_0=2(\Phi_1^1)^{-1}u_{01}\Phi^1_0.
\end{align*}
The frame $\Psi=\asso^1\Phi$ of $\asso^1f$ fulfils $\Psi_1^\lambda=(\bm1+\lambda \asso^1 u_{01})\Psi_0^\lambda$ and, hence, $\Psi^\lambda(t_k)=2^k(\Phi^1(t_k))^{-1}$ where $t_k\in\V$ denotes the $k$-th vertex. The transport matrix $\asso^\lambda\asso^1 u$ of the curve $\asso^\lambda\asso^1f$ is given by~\eqref{eq:assoEdge} applied to $\asso^1 u$ and we compute
\begin{align*}
\asso^\lambda\asso^1=(\Psi_1^\lambda)^{-1}(\bm 1-\lambda \asso^1 u_{01})\Psi^\lambda_0=\frac12\Phi^1_1(\bm 1-\lambda \asso^1 u_{01})(\Phi_0^1)^{-1}=u_{01}.
\end{align*}
Hence, $\asso^\lambda\circ \asso^1 f=f$.
\end{proof}

\begin{rem}
Using similar techniques one can also show that $\asso^1\circ \asso^\lambda=c\, \id$ on $\CE$ for some constant $c\in\mathbb{R}$.
\end{rem}

%%%%%%%%%%%%%%%%%%%%%%%%%%%%%%%%%%%%%%%%
\section{Discrete constrained elastic curves in space forms}\label{sec:constr_elastic}
In this section we introduce discrete elastic and area-constrained elastic curves in various space forms via a curvature equation. This definition generalizes the known discrete notion of (area-constrained) elastic curves in Euclidean space as discussed in~\cite{ lagrangeTop, hoffmannSmokeRingFlow, HKtodaLattice, fairing_elastica}.

Based on this definition, we show that these curves satisfy discrete versions of several properties well-known from smooth theory such as the existence of a straight or circular directrix and a certain proportionality condition for the curvature. These observations will then be key to prove in Section~\ref{sec:flows} that discrete constrained elastic curves in space forms are invariant curves of discrete flows built from B\"acklund transformations.
%%%%%%%%%%%%%%%%%%%%%%%%
%
%
%%%%%%%%%%%%%%%%%%%%%%%%

Throughout this section we consider regular discrete curves, $f \in \mathcal{C}_\mathcal{Q}$, which are parametrized by arc-length in the space form $\Q$ of constant sectional curvature $\varepsilon=\kappa_\Q$.

\begin{defi}\label{def:curvature_equ}
A discrete curve $f \in \mathcal{C}_\mathcal{Q}(\eta)$ is said to be \emph{constrained elastic in the space form~$\mathcal{Q}$} if it satisfies the following \textit{curvature equation} 
\begin{equation}\label{eq:elasticaCurvature}
\kappa_{\bar{1}} + \kappa_1 = \frac{\xi \kappa_0 + \delta}{1+\frac{\dDelta^2}{4}\kappa_0^2},
\end{equation}
where $\dDelta^2:= \eta^2(1-\varepsilon\frac{\eta^2}{4})$ and $\xi, \delta \in \mathbb{R}$ are suitable constants.

More specifically, the curve is called \emph{elastic} if $\delta=0$ and \emph{area-constrained elastic}, otherwise.
\end{defi}

As a first simple example we see that discrete circles, that is, curves with non-vanishing constant curvature, satisfy the curvature equation. In particular, in this case the choice of the constants~$\xi$ and~$\delta$ is not unique. To avoid case studies, in what follows we will consider discrete circles as area-constrained elastic curves.  
%%%%%%%%%%%%%%%%%%%%%%%%%%
%
%
%%%%%%%%%%%%%%%%%%%%%%%%%%
\begin{prop}\label{prop:proportionality}
A discrete curve $f \in \mathcal{C}_\mathcal{Q}$ is constrained elastic in the space form~$\mathcal{Q}$ if and only if there exists a vector~$\dir \in \mathbb{R}^{3,2}\setminus \{ 0\}$ and a conserved quantity $\omega \in \mathbb{R}^\times$ such that $\kappa_0 = \omega\inner{\f_0, \dir}$.
\end{prop}
\begin{proof}
Suppose that the curvature of the discrete curve~$f \in \mathcal{C}_\mathcal{Q}(\eta)$ fulfills $\kappa_0 = \omega\inner{\f_0, \dir}$, then we obtain the following two identities
\begin{equation*}
\begin{aligned}
\kappa_0 - \kappa_1 &= \omega\eta \inner{\nor_{01}-\p, \dir} = \omega \inner{\f_0 - \f_1,\dir},
\\\kappa_0 \kappa_1 &= -2 \omega \inner{\tang_{01}, \dir},
\end{aligned}
\end{equation*}
where the latter equation follows from the first of the Frenet-type equations~\eqref{eq:frenet_eq_1}. Together with \eqref{eq:frenet_eq_2}, we then conclude that
\begin{equation*}
\begin{aligned}
(\kappa_{\bar{1}} + \kappa_1)(1+\frac{\dDelta^2}{4} \kappa_0^2)&= \kappa_{\bar{1}}+\kappa_1 + \frac{\dDelta^2}{4} \kappa_0 (\kappa_{\bar{1}}\kappa_0 + \kappa_0\kappa_1)
\\&= 2 \kappa_0 + \omega \eta^2 \{ \frac{1}{\eta} \inner{\nor_{\bar{1}0} - \nor_{01}, \dir} - \frac{\kappa_0}{2}(1+\frac{\inner{\q,\q}}{4}\eta^2) \inner{\tang_{\bar{1}0} + \tang_{01}, \dir} \}
\\&= (2 - \omega \dDelta^2 \inner{\p, \dir} + \eta^2 \inner{\q,\q})\kappa_0 + \omega \eta^2 \inner{\q,\dir}.
\end{aligned}
\end{equation*}
Thus, any such discrete curve satisfies the curvature equation with $\xi:=2 - \omega \dDelta^2 \inner{\p, \dir}+ \eta^2 \inner{\q,\q}$ and $\delta:=\omega \eta^2 \inner{\q,\dir}$.

\bigskip

%%%%%%%%%%%%%%%%%%%%%%%%%%%%%%%%%%%%%%%
%
%%%%%%%%%%%%%%%%%%%%%%%%%%%%%%%%%%%%%%%
Conversely, suppose that $f \in \mathcal{C}_\mathcal{Q}(\eta)$ satisfies the curvature equation~(\ref{def:curvature_equ}). Moreover, we define the following vectors~$\mathfrak{a}_0 \in \mathbb{R}^{3,2}\setminus \{ 0 \}$ along the curve: 
\begin{equation*}
    \mathfrak{a}_0:= x_0 \f_0 + y_0 \{ \f_{\bar{1}} - \f_1 \} + z_0 \{ \tang_{\bar{1}0} + \tang_{01} -2\p \} - \kappa_0 \q + \alpha\p, 
\end{equation*}
where $\alpha:= - \frac{1}{\dDelta^2} (2 - \xi + \eta^2 \inner{\q,\q})$ and the further coefficients along the curve are given by  
\begin{equation*}
\begin{aligned}
           x_0&:=-\frac{\delta}{\eta^2} - \kappa_0 \inner{\q,\q},
           \\y_0&:=- \frac{\kappa_{\bar{1}}-\kappa_1}{2 \inner{\f_{\bar{1}}, \f_1}},
           \\z_0&:= \frac{1}{2\inner{\tang_{\bar{1}0}, \tang_{01}}+4} (2 \alpha - \frac{\kappa_0}{2}(\kappa_{\bar{1}} + \kappa_1)).
\end{aligned}
\end{equation*}
We will prove that the vectors $\mathfrak{a}_0$ are indeed constant along the discrete curve~$f$. Since $\kappa_0=\inner{\f_0, \mathfrak{a}_0}$, this observation will then provide the sought-after equation with $\omega:=1$.

To show that $\mathfrak{a}_0 = \mathfrak{a}_1$, we first observe that 
\begin{equation*}
    \begin{aligned}
        \inner{\mathfrak{a}_0, \p} = \inner{\mathfrak{a}_1, \p} = -\alpha, \ \ & \ \inner{\mathfrak{a}_0, \q} =  \inner{\mathfrak{a}_1, \q} =\frac{\delta}{\eta^2},
        \\\inner{\mathfrak{a}_0, \tang_{01}} = \inner{\mathfrak{a}_1, & \tang_{01}} = \kappa_0 \kappa_1.
    \end{aligned}
\end{equation*}
Moreover, we have $\inner{\mathfrak{a}_0, \f_1} = \kappa_1$  and $\inner{\mathfrak{a}_1, \f_0} = \kappa_0$. These two latter equations are obtained by using the curvature equation and the following identity
\begin{equation*}
    \frac{\kappa_0 \dDelta^2}{2} = \frac{\inner{\f_1, \tang_{\bar{1}0}}}{ \inner{\tang_{\bar{1}0},\tang_{01}} +2} = 
    \frac{\inner{\f_{\bar{1}}, \tang_{01}}}{ \inner{\tang_{\bar{1}0},\tang_{01}} +2}.
\end{equation*}
It follows from taking inner product of the Frenet-type equation~(\ref{eq:frenet_eq_2}) with the tangents $\tang_{\bar{1}0}$ and~$\tang_{01}$.

As a consequence, we conclude that $\mathfrak{a}_0 - \mathfrak{a}_1$ is orthogonal to all vectors in the basis $\{ \f_0 + \f_1, \f_0 - \f_1, \tang_{01}, \q, \p \}$ and therefore has to vanish.
\end{proof}
%%%%%%%%%%%%%%%%%%%%%%%%
%
%
%%%%%%%%%%%%%%%%%%%%%%%%
\begin{figure}
    \begin{minipage}{8cm}
        \vspace*{2.5cm}
        \includegraphics[scale=0.4]{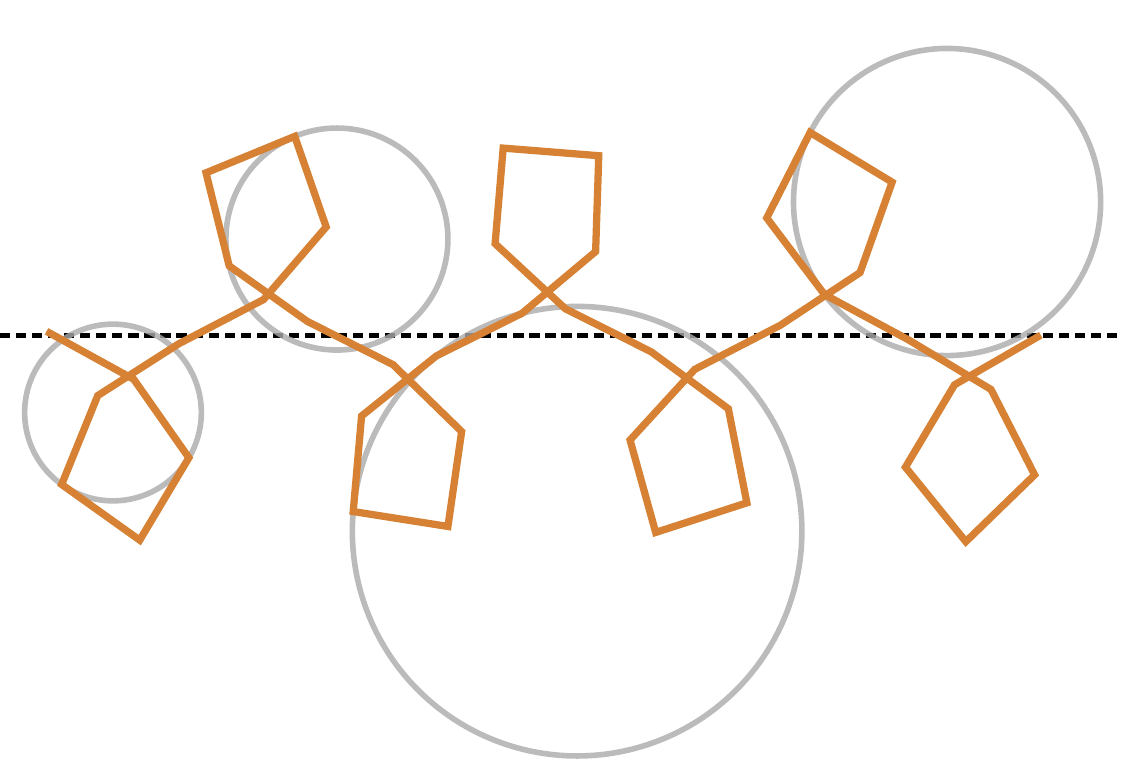}
    \end{minipage}
    \begin{minipage}{5cm}
        \hspace*{0.42cm}\includegraphics[scale=0.5]{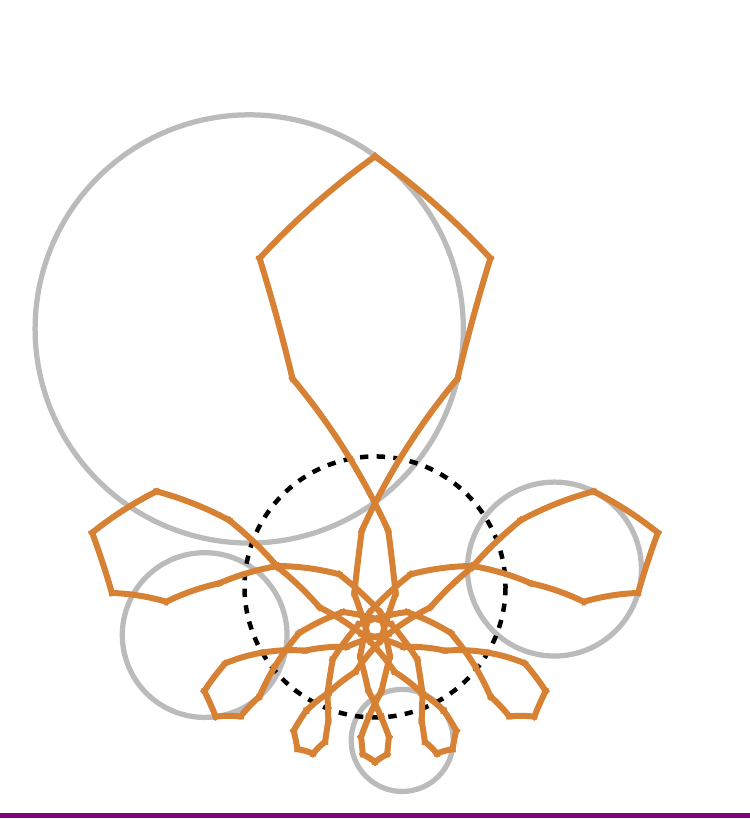}
    \end{minipage}
        \caption{A discrete elastic curve in Euclidean space and a discrete area-constrained elastic curve in hyperbolic space with their directrices (dotted) and some double-curvature circles (gray). The latter intersect the directrix at a constant angle.}
    \label{fig:E_distances}
\end{figure}
%
%
%%%%%%%%%%%%%%%%%%%%%%%%
%
%
%%%%%%%%%%%%%%%%%%%%%%%%
Smooth constrained elastic curves in space forms can be characterized by the existence of a special tangential  circle congruence, which gives rise to a distinguish directrix for each such curve~\cite{chopembszew}. Next we will prove the existence of a similar circle congruence in the discrete case, the double-curvature circle congruence. 

As illustrated in Figure~\ref{fig:double_circles}, this double-curvature circle associated to the curve point $f_0$ is in oriented contact to the two adjacent tangent lines~$t_{\bar{1}0}$ and $t_{01}$ and goes through the curve points $f_{\bar{1}}$ and $f_1$.

\begin{propdefi}\label{prop:elastic_complex}
A discrete curve $f \in \mathcal{C}_\mathcal{Q}$ is constrained elastic in the space form~$\mathcal{Q}$ if and only if the double-curvature circle congruence $\mathfrak{c}_0=\frac{\kappa_0}{2}\f_{\bar{1}} + \tang_{\bar{1}0}=\frac{\kappa_0}{2}\f_1 + \tang_{01}$ lies in a fixed linear complex; that is, there exists a linear complex~$a\in \mathbb{P}(\mathbb{R}^{3,2})$ such that $\inner{\mathfrak{c}_0, \dir}=0$ holds for all circles~$c_0$ in the congruence.

The directrix $a^\star$ of this associated linear circle complex $a$ is then called the \emph{directrix of the constrained elastic curve}. 
\end{propdefi}
\begin{proof}
Firstly, suppose that the discrete curve is constrained elastic in a space form. By Proposition~\ref{prop:proportionality}, we then know that $\kappa_0 = \omega \inner{\f_0, \tilde{\dir}}$ for a suitable linear complex~$\tilde{\dir} \in \mathbb{R}^{3,2}$ and a real constant~$\omega$. The following equations
\begin{equation*}
\begin{aligned}
\inner{\mathfrak{c}_0, \tilde{\dir}}&=\inner{\frac{\kappa_0}{2}\f_{\bar{1}} + \tang_{\bar{1}0},\tilde{\dir}}=\frac{\kappa_{\bar{1}} \kappa_0}{2\omega} + \inner{\tang_{\bar{1}0},\tilde{\dir}} = \inner{\frac{\kappa_{\bar{1}}}{2}\f_0 + \tang_{\bar{1}0},\tilde{\dir}} = \inner{\mathfrak{c}_{\bar{1}}, \tilde{\dir}},
\\\inner{\mathfrak{c}_0, \tilde{\dir}}&=\inner{\frac{\kappa_0}{2}\f_1 + \tang_{01},\tilde{\dir}}=\frac{\kappa_0 \kappa_1}{2\omega} + \inner{\tang_{01},\tilde{\dir}}=\inner{\mathfrak{c}_1, \tilde{\dir}}
\end{aligned}
\end{equation*}
then show that $\inner{\mathfrak{c}_0, \tilde{\mathfrak{a}}} \equiv \chi \in \mathbb{R}$ along the curve. Thus, the double-circle congruence $c$ satisfies $\inner{\mathfrak{c}_0, \tilde{\dir}-\chi \p} = 0$ and therefore lies in the fixed linear complex~$\dir:=\tilde{\dir}-\chi \p$.

Conversely, assume that the double circle congruence lies in a fixed linear complex~$a$. From \eqref{eq:double_circles} we obtain that
\begin{equation*}
\kappa_0=-2\frac{\inner{\tang_{\bar{1}0},\dir}}{\inner{\f_{\bar{1}},\dir}} = -2\frac{\inner{\tang_{01},\dir}}{\inner{\f_1,\dir}} \ \text{and } \kappa_1=-2\frac{\inner{\tang_{01},\dir}}{\inner{\f_0,\dir}} = -2\frac{\inner{\tang_{12},\dir}}{\inner{\f_2,\dir}}.
\end{equation*}
and conclude that
\begin{equation*}
\frac{\kappa_{\bar{1}}}{\inner{\f_{\bar{1}}, \dir}}=\frac{\kappa_0}{\inner{\f_0, \dir}}=-2\frac{\inner{\tang_{\bar{1}0},\dir}}{\inner{\f_{\bar{1}},\dir}\inner{\f_0, \dir}} = -2\frac{\inner{\tang_{12},\dir}}{\inner{\f_1,\dir}\inner{\f_2, \dir}}=\frac{\kappa_1}{\inner{\f_1, \dir}}=\frac{\kappa_2}{\inner{\f_2, \dir}} \equiv: \omega.
\end{equation*}
Thus, by Proposition~\ref{prop:proportionality}, this proves that the discrete curve is indeed constrained elastic.
\end{proof}

From the proof of Proposition~\ref{prop:proportionality}, we therefore conclude that a discrete constrained elastic curve is discrete elastic, $\delta=0$, if and only if its directrix is a geodesic in the corresponding space form~$\Q$, that is, $\inner{\mathfrak{a}, \q} = \inner{\mathfrak{a}^\star, \q}=0$.
%%%%%%%%%%%%%%%%%%%%%%%%
%
%
%%%%%%%%%%%%%%%%%%%%%%%%
\begin{figure}
    \begin{minipage}{7cm}
        \includegraphics[scale=0.27]{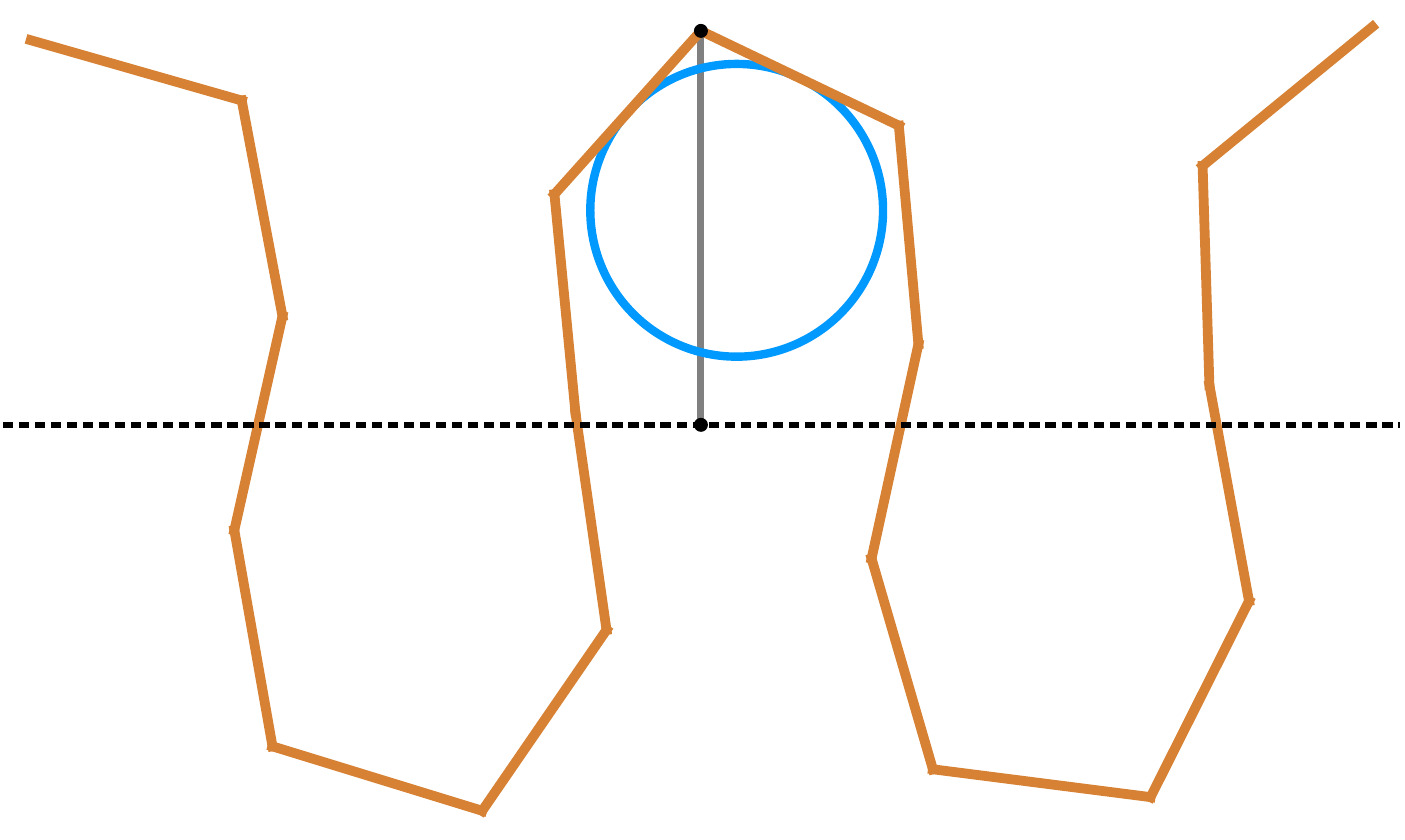}
    \end{minipage}
    \begin{minipage}{5cm}
        \includegraphics[scale=0.4]{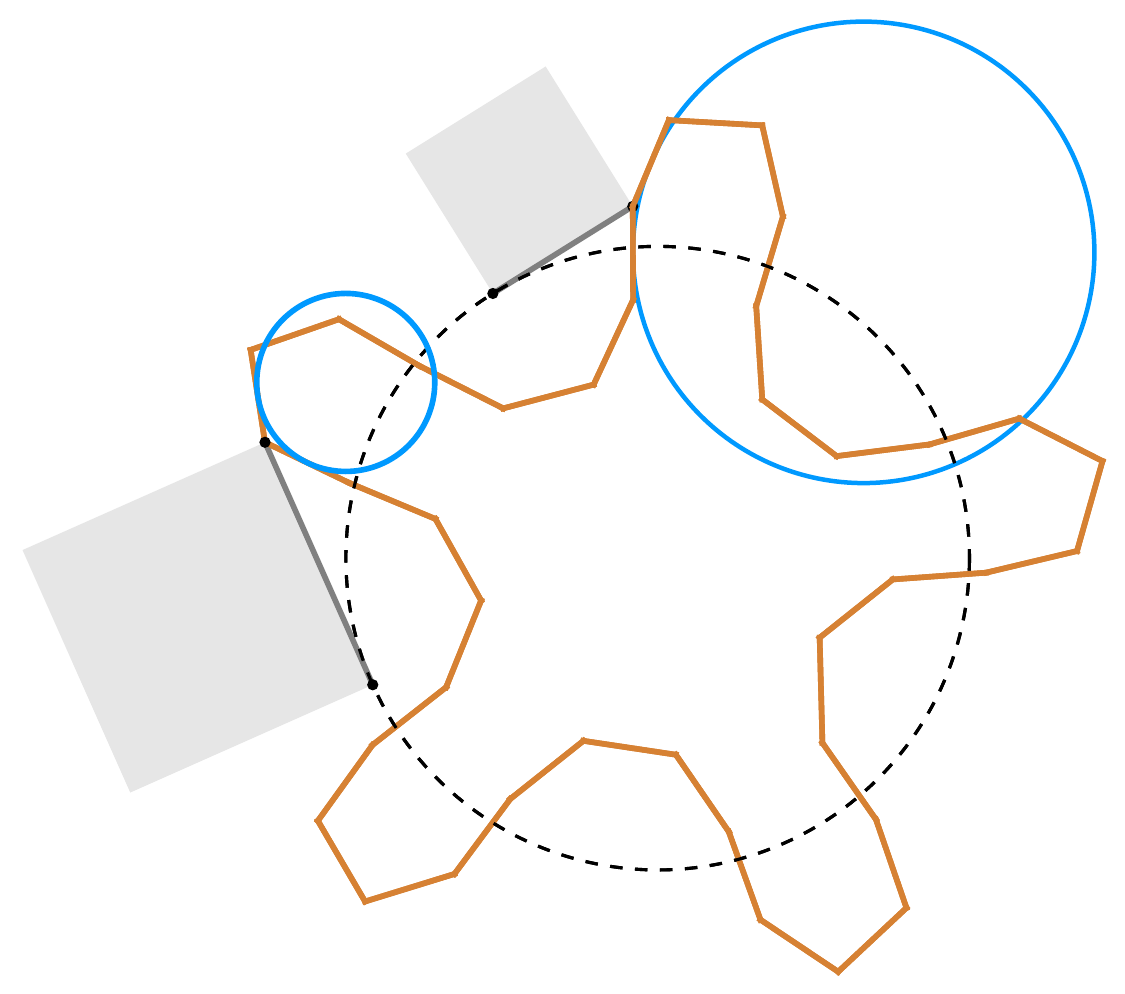}
    \end{minipage}
        \caption{A discrete (area-constrained) elastic curve in Euclidean space: its curvature and orthogonal (squared tangential) distance to the corresponding directrix are proportional (see~Corollary~\ref{cor:ConstrElasticProp}).}
    \label{fig:E_distances}
\end{figure}
%
%
%%%%%%%%%%%%%%%%%%%%%%%%%%%%%%%%%%%%

\bigskip

For smooth elastic curves in Euclidean space, it is well-known that they may be characterized as curves whose curvature is proportional to the distance of a line, its directrix. As discussed in~\cite{PhysRevE.65.031801}, also smooth area-constrained elastic curves can be described in a similar way: their curvature is proportional to the squared tangential distance to a fixed circle. 

The \emph{tangential distance $d$} in Euclidean space of a point $f\in\mathbb{R}^2$ to a circle with center $c\in\mathbb{R}^2$ and radius $r$ is given by
\begin{align*}
d:=\sqrt{\|f-c\|^2-r^2}.
\end{align*}
For $r\in\mathbb{R}$ and the point~$f$ outside the circle this is the distance of $f$ to a point on the circle such that the line through $f$ and this point is tangent to the circle. However, we extend this definition to include imaginary radii and points inside the circle and, thus, $d$ can be imaginary.

\bigskip

As a consequence of Proposition~\ref{prop:elastic_complex}, we obtain these characterizations also in the discrete case~(cf.\,Figure~\ref{fig:E_distances}):
\begin{cor}
\begin{itemize}
    \item[(i)] A curve $f\in\CE$ is elastic in $\mathbb{E}^2$ if and only if there is a line such that the signed orthogonal distance $d$ of each curve point to the line is proportional to the curvature at this vertex: $d=c \kappa$ for a constant $c\in\mathbb{R}$.\label{cor:elasticProp}

    \item[(ii)] A curve $f\in\CE$ is area-constrained elastic in $\mathbb{E}^2$ if and only if there exists a circle such that the squared tangential distance $d$ of each curve point to this circle is proportional to the curvature, that is, $d^2=c \kappa$ for a constant $c\in\mathbb{R}$.\label{cor:ConstrElasticProp}
\end{itemize}
\end{cor}

\begin{proof}
The key to this proof is the characterization given in Proposition~\ref{prop:proportionality} together with an understanding of concentric Euclidean circles centered at the curve points. 

Thus, let us fix a Euclidean space form $\Q$ determined by the vector~$\q:=\q_0$. Moreover, let $m \in \mathbb{P}(\mathcal{L})$ be a point and fix homogeneous coordinates such that $\inner{\mathfrak{m}, \q}=-1$. Then the family of oriented circles with (Euclidean) center~$m$ and radius~$r$ is given by
\begin{equation*}
    r \mapsto s_r:=\mathfrak{m}-\frac{r^2}{2}\q + r\p.
\end{equation*}
The signed distance $d$ of a point $m$ to a Euclidean line $l \in \mathbb{P}(\mathcal{L})$, $\inner{\mathfrak{l}, \q}=0$, is given by the radius of the circle $s_r$ which is tangent to the line $l$. Thus, the distance $d$ is given by
\begin{equation*}
    0 = \inner{\mathfrak{m}-\frac{d^2}{2}\q + d\p, \mathfrak{l}} \ \ \Rightarrow \ \ d=-\frac{\inner{\mathfrak{m},\mathfrak{l}}}{\inner{\mathfrak{l},\p}}. 
\end{equation*}

Moreover, the signed tangential distance $d$ of a point $m$ to a Euclidean circle $l \in \mathbb{P}(\mathcal{L})$ is given by the radius of the (possibly imaginary) circle~$s_r$ which intersects~$l$ orthogonally:
\begin{equation*}
    0 = \inner{\mathfrak{m}-\frac{d^2}{2}\q + d\p, \mathfrak{l} + \inner{\mathfrak{l},\p}\p} \ \ \Rightarrow \ \ d^2=\frac{2\inner{\mathfrak{m},\mathfrak{l}}}{\inner{\mathfrak{l},\q}}. 
\end{equation*}
The claim follows now directly by setting $m=f_0$ and using Proposition~\ref{prop:proportionality}.
\end{proof}
%%%%%%%%%%%%%%%%%%%%%%%%%%%%%%%%%%%%
%
%
%%%%%%%%%%%%%%%%%%%%%%%%%%%%%%%%%%%%

\section{Flows and invariant curves}
\label{sec:flows}

We will now apply an integer number of B\"acklund transformations to a discrete curve which we can interpret as a step of a discrete flow of the curve. The invariant curves of each flow give rise to a hierarchy of curves and, 
as we will see in Theorem \ref{thm:invariantCurvesConstElastica}, this hierarchy contains elastic and constrained elastic curves. 

\subsection{B\"acklund transformations and $n$-invariant curves}

In this section we investigate transformations of curves built from elementary quads. For each quad~$(a,b,c,d)$ 
we will assume the regularity condition of non-vanishing diagonals $a\neq c$ and $b\neq d$. However, we will encounter some special cases where neighboring points can coincide. 

\begin{defi}
We call a quad $(a,b,c,d)$ a \emph{$\Q$-Darboux butterfly} if there is a reflection in a geodesic of the space form $\Q$ interchanging the pairs $a\leftrightarrow c$ and $b\leftrightarrow d$.
\end{defi}

\begin{figure}[h!]
  \centering
  \begin{overpic}[scale=.2,tics=10]{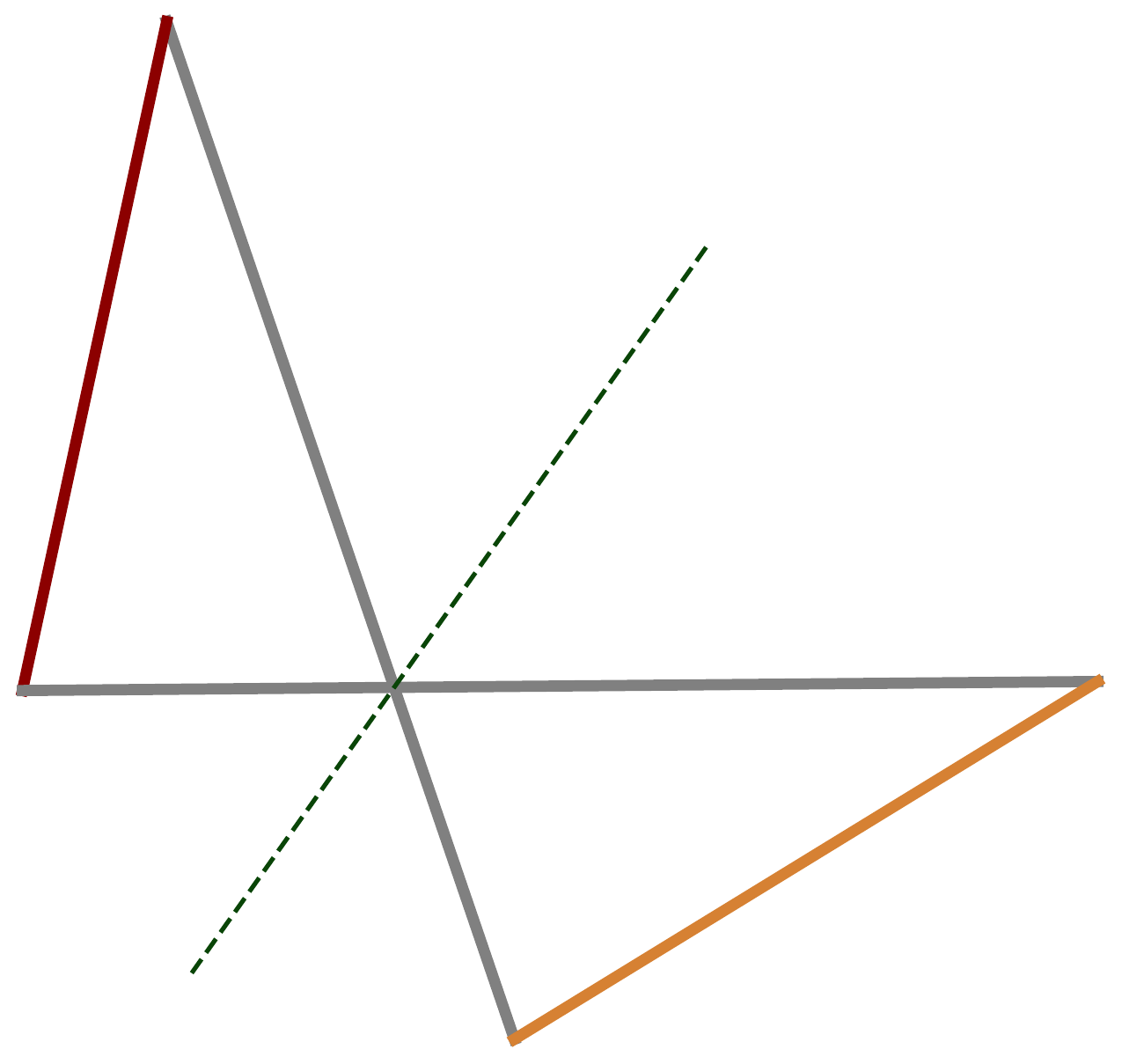}
  \put(38,0){$a$}
  \put(100,31){$b$}
  \put(0,27){$c$}
  \put(17,92){$d$}
  \end{overpic}
  \hspace{3mm}
  \begin{overpic}[scale=.2,tics=10]{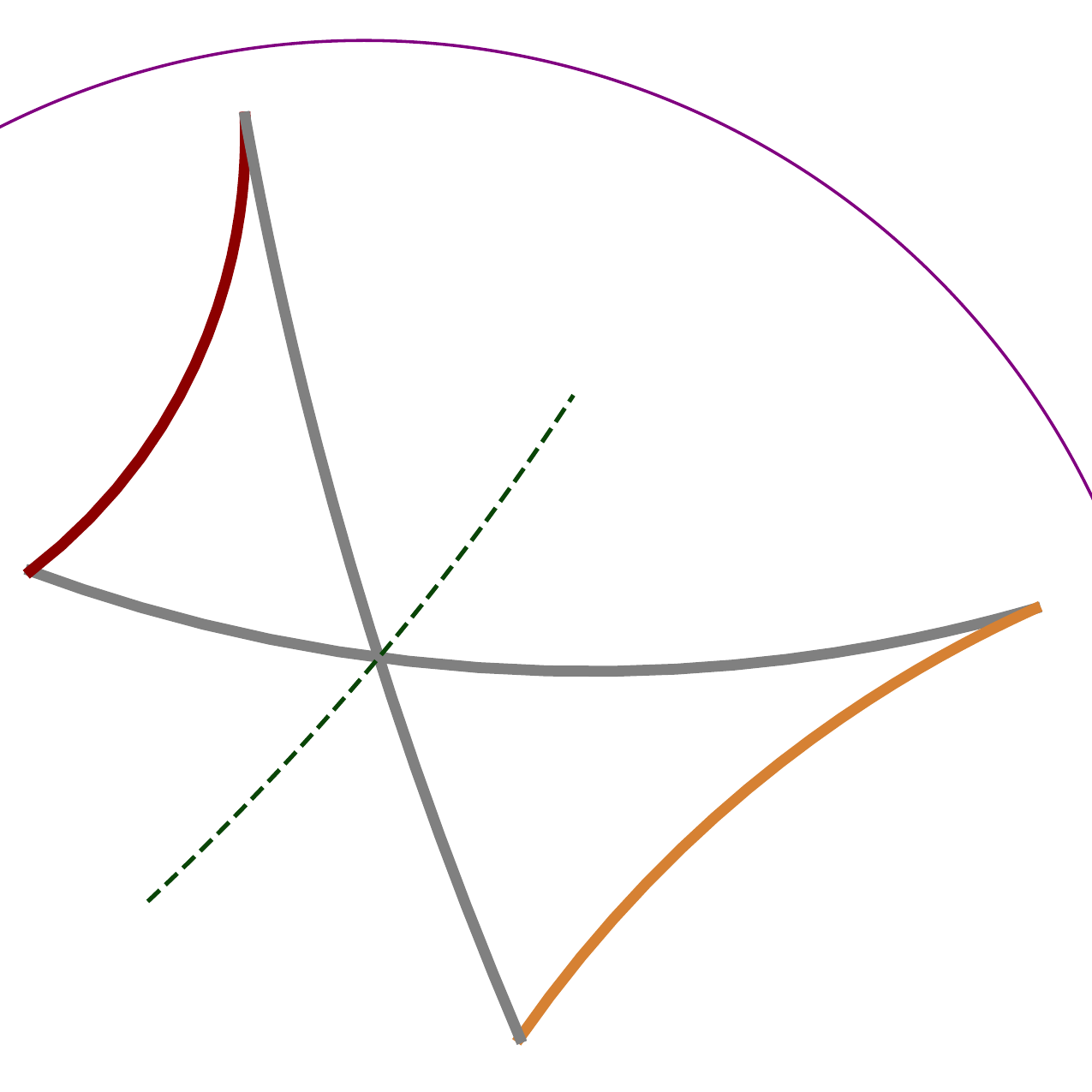}
  \put(42,0){$a$}
  \put(97,42){$b$}
  \put(0,40){$c$}
  \put(24,89){$d$}
  \end{overpic}
  \hspace{8mm}
  \begin{overpic}[scale=.32,tics=10]{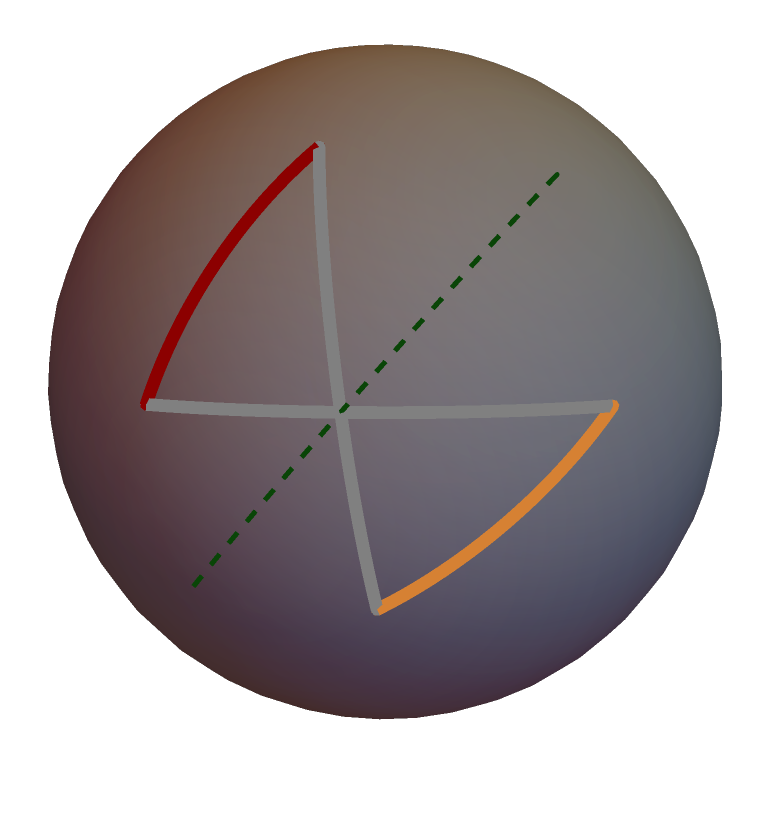}
  \put(40,20){$a$}
  \put(77,48){$b$}
  \put(15,43){$c$}
  \put(42,83){$d$}
  \end{overpic}
  \caption{An $\QE$-,$\QH$- and $\QS$-Darboux butterfly with its axis of reflection (green).}
\end{figure}

The name \emph{Darboux butterfly} is adopted from \cite{discrete_bicycle}. In Euclidean geometry these quads are \emph{anti-parallelograms}. 
Any $\Q$-Darboux butterfly is circular and opposite edges have the same length in $\Q$. On the other hand, adjacent edges of the quad do not have the same length: If $a$ would have the same distance to $b$ and $d$ then the reflection interchanging $b\leftrightarrow d$ would map $a$ to itself violating the regularity assumption. 
A $\Q$-Darboux butterfly is uniquely determined by any three of its points. 
We study these quads in the matrix model. 
For $\Q=\QS$ and $\Q=\QH$ we assume neighboring points not to be antipodal. 
Many equations can be translated to the light cone model using a Moutard equation (see~\cite[Section 2.3]{bobenkoSurisBook}). 

\begin{prop}
\label{prop:butterflyFormula}
A quad $(a,b,c,d)$ is a $\Q$-Darboux butterfly if and only if its matrix representations $A,B,C,D$ fulfil
\begin{enumerate}
\item $(C-B)(B-A)=(C-D)(D-A)$ for $\Q=\mathbb{E}$ or
\item $CB+BA=CD+DA$ for $\Q=\mathbb{S}$ and $\Q=\mathbb{H}$.
\end{enumerate}
\end{prop}

\begin{proof}
In Euclidean geometry the statement is translation invariant and we can assume that the fixed line of the reflection interchanging $B\leftrightarrow D$ intersects the origin. This implies $\det B=\det D$ and the equation reads
\begin{align*}
0=(C-B)(B-A)-(C-D)(D-A)=C(B-D)+(B-D)A.
\end{align*}
This is equivalent to
\begin{align*}
C=-(B-D)A(B-D)^{-1}=\left(\bm k(B-D)\right)A\left(\bm k (B-D)\right)^{-1}
\end{align*}
which is exactly the condition that the pair $A,C$ gets interchanged by the same reflection as the pair $B,D$. 
Similarly, in non-Euclidean geometry the equation is equivalent to
\begin{align*}
C=-(B-D)A(B-D)^{-1}
\end{align*}
which, again, is the condition that the pair $A,C$ gets interchanged by the same reflection as the pair $B,D$.
\end{proof}

\begin{defi}
Two curves $f,\tilde f\in\C$ are said to be related by a \emph{B\"acklund transformation} if each quad $(f_0,f_1,\tilde f_1,\tilde f_0)$ forms a $\Q$-Darboux butterfly.
\end{defi}

This transformation is also a special \emph{Darboux transformation}  since all quads have the same cross-ratio (which is positive) and, in the Euclidean case, it is known as the \emph{discrete bicycle transformation}~\cite{discrete_bicycle}.

\begin{figure}[h!]
  \centering
  \includegraphics[width=0.4\textwidth]{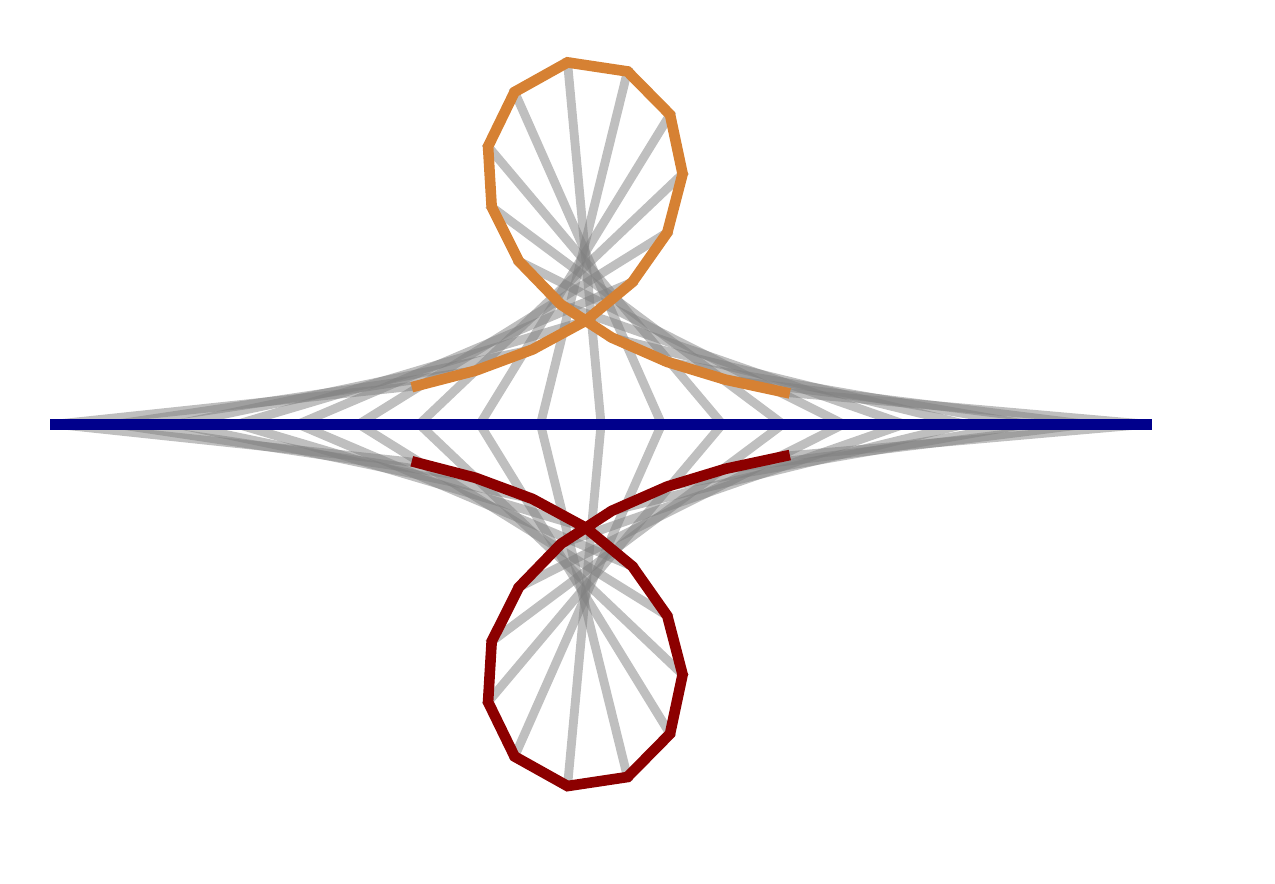}
  \includegraphics[width=0.4\textwidth]{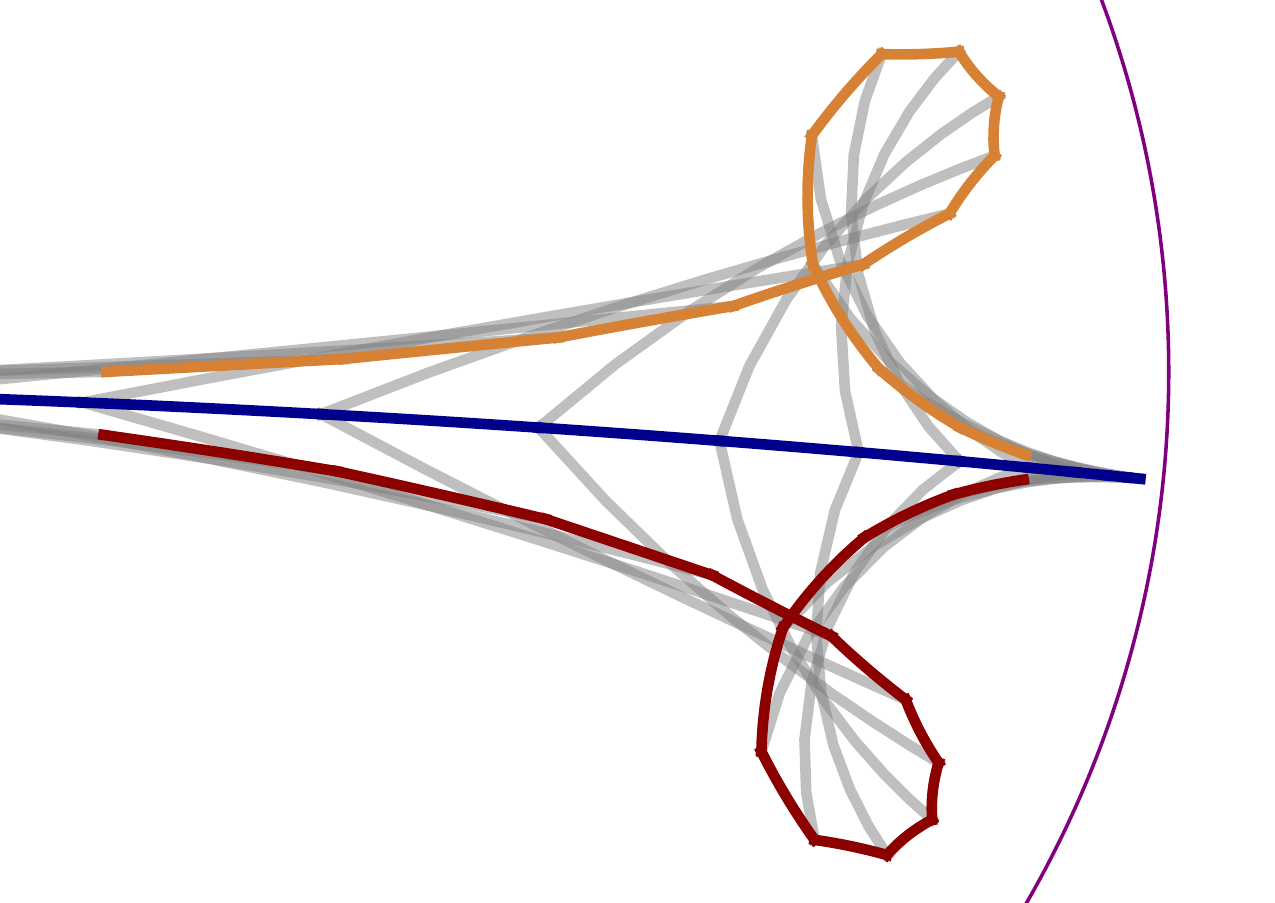}
   \caption{A B\"acklund transformation of an equally sampled geodesic in a space form (blue) yields a discrete Euler loop (orange). Since its reflection (red) along the geodesic is also a B\"acklund transform of the geodesic the Euler loop is $2$-invariant.}
\label{fig:eulerloop}
\end{figure}

One initial point $\tilde{f}(t_0)$ already determines the B\"acklund transformation of a curve $f\in\C$. 
Here, the point $\tilde{f}(t_0)$ can be any point (distinct from the antipodal $f(t_0)$ point in the non-Euclidean cases) that does not have the arc-length of the curve as distance to $f(t_0)$. 

In particular, if $\tilde{f}(t_0)=f(t_0)$ we obtain the trivial case where $\tilde{f}=f$ holds for all vertices and we call this the \emph{identity transformation}. 
This is the only case for which different points of a quad can coincide. 

\bigskip

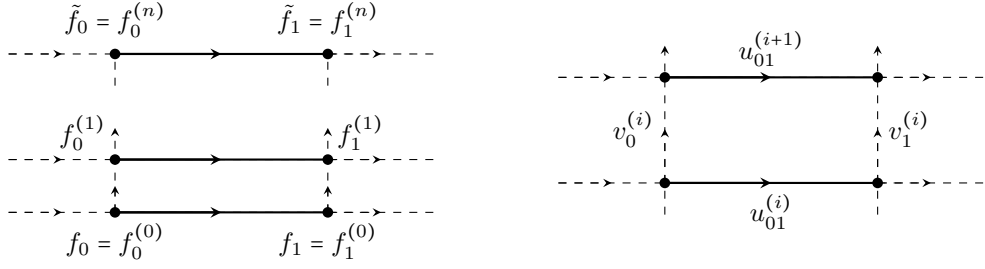
\begin{figure}
  \centering
\begin{minipage}{5cm}
\begin{tikzpicture}[scale=1.4]
%%%%%%%%%%%%%%%  line y=0  %%%%%%%%%%%%%%%%%%%%%%%%%
\draw [dashed]   (-2,0) -- (2,0);
\draw [thick]   (-1,0) -- (1,0);
\draw [dashed,-stealth]   (-2,0) -- (-1.5,0);
\draw [dashed,-stealth]   (1,0) -- (1.5,0);
\draw [thick,-stealth]   (-1,0) -- (0,0);
      \fill[black] (-1,0) circle (0.5mm) node[below=0.2mm] {$f_0=f^{(0)}_0$}
                   (1,0) circle (0.5mm) node[below=0.2mm] {$f_1=f^{(0)}_1$}
                   ;
%%%%%%%%%%%%%%%  line y=yy  %%%%%%%%%%%%%%%%%%%%%%%%%
\def\yy{.5}
\draw [dashed]   (-2,\yy) -- (2,\yy);
\draw [thick]   (-1,\yy) -- (1,\yy);
\draw [dashed,-stealth]   (-2,\yy) -- (-1.5,\yy);
\draw [dashed,-stealth]   (1,\yy) -- (1.5,\yy);
\draw [thick,-stealth]   (-1,\yy) -- (0,\yy);
      \fill[black] (-1,\yy) circle (0.5mm) node[above left] {$f^{(1)}_0$}
                   (1,\yy) circle (0.5mm) node[above right] {$f^{(1)}_1$}
                   ;
%%%%%%%%%%%%%%%  line y=yyy  %%%%%%%%%%%%%%%%%%%%%%%%%
\def\yyy{1.5}
\draw [dashed]   (-2,\yyy) -- (2,\yyy);
\draw [thick]   (-1,\yyy) -- (1,\yyy);
\draw [dashed,-stealth]   (-2,\yyy) -- (-1.5,\yyy);
\draw [dashed,-stealth]   (1,\yyy) -- (1.5,\yyy);
\draw [thick,-stealth]   (-1,\yyy) -- (0,\yyy);
      \fill[black] (-1,\yyy) circle (0.5mm) node[above=1mm] {$\tilde f_0=f^{(n)}_0$}
                   (1,\yyy) circle (0.5mm) node[above=1mm] {$\tilde f_1=f^{(n)}_1$}
                   ;
%%%%%%%%%%%%%%%  %%%%%%%%%%%%%%%%%%%%%%%%%
\draw [dashed,-stealth]   (-1,0) -- (-1,\yy+.3);
\draw [dashed,-stealth]   (-1,0) -- (-1,.25);
\draw [dashed]   (-1,\yyy-.3) -- (-1,\yyy);
\draw [dashed,-stealth]   (1,0) -- (1,\yy+.3);
\draw [dashed,-stealth]   (1,0) -- (1,.25);
\draw [dashed]   (1,\yyy-.3) -- (1,\yyy);
\end{tikzpicture}
\end{minipage}
\hspace{2cm}
\begin{minipage}{5cm}
\begin{tikzpicture}[scale=1.4]
%%%%%%%%%%%%%%%  line y=0  %%%%%%%%%%%%%%%%%%%%%%%%%
\draw [dashed]   (-2,0) -- (2,0);
\draw [thick]   (-1,0) -- (1,0);
\draw [dashed,-stealth]   (-2,0) -- (-1.5,0);
\draw [dashed,-stealth]   (1,0) -- (1.5,0);
\draw [thick,-stealth]   (-1,0) -- (0,0);
      \fill[black] (-1,0) circle (0.5mm) %node[below left] {$f^{(i)}_0$}
                   (1,0) circle (0.5mm) %node[below right] {$f^{(i)}_1$}
                   (0,0) node[below=0.2mm] {$u^{(i)}_{01}$}
                   ;
%%%%%%%%%%%%%%%  line y=yy  %%%%%%%%%%%%%%%%%%%%%%%%%
\def\yy{1}
\draw [dashed]   (-2,\yy) -- (2,\yy);
\draw [thick]   (-1,\yy) -- (1,\yy);
\draw [dashed,-stealth]   (-2,\yy) -- (-1.5,\yy);
\draw [dashed,-stealth]   (1,\yy) -- (1.5,\yy);
\draw [thick,-stealth]   (-1,\yy) -- (0,\yy);
      \fill[black] (-1,\yy) circle (0.5mm) %node[above left] {$f^{(i+1)}_0$}
                   (1,\yy) circle (0.5mm) %node[above right] {$f^{(i+1)}_1$}
                   (0,1) node[above=0.2mm] {$u^{(i+1)}_{01}$}
                   (-1,.5\yy) node[left=0.2mm] {$v^{(i)}_{0}$}
                   (1,.5\yy) node[right=0.2mm] {$v^{(i)}_{1}$}
                   ;
%%%%%%%%%%%%%%%  %%%%%%%%%%%%%%%%%%%%%%%%%
\draw [dashed,-stealth]   (-1,0) -- (-1,.5\yy);
\draw [dashed,-stealth]   (-1,-.3) -- (-1,\yy+.3);
\draw [dashed,-stealth]   (1,0) -- (1,.5\yy);
\draw [dashed,-stealth]   (1,-.3) -- (1,\yy+.3);
\end{tikzpicture}
\end{minipage}
\caption{Notation convention for a sequence of curves (\textit{left}) and, in particular, the transport matrices at one quad (\textit{right}).}
    \label{fig:notationBL}
\end{figure}

Now, we consider a sequence of curves $f^{(0)},...,f^{(n)}\in\C$. We consider this sequence~$f$ to live on the directed graph $\G_n=(\V_n,\E_n)$ with vertices $\V_n=\V\times\{0,...,n\}$ and $\mathbb{Z}^2$ combinatorics (see Figure~\ref{fig:notationBL}). The edge set $\E_n$ consists of edges directed in curve direction and edges directed in transformation direction. Analogous to the transport matrix $u$ in curve direction we define the transport matrix $v$ in transformation direction by $v^{(i)}_0:=F^{(i+1)}_0-F^{(i)}_0$ in Euclidean space and by $v^{(i)}_0:=F^{(i+1)}_0(F^{(i)}_0)^{-1}$ in non-Euclidean space.

We are interested in sequences of B\"acklund transformations, i.e., in the case where consecutive curves $f^{(i)}$ and $f^{(i+1)}$ are related by a B\"acklund transformation. As a direct consequence of Proposition \ref{prop:butterflyFormula} we have

\begin{cor}
\label{prop:TrafoIsSkewParNet}
A sequence of curves $f^{(0)},...,f^{(n)}$ is a sequence of B\"acklund transformations if and only if the corresponding map $p=(u,v):\E_n\to\matC$ defined by the transport matrices fulfills
\begin{align}
\label{eq:paradd}
u^{(i)}_{01}+v^{(i)}_1=v^{(i)}_0+u^{(i+1)}_{01}    
\end{align}
and
\begin{align}
\label{eq:parmult}
v^{(i)}_1u^{(i)}_{01}=u_{01}^{(i+1)}v_0^{(i)}
\end{align}
at every quad.
\end{cor}

One can also consider general maps $p=(u,v):\E_n\to\matC$ (not necessarily coming from transport matrices).

\begin{defi}
An edge based map $p:=(u,v):\E_n\to\matC$ is called a \emph{skew parallelogram net} if it fulfills \eqref{eq:paradd} and~\eqref{eq:parmult} at every quad.
\end{defi}

Thus, the transport matrices of a sequence of B\"acklund transformations form a skew parallelogram net. Skew parallelogram nets are studied in \cite{skewParallelogramNets} as lattice equation system encompassing many integrable systems studied in discrete differential geometry. 

\begin{figure}[h!]
  \centering
  \includegraphics[width=0.24\textwidth]{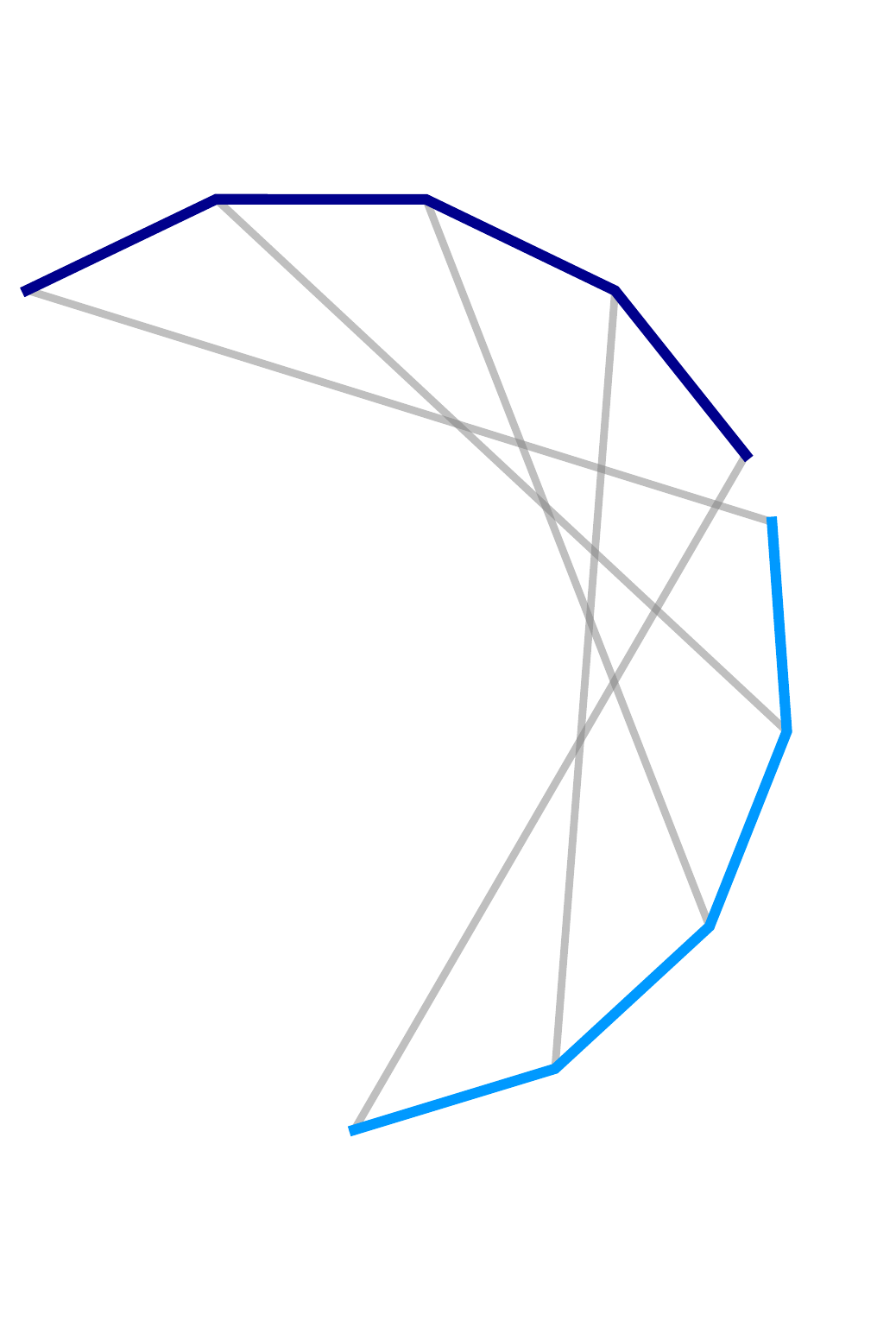}
  \includegraphics[width=0.24\textwidth]{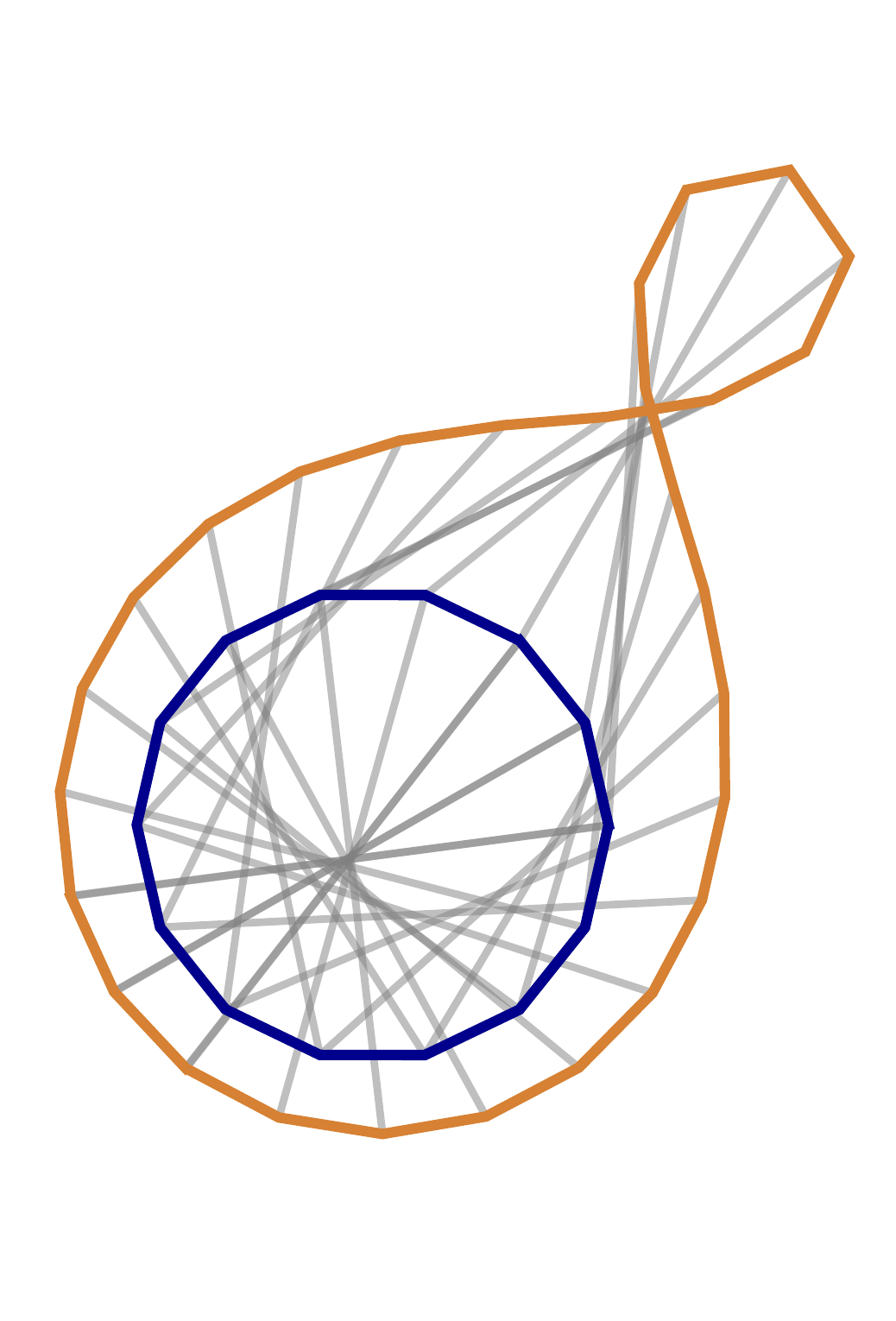}
  \includegraphics[width=0.24\textwidth]{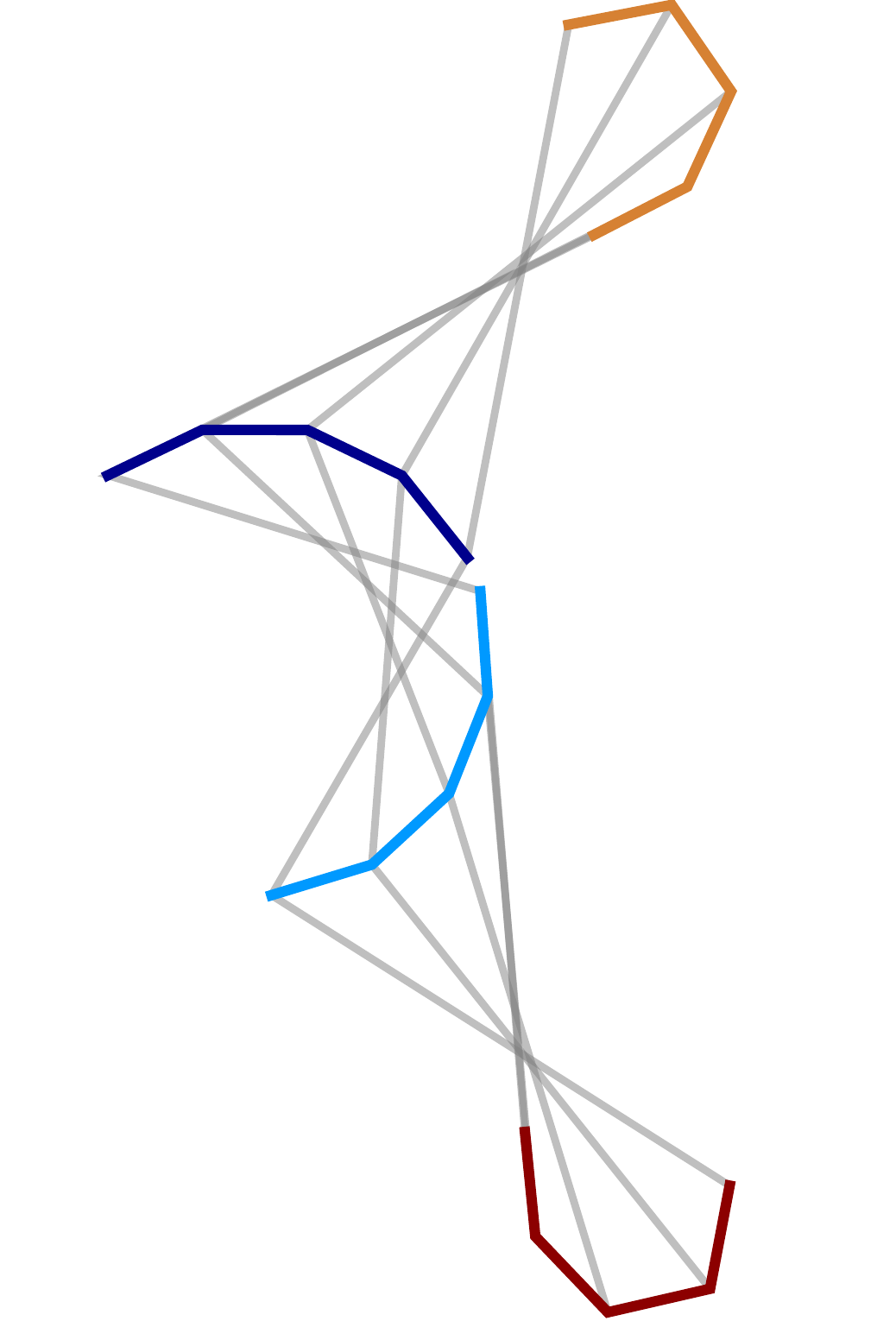}
  \includegraphics[width=0.24\textwidth]{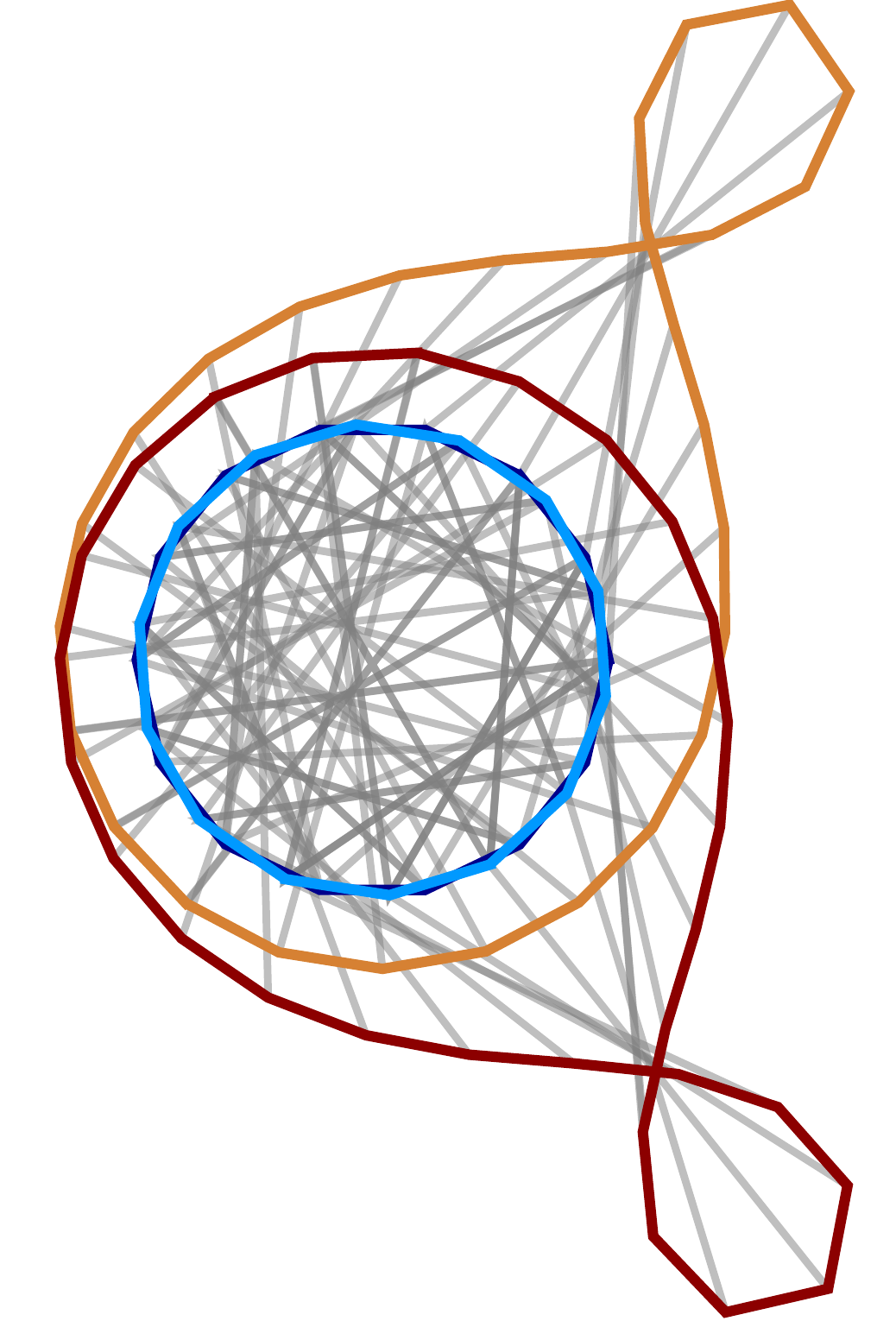}
  \includegraphics[width=0.24\textwidth]{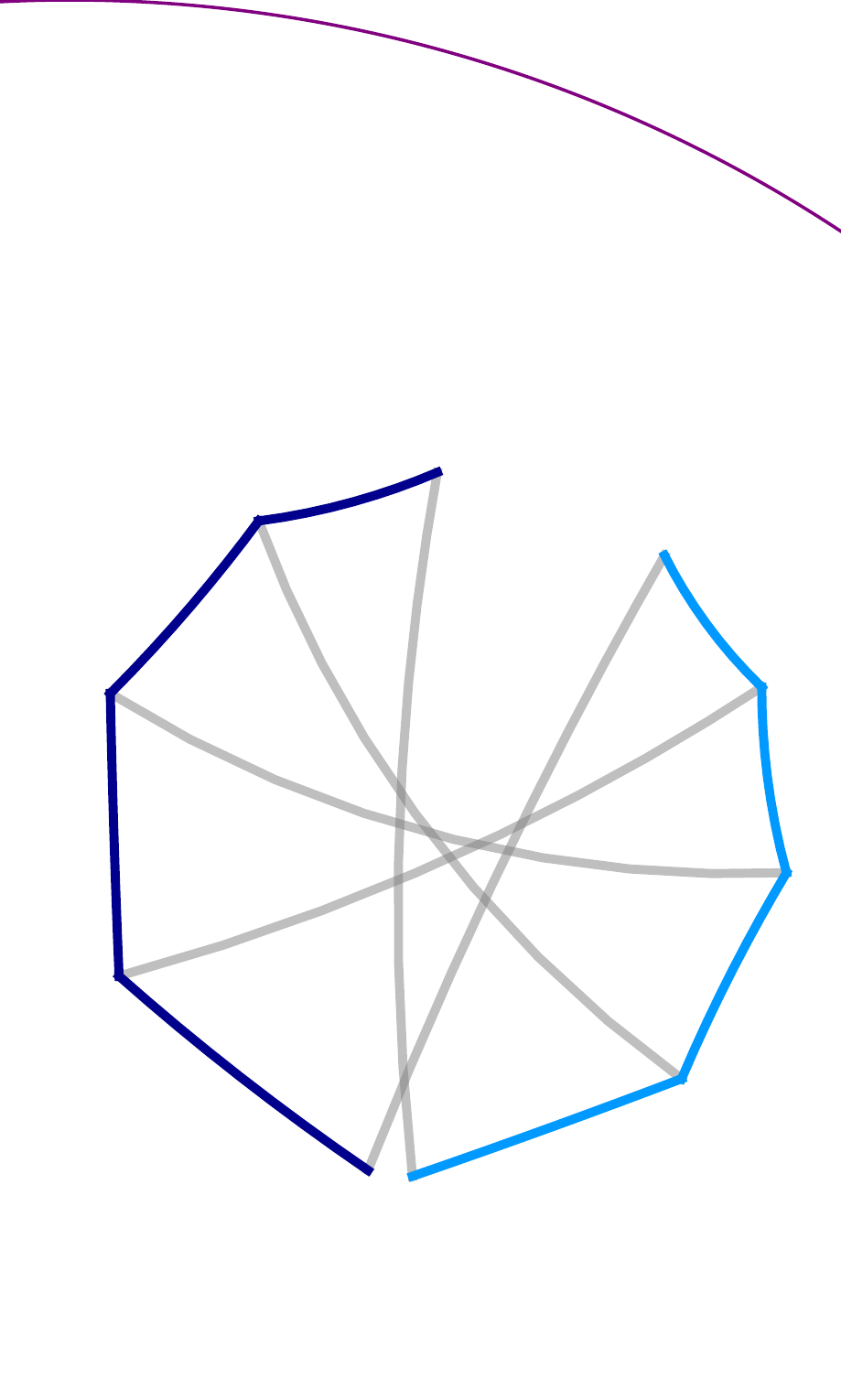}
  \includegraphics[width=0.24\textwidth]{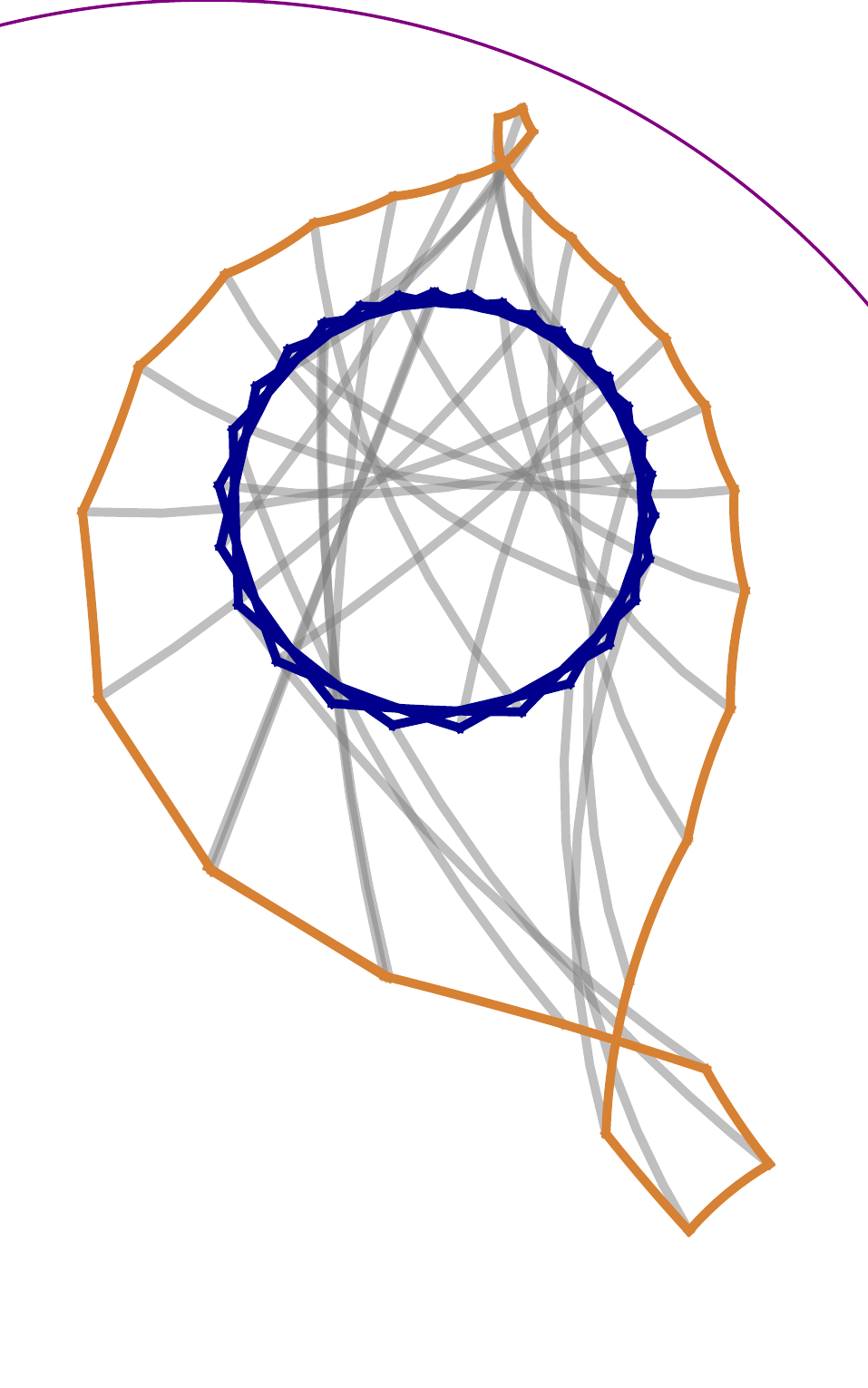}
  \includegraphics[width=0.24\textwidth]{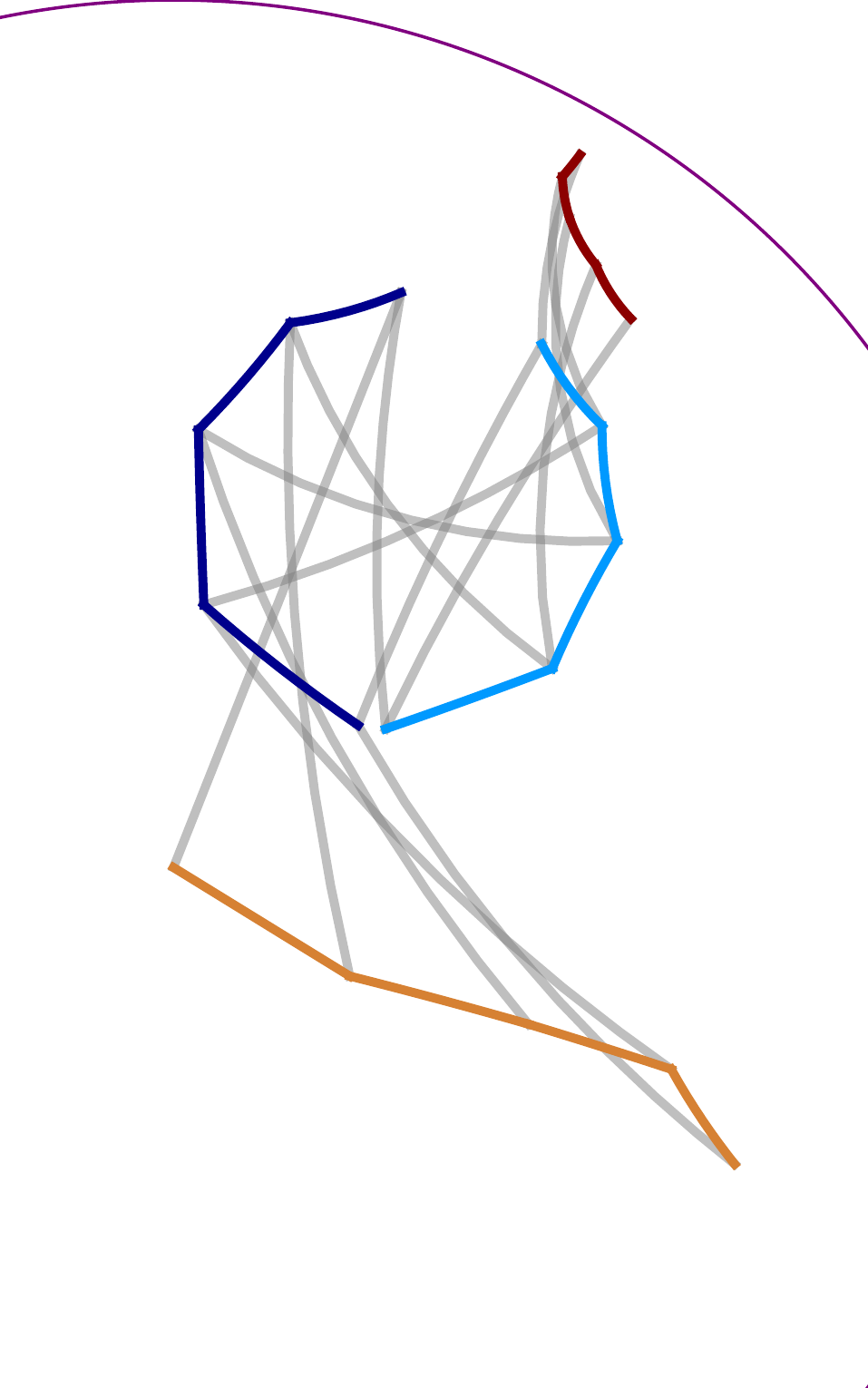}
  \includegraphics[width=0.24\textwidth]{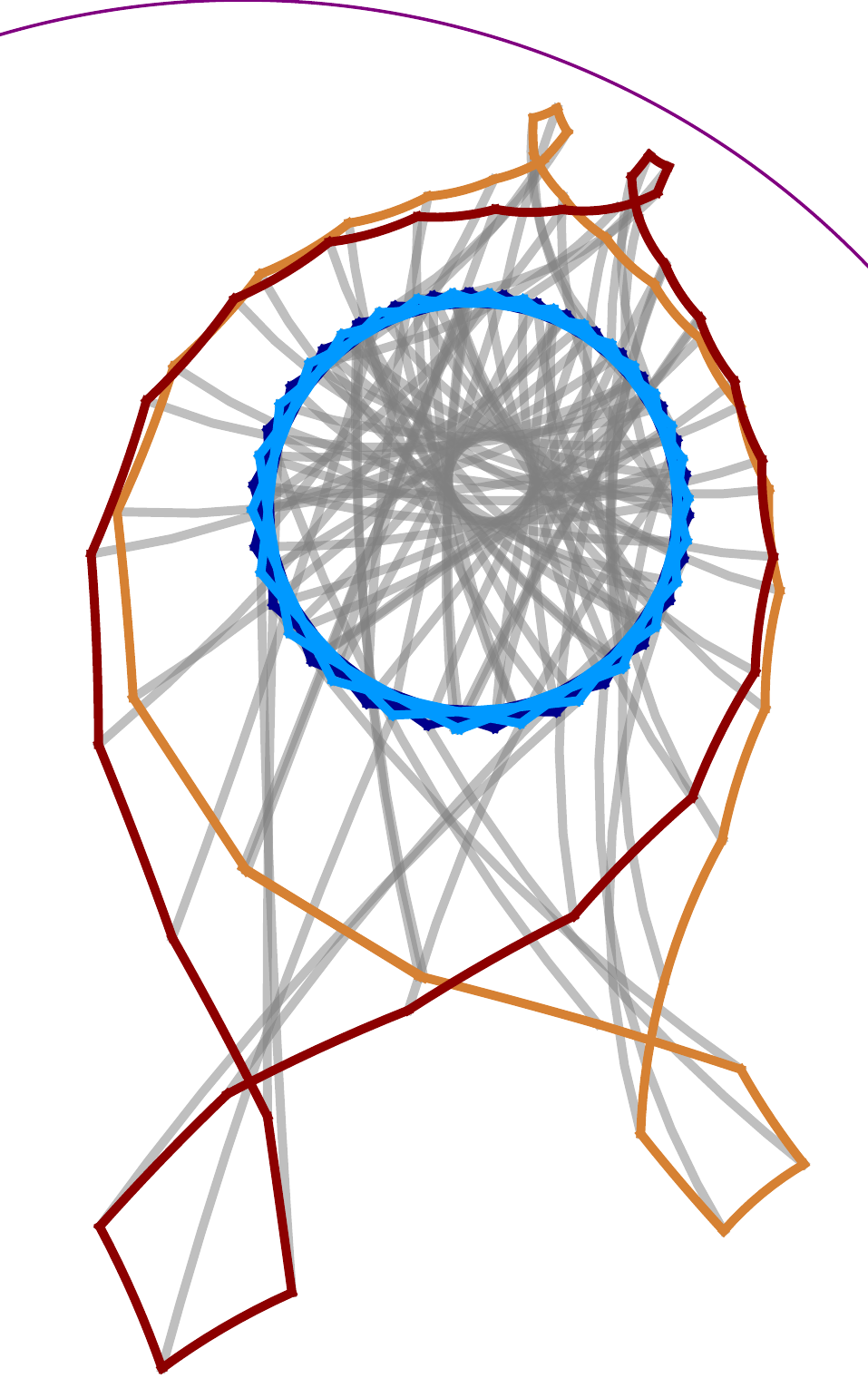}
  \caption{B\"acklund transformations of an equally sampled circle in Euclidean space (top row) and hyperbolic space (bottom row). The figures on the right show a sequence of the transformations seen on the left which causes the orange curve to be $3$-invariant.}
\label{fig:BLcircle}
\end{figure}

We are now interested in a sequence of B\"acklund transformations $f=f^{(0)},f^{(1)},...,f^{(n)}=\tilde{f}$ that preserves the shape of the curve $f$, meaning that the first and last curve $f$ and $\tilde{f}$ are related by an isometry. We express the action of the isometry on the transport matrices by $\tilde u_{01}= E^{-1}u_{01}E$ for an invertible (split-)quaternion $E$. 

Equations \eqref{eq:paradd} and~\eqref{eq:parmult} yield
\begin{align}
\label{eq:compCond}
(\bm1+\lambda v^{(i)}_{1})(\bm1+\lambda u^{(i)}_{01})=(\bm1+\lambda u^{(i+1)}_{01})(\bm1+\lambda v^{(i)}_{0})
\end{align}
for any $\lambda\in\mathbb{C}$ and at every quad. 
Thus, multiplication of such linear factors built from the transport matrices is path-independent. In particular, we have
\begin{align*}
(\bm 1+\lambda v_1^{(n-1)})\cdots(\bm 1+\lambda v^{(0)}_1)(\bm 1+\lambda u_{01})&=(\bm 1+\lambda \tilde u_{01})(\bm 1+\lambda v_0^{(n-1)})\cdots(\bm 1+\lambda v^{(0)}_0)\\
&=E^{-1}(\bm 1+\lambda u_{01})E(\bm 1+\lambda v_0^{(n-1)})\cdots(\bm 1+\lambda v^{(0)}_0)
\end{align*}
and, therefore, the vertex based polynomial
\begin{align}
\label{eq:poly}
\mathcal{P}_0(\lambda):=E(\bm 1+\lambda v_0^{(n-1)})\cdots(\bm 1+\lambda v_0^{(0)})=:C_0^0+\lambda C_0^1+\cdots+\lambda^n C_0^n
\end{align}
fulfills
\begin{align}
\label{eq:polySim}
\mathcal{P}_1(\bm 1+\lambda u_{01})=(\bm 1+\lambda u_{01})\mathcal{P}_0.
\end{align}
Other examples can be obtained by going back and forth (for example, by $f,f^{(1)},f^{(2)},f^{(1)},f$) or by using permutability of B\"acklund transformations (see \cite{pinkallSmokeRingFlow,schief2007chebyshev}). 
All these special sequences can be identified by the fact that the polynomial $\mathcal{P}$ is scalar, i.e., by $\overrightarrow{\mathcal{P}}=0$. On the other hand, we call a sequence of B\"acklund transformations  \emph{regular} if it fullfils $\overrightarrow{\mathcal{P}}\neq0$. 

Finding such a regular sequence of transformations is more complicated and, in fact, it only exists for a special class of curves: 

\begin{defi}
A curve $f\in\C$ is called \emph{$n$-invariant} if there exists a regular sequence of $n$ B\"acklund transformations $f=f^{(0)},...,f^{(n)}=\tilde{f}$ such that the $n$-th transform $\tilde f$ is related to $f$ by an isometry of $\Q$ which is
\begin{enumerate}
\item orientation preserving if $n$ is odd and
\item orientation reversing if $n$ is even.
\end{enumerate}
\end{defi}

A corresponding version of this hierarchy for curves in Euclidean $3$-space is given by \cite[Definition 3.2]{skewParallelogramNets} where the isometry is a screw motion. By restriction to the plane one needs to distinguish the two cases of the axis of the screw motion being contained in the plane or being orthogonal to the plane. For each $n$ only one case yields interesting curves and this depends on the parity of $n$, thus, motivating our condition on the orientation of the isometry. 

Note that by adding two identity transformations to a sequence of B\"acklund transformations we can observe that every $n$-invariant curve is also $(n+2)$-invariant. 
Also, by adding one identity transformation, we can observe that curves invariant after a sequence of $n$ B\"acklund transformations with the isometry having 'wrong' orientation (i.e., reversing for odd $n$, preserving for even $n$) are $(n+1)$-invariant.

One can interpret the $n$ transformations as fundamental piece of a (bi-)infinite sequence of B\"acklund transformations with period $n$. 
In fact, one could analogously define $n$-invariant curves via the existence of such a periodic infinite sequence. This observation also shows that the intermediate curves $f^{(1)},...,f^{(n-1)}$ are $n$-invariant as well. 
Note that in this way we can interpret the $n$ B\"acklund transformations as step of a discrete flow (as in \cite{hoffmannSmokeRingFlow,pinkallSmokeRingFlow} for $n=2$) and the $n$-invariant curves as the invariant curves of this flow.

\begin{ex}
\label{ex:circle1inv}
A B\"acklund transform of an equally sampled circle with initial point on the circle (as seen in Figure~\ref{fig:BLcircle} on the left) is a rotated version of the original curve and, as such, $1$-invariant. Such discrete circles (and lines) are the only $1$-invariant curves.
\end{ex}

\begin{ex}
A B\"acklund transform of an equally sampled straight line yields a discrete Euler loop as seen in Figure~\ref{fig:eulerloop}. Reflecting this construction along the line yields a second Euler loop which is related to the first Euler loop by two B\"acklund transformations and by reflection. As such, the Euler loop is $2$-invariant.
\end{ex}

\begin{prop}
A B\"acklund transform $\hat f$ of an $n$-invariant curve $f$ is $(n+2)$-invariant.
\end{prop}

\begin{proof}
There is a curve $\tilde{f}$ related to $f$ by a regular sequence of $n$ B\"acklund transformations and by an isometry $\varphi$ (i.e., $\tilde{f}=\psi(f)$). Then, application of the isometry to the curve $\hat f$ yields a curve~$\psi(\hat f)$ which is related to $\hat f$ by the regular sequence of $(n+2)$ B\"acklund transformations $\hat{f},f=f^{(0)},...,f^{(n)}=\tilde{f}=\psi(f),\psi(\hat f)$.  
\end{proof}

\begin{ex}
A B\"acklund transform of an equally sampled circle with an arbitrary initial point is $3$-invariant as seen in Figure \ref{fig:BLcircle} on the right. This includes closed Darboux transforms of discrete circles as studied in~\cite{CHO2023102065}.
\end{ex}

\begin{figure}[h!]
  \centering
  \includegraphics[width=0.34\textwidth]{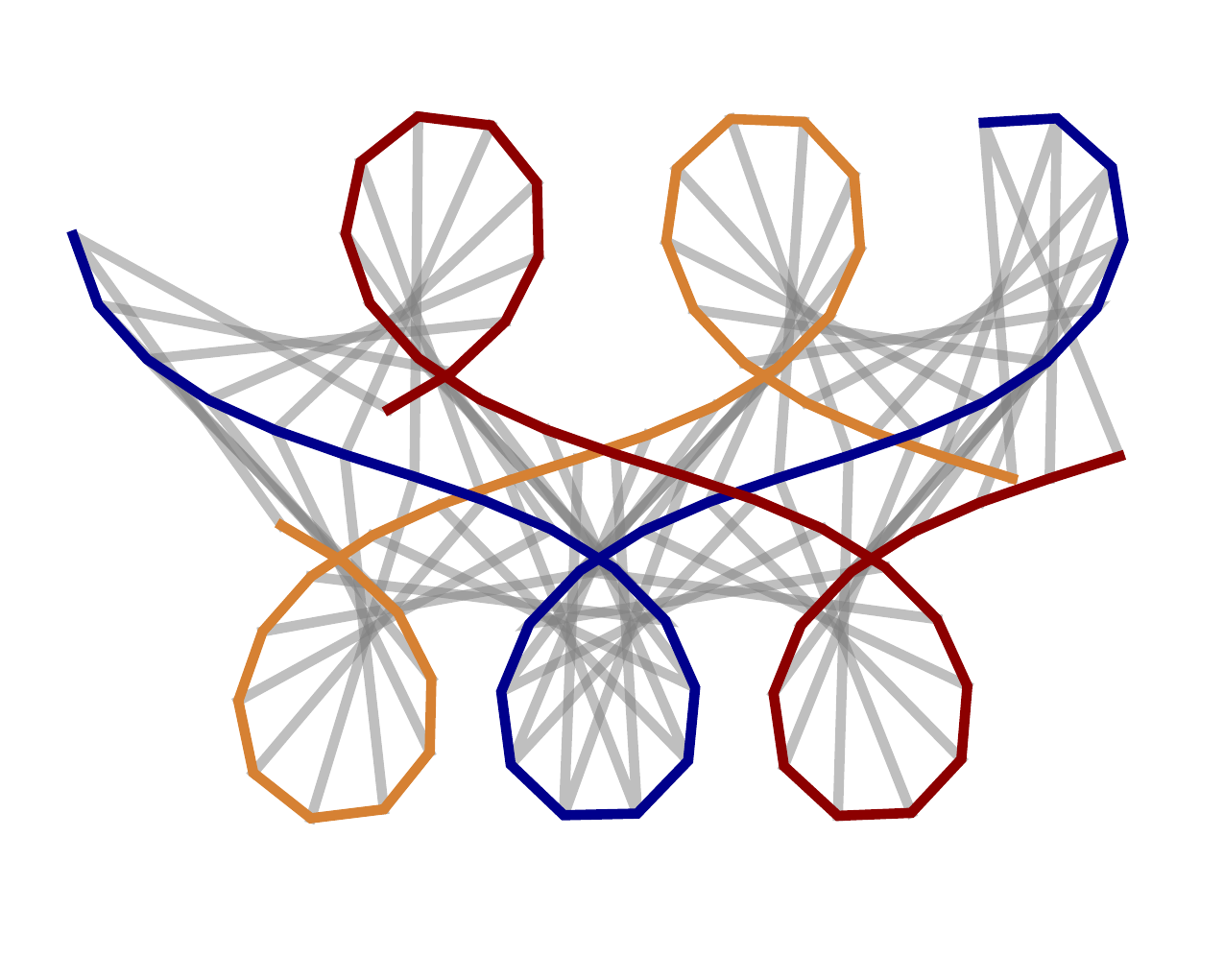}
  \includegraphics[width=0.34\textwidth]{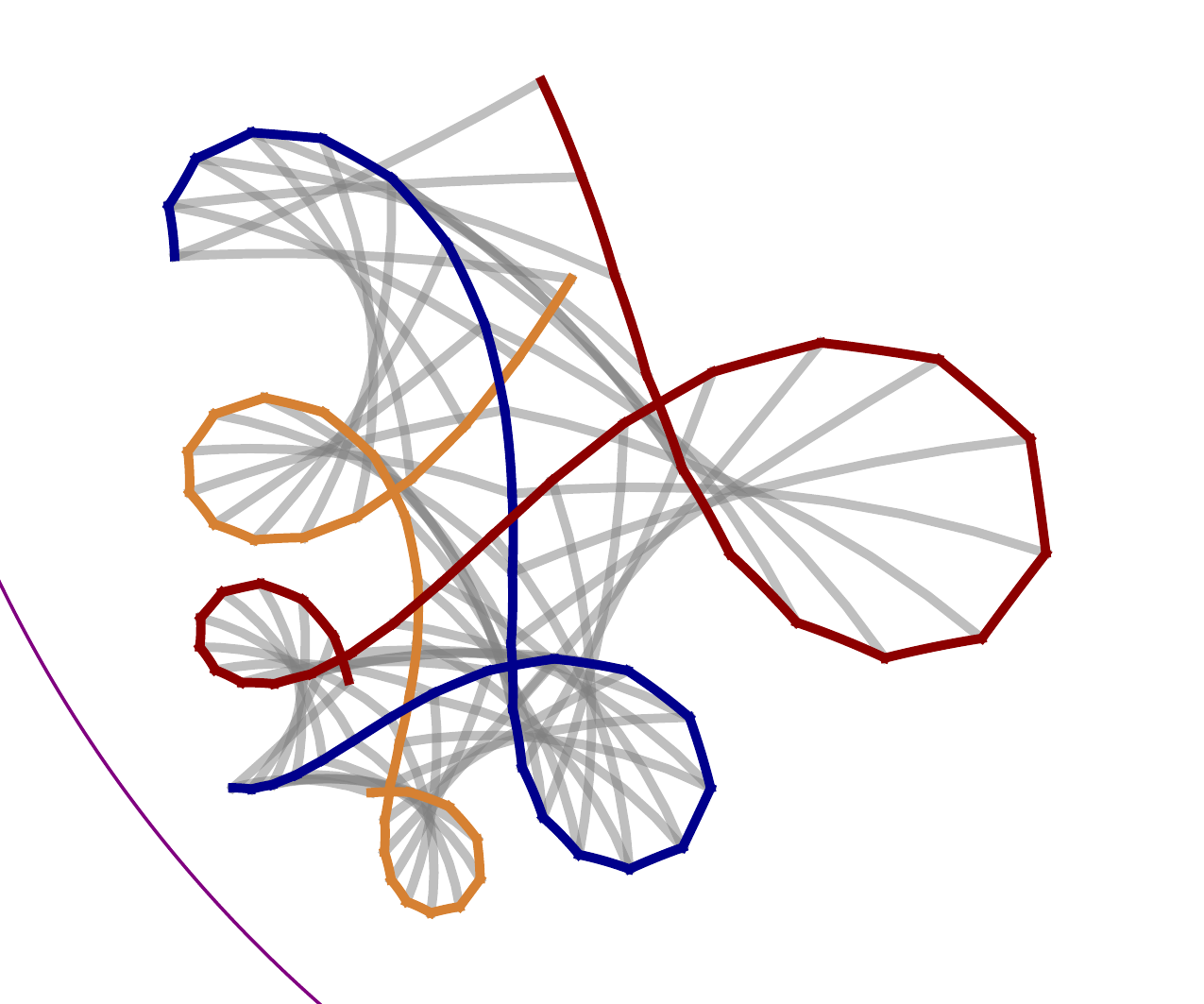}
  \includegraphics[width=0.28\textwidth]{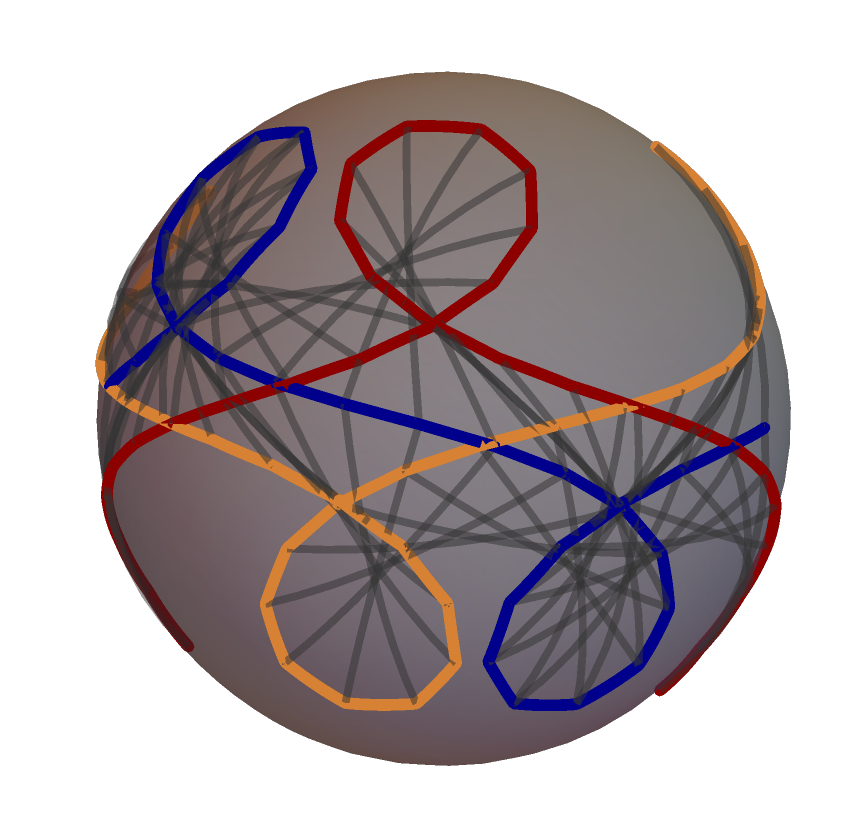}
  \includegraphics[width=0.34\textwidth]{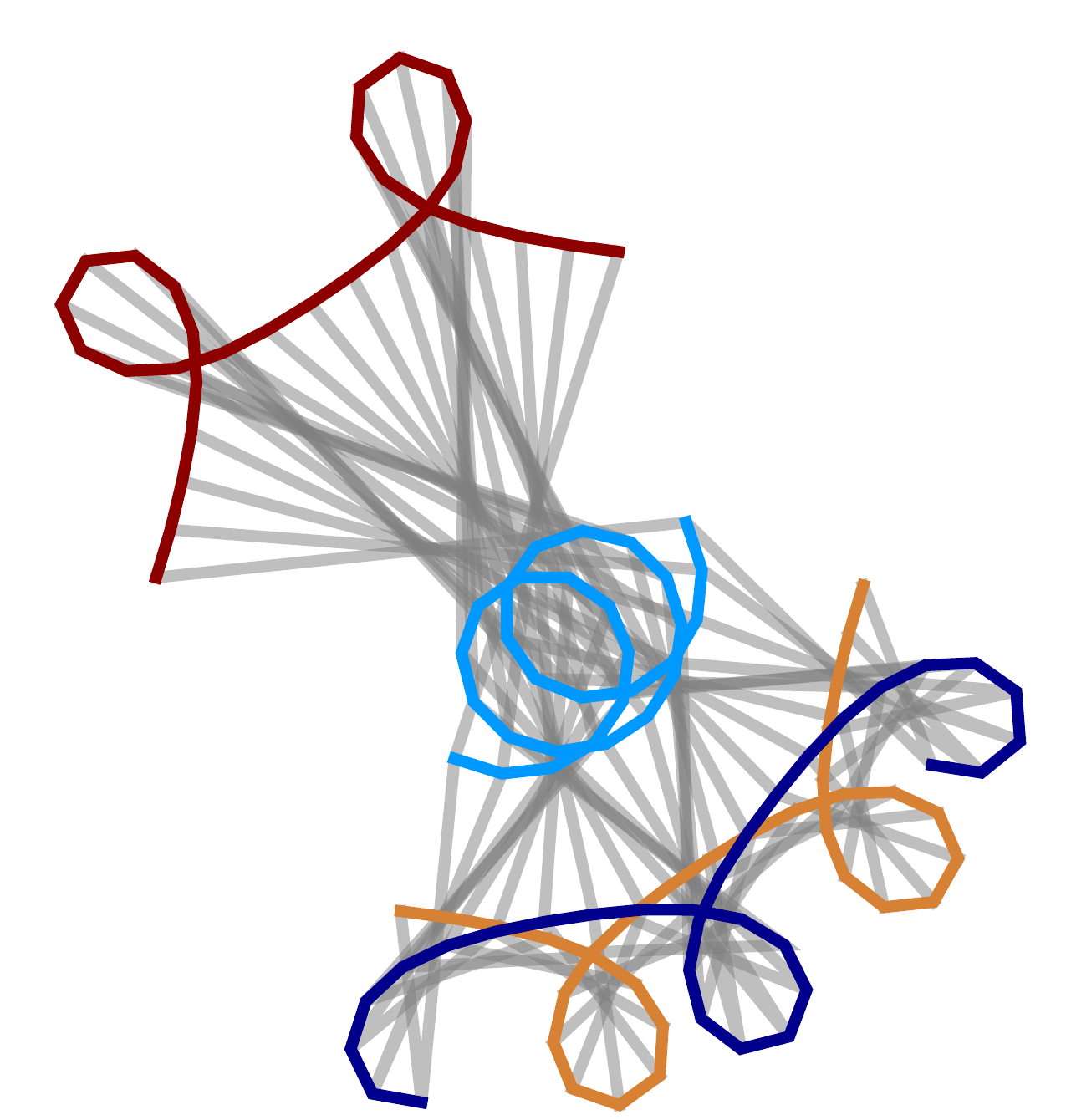}
  \includegraphics[width=0.34\textwidth]{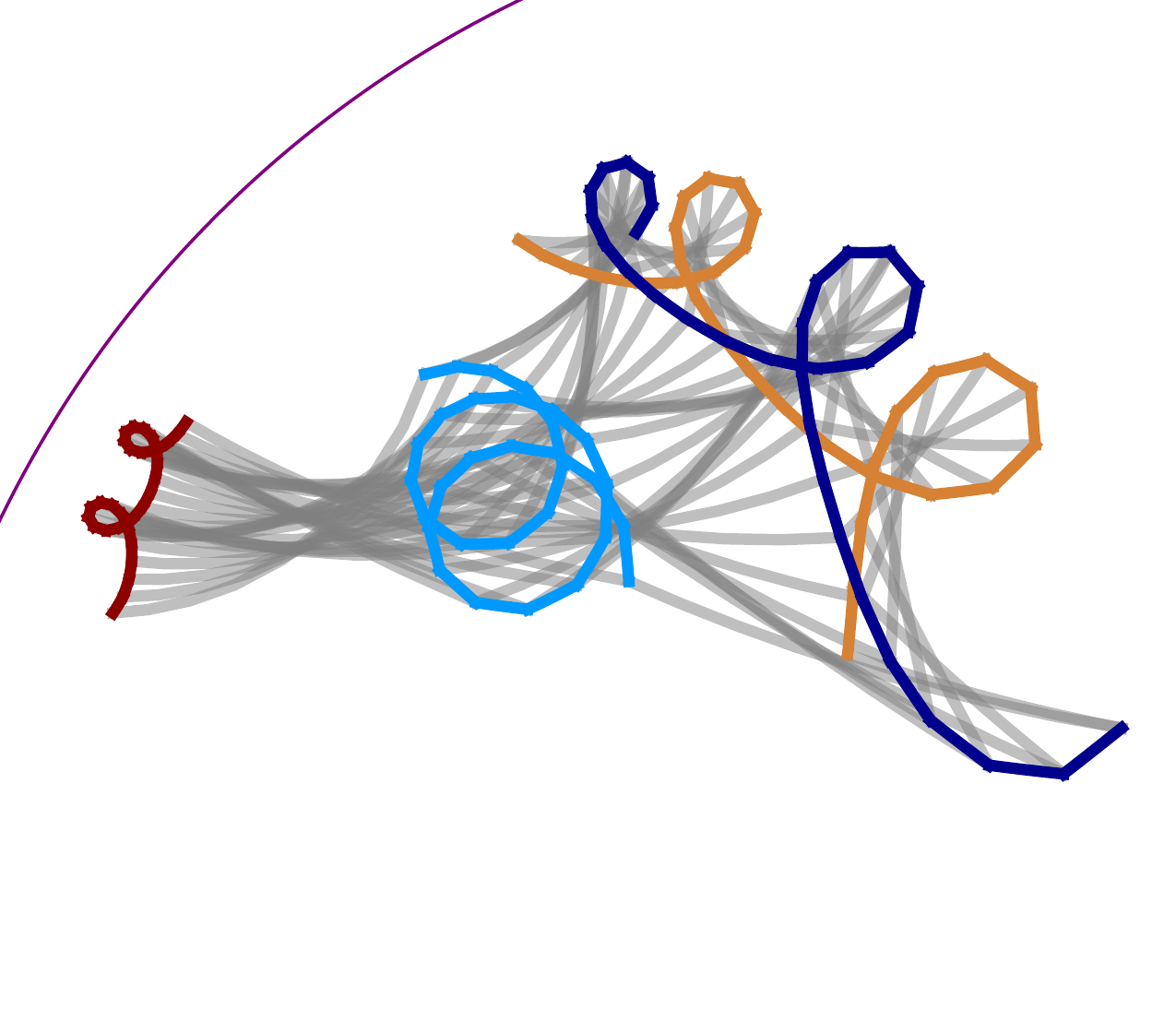}
  \includegraphics[width=0.28\textwidth]{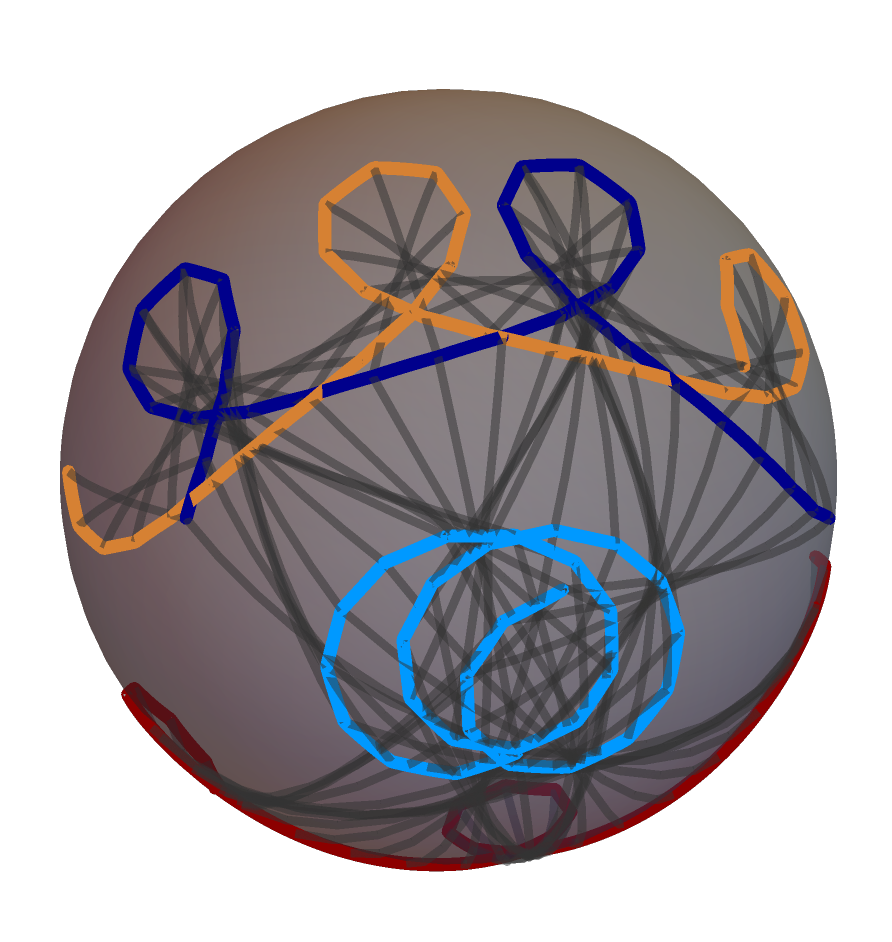}
  \caption{Elastic and area-constrained elastic curves 
  in each space form and their corresponding B\"acklund transformations. The first and last curve (orange and red) are isometric in the respective space form.}
\label{fig:nInvExamples}
\end{figure}

We are prepared to state the main theorem of this section:

\begin{thm}
\label{thm:invariantCurvesConstElastica}
A curve $f\in\C$ is $2$-invariant if and only if it is a discrete elastic curve in $\Q$. It is $3$-invariant if and only if it is a discrete constrained elastic curve in $\Q$.
\end{thm}

In particular, as demonstrated in Figure \ref{fig:eight23inv}, elastic curves are both $2$- and $3$-invariant, however, the relation between the corresponding sequences of B\"acklund transformations is not obvious. 

\begin{figure}[h!]
  \centering
  \includegraphics[width=0.2\textwidth]{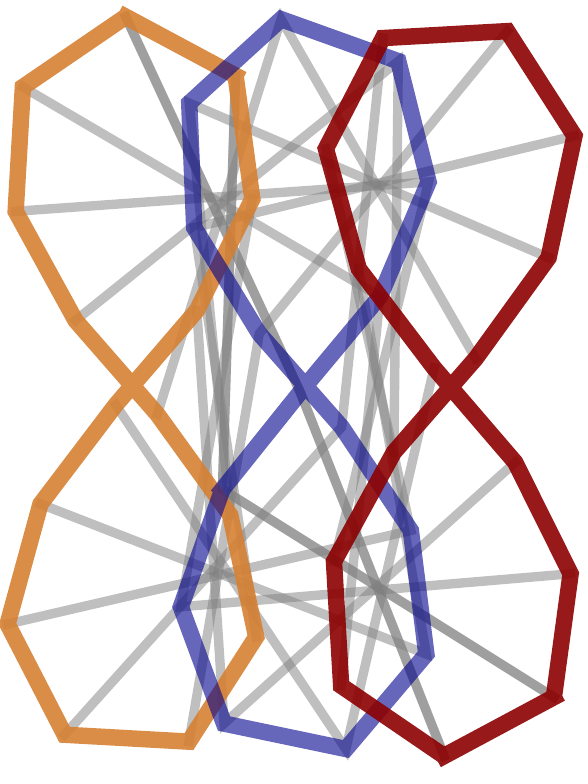}
  \hspace{1.5cm}
  \includegraphics[width=0.505\textwidth]{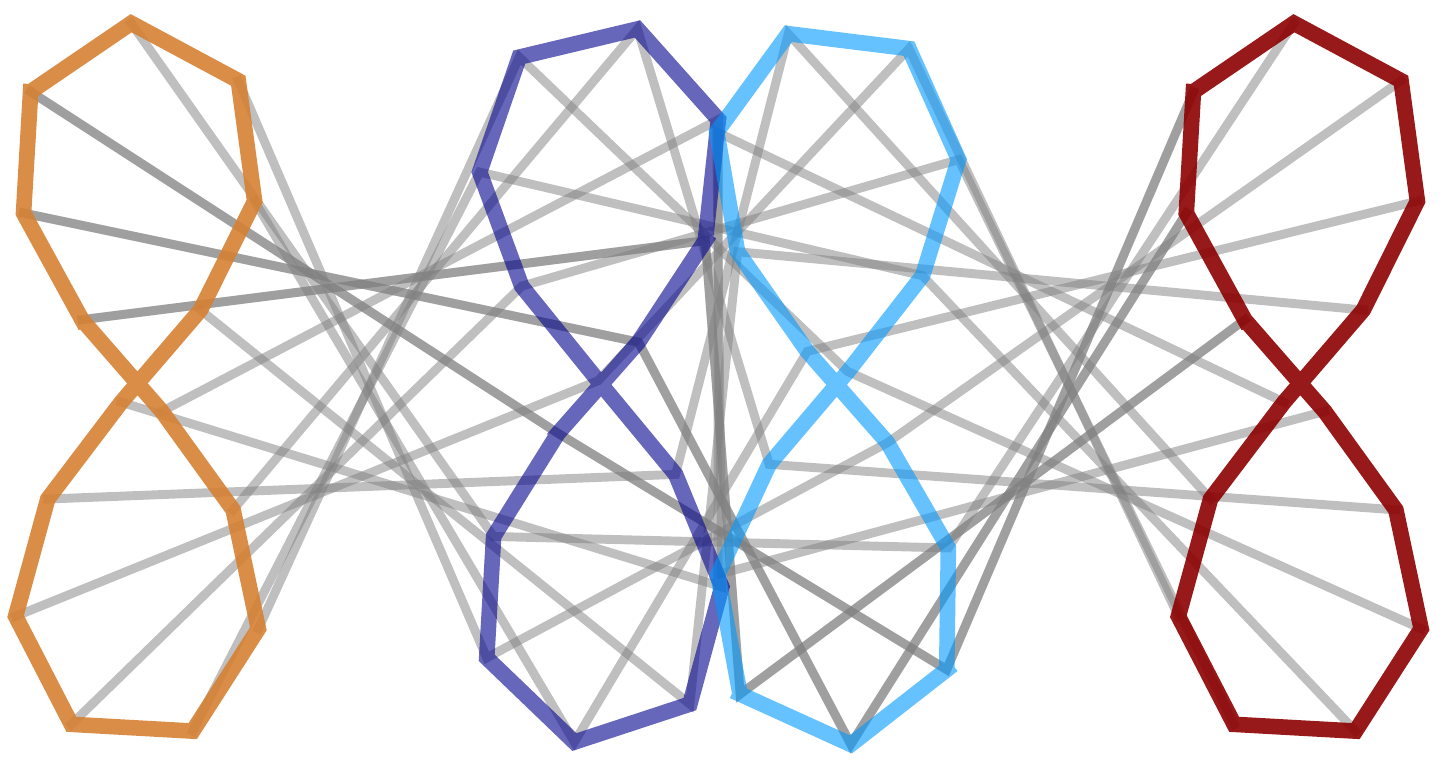}
  \caption{An elastic curve is $2$-invariant (left) and $3$-invariant (right). Observe that the last curve of the sequence (red) is obtained from the first curve (orange) by reflection and translation on the left but only by translation on the right.}
\label{fig:eight23inv}
\end{figure}

In Section \ref{sec:eucFlows}, we will prove this theorem for curves in Euclidean space. In Section \ref{sec:nonEucFlows}, we will extend the result to space forms using the associated family.
%%%%%%%%%%%%%%%%%%%%%%%%%%%%%%%%%%%%
%
%
%%%%%%%%%%%%%%%%%%%%%%%%%%%%%%%%%%%%
\subsection{Euclidean $n$-invariant curves}
\label{sec:eucFlows}

We will now consider $n$-invariant curves $f\in\CE(\EucDelta)$. First, we discuss how the coefficients $C^j$ of the polynomial $\mathcal{P}$ given by \eqref{eq:poly} contain invariants along the curve and how they are related to the curvature and tangential vector of the curve. 

Note that in the case of an identity transformation the transport matrix $v^{(i)}_0$ vanishes and the corresponding linear factor becomes trivial. Thus, the polynomial $\mathcal{P}$ can have degree smaller than $n$ and some of the higher coefficients can vanish. 

Now, each coefficient $C^j$ is a product of $E$ with $j$ transport matrices $v^{(i)}$. Recall that $v\in\spannR{\bm i,\bm j}$ and $E\in \spannR{\bm i,\bm j}$ for even $n$ and $E\in\spannR{\bm1,\bm k}$ for odd $n$ and, therefore, we have 
\begin{align}
 C^n,C^{n-2},...\in\spannR{\bm i,\bm j}   \quad\text{and}\quad
 C^{n-1},C^{n-3},...\in\spannR{\bm 1,\bm k}.
\end{align} 
From \eqref{eq:polySim} we see that the real polynomials $\tr\mathcal{P}$ and $\det\mathcal{P}$ remain constant along the curve and their coefficients are invariants. The coefficients of $\tr \mathcal{P}$ are given by $r_j:=\frac12\tr C^j$ and vanish for $j=n,n-2,...$. Also, we have
\begin{align*}
\det\mathcal{P}=\det E(1+\lambda^2\det v^{(n-1)})\,\cdots\,(1+\lambda^2\det v^{(0)})
\end{align*}
and, hence, all coefficients of $\det\mathcal{P}$ are given by $\det E$ and $\det v^{(i)}$ which we already knew to be constant. 
For later use, recall that non-vanishing diagonals of the Darboux butterflies implies $\det v^{(i)}\neq\EucDelta^2$ and, therefore, we have $\det\mathcal{P}(\frac{i}{\EucDelta})\neq 0$. %since neighboring edge lengths are distinct. 

Invariants contain information about the curve and the B\"acklund transformations. The coefficients of the invariant polynomial $\det\overrightarrow{\mathcal{P}}=\det\mathcal{P}-(\tr\mathcal{P})^2=-(\overrightarrow{\mathcal{P}})^2$ will turn out to only depend on the curve. The $\lambda^j$ coefficient is given by
\begin{align}
\label{eq:theta}
\ninv_j:=\sum_{k=\max\{0,j-n\}}^{\min\{j,n\}} \innerM{\vec C^k,\vec C^{j-k}}.
\end{align}
We have that $\vec C^k\perp \vec C^l$ if $k$ and $l$ have different parity and, hence, $\theta_j=0$ holds for odd $j$. For $n\leq3$ the invariants $\ninv_j$ are given by the following table: 
\\\begin{center}
    \begin{tabular}{c | c c c}
    & $n=1$ & $n=2$ & $n=3$ \\
    $\ninv_0$ & $\det \vec C^0$ & $\det C^0$ & $\det \vec C^0$ \\
%    $\mu^1$ & 0 & 0 & 0 \\
    $\ninv_2$ & $\det C^1$ & $\det\vec C^1+2\innerM{C^0,C^2}$ & $\det C^1+2\innerM{\vec C^0,\vec C^2}$ \\
%    $\mu^3$ &  & 0 & 0 \\
    $\ninv_4$ &  & $\det C^2$ & $\det \vec C^2+2\innerM{C^1,C^3}$ \\
%    $\mu^5$ &  &  & 0 \\
    $\ninv_6$ &  &  & $\det C^3$ \\
\end{tabular}
\end{center}
\ \\The following lemma shows how the polynomial contains information about the curve.

\begin{lem}
\label{lem:fixedPointRelations}
Let $f\in\CE(\EucDelta)$ be $n$-invariant with polynomial $\mathcal{P}$ given by \eqref{eq:poly}. Then, the vertex based quaternions
\begin{align*}
A&:=C^n-\EucDelta^2C^{n-2}+\EucDelta^4C^{n-4}-...\in \spannR{\bm i,\bm j}\\
B&:=C^{n-1}-\EucDelta^2C^{n-3}+\EucDelta^4C^{n-5}-...\in \spannR{\bm 1,\bm k}
\end{align*}
and an invariant $\beta\in\mathbb{R}$ given by $\beta^2=-(-\EucDelta^2)^{n-1}\det\overrightarrow{\mathcal{P}}(\frac{i}{\EucDelta})\neq0$ have the following relations to the curve:
\begin{enumerate}
\item $A$ is related to the tangential vector via $A=\beta T$ and
\item $\vec B$ encodes the curvature via $\beta-\vec B=\beta H$ and, thus, $\vec B=-\kappa\frac{\beta\EucDelta}{2}\bm k$.
\end{enumerate}
\end{lem}

\begin{proof}
Expanding \eqref{eq:polySim} in powers of $\lambda$ gives us equations for the coefficients $C^i$:
\begin{align}
\label{eq:polyEvol}
\lambda^0:\quad &C^0_1=C^0_0\,(=E)\nonumber\\
\lambda^i:\quad &C_1^i+C_1^{i-1}u_{01}=C^i_0+u_{01}C^{i-1}_0,\qquad 1\leq i\leq n\\
\lambda^{n+1}:\quad &C^n_1u_{01}=u_{01}C^n_0\nonumber
\end{align}
The equations for $0\leq i\leq n$ already determine $\mathcal{P}_1$ from $\mathcal{P}_0$ and $u_{01}$. Plugging these into the last equation provides a condition relating $\mathcal P_0$ and $u_{01}$:
\begin{align*}
u_{01}C^n_0&=C^n_1u_{01}=C^n_0u_{01}+u_{01}C^{n-1}_0u_{01}-C^{n-1}_1u_{01}^2=\cdots\\
&=C^n_0u_{01}+u_{01}C^{n-1}_0u_{01}-C^{n-1}_0u_{01}^2-u_{01}C_0^{n-2}u_{01}^2+C_0^{n-2}u_{01}^3+\cdots
\end{align*}
With $u_{01}^2=-\EucDelta^2$ this can be rearranged to
\begin{align*}
u_{01}(C^n_0-C^{n-1}_0u_{01}-\EucDelta^2C^{n-2}_0+\cdots)=(C^n_0-C^{n-1}_0u_{01}-\EucDelta^2C^{n-2}_0+\cdots)u_{01}
\end{align*}
which reads
\begin{align*}
u_{01}(A_0-B_0u_{01})=(A_0-B_0u_{01})u_{01}.
\end{align*}
Two imaginary quaternions commute if and only if they are multiples of each other and, hence, there exists $\beta\in\mathbb{R}$ with $A_0-B_0u_{01}=(\beta-\frac12\tr B_0) u_{01}$ which is equivalent to $A_0=(\beta+\vec B_0)u_{01}$. 
We can also obtain $A_0=-(\hat\beta+\vec B_0)u_{\bar 10}$ for some $\hat\beta\in\mathbb{R}$ since $-u_{\bar10}$ must fulfil the same equations as $u_{01}$ due to the symmetry of the construction. 
Taking the determinant on both sides of these equations and using $(-i\EucDelta)^n\mathcal{P}(\frac{i}{\EucDelta})=A-i\EucDelta B$ yields
\begin{align*}
\EucDelta^2\hat\beta^2=\EucDelta^2\beta^2=\det A-\EucDelta^2\det\vec B=\det(A-i\EucDelta\vec B)=(-\EucDelta^2)^n\det\overrightarrow{\mathcal{P}}(\frac{i}{\EucDelta}).
\end{align*}
Since the curve is regular $-u_{\bar10}$ and $u_{01}$ are distinct and, thus, we have $\hat\beta=-\beta\neq 0$ and $A_0=(\beta-\vec B_0)u_{\bar10}=u_{\bar10}(\beta+\vec B_0)$. 
Note that a quaternion $X\in\spannR{\bm 1,\bm k}$ is determined up to scaling by the equation $u_{01}=X^{-1}u_{\bar10}X$. Both quaternions $\beta+\vec B_0$ and $H_0^*=\bm1-\vec H_0$ fulfill this equation and, hence, %claim (2) follows. 
$\beta+\vec B_0=\beta H_0^*$ and $A_0=\beta H_0^*u_{01}=\beta T_0$ show both statements.

Finally, we prove that $\beta$ is constant (not only up to sign). At neighboring vertices we write
\begin{align*}
A_0+u_{01}B_0&=(\beta_0+\vec B_0)u_{01}+u_{01}B_0=(\beta_0+\frac{1}{2}\tr B_0)u_{01},\\
A_1+B_1u_{01}&=u_{01}(\beta_1+\vec B_1)+B_1u_{01}=(\beta_1+\frac{1}{2}\tr B_0)u_{01}.
\end{align*}
Both of these terms agree which follows from \eqref{eq:polySim} with $\lambda=\frac{i}{\EucDelta}$ which reads
\begin{align*}
(A_1-i\EucDelta B_1)(\bm 1+\frac{i}{\EucDelta}u_{01})=(\bm 1+\frac{i}{\EucDelta}u_{01})(A_0-i\EucDelta B_0).
\end{align*}
Thus, we have $\beta_1=\beta_0$.
\end{proof}

\begin{rem}
There is one special case which, for brevity, is not mentioned in the Lemma: If $A_0=\vec B_0=0$ or, equivalently, $\overrightarrow{\mathcal{P}}_0(\frac{i}{\EucDelta})=0$ at some point $f_0$ there is no condition on $u$. In this case we have $\det\overrightarrow{\mathcal{P}}(\frac{i}{\EucDelta})=0$ and since this is an invariant along the curve we have $\overrightarrow{\mathcal{P}}(\frac{i}{\EucDelta})=0$ at every point. Therefore, $1+\lambda^2\EucDelta^2$ is a factor of $\overrightarrow{\mathcal{P}}$ and the unique polynomial $\mathcal{Q}$ defined by $\overrightarrow{\mathcal{P}}=(1+\lambda^2\EucDelta^2)^k\mathcal{Q}$ with $\mathcal{Q}(\frac{i}{\EucDelta})\neq0$ still fulfils \eqref{eq:polySim}. Thus, the statement of the Lemma still holds true if we replace $\mathcal{P}$ with the reduced polynomial $\mathcal{Q}$. Factoring of such real polynomials until the polynomial has no divisors as in Proposition \ref{prop:izosimovFact} (2b) also allows us to find the minimal $n$ such that the curve is $n$-invariant.
\end{rem}

\begin{ex}
    For a $1$-invariant curve, we have that $B=C^0=E$ is constant. Hence, its curvature is constant and the curve is a discrete circle or a line (cf.\,Example \ref{ex:circle1inv}). Specifically, with normalized~$\vec E=\bm k$ the center of the circle turns out to be the point $X=F-\frac{1}{2}\bm k C^1$ (which can be shown to be constant, see \eqref{eq:centerConst}) and its squared radius is given by
\begin{align*}
\|F-X\|^2&=\det(F-X)=\frac{1}{4}\det C^1=\frac{1}{4}\ninv^2.
\end{align*}
\end{ex}

The evolution formula \eqref{eq:polyEvol} allows us to make some initial observations about the coefficients~$C^i$. Clearly, $C^0=E$ is constant and contains information about the isometry. As exemplified in Figure~\ref{fig:flowVF}, we can interpret the coefficients as flow vector fields. For example, the coefficient $C^1$ evolves by $C^1_1-C^1_0=u_{01}E-Eu_{01}=2u_{01}\times\vec E$. For odd $n$ it describes the Euclidean motion induced by the isometry. This interpretation extends to even $n$ if we consider our Euclidean plane as plane in Euclidean $3$-space $\mathbb{R}^3\cong\spannR{\bm i,\bm j,\bm k}$: We have $\vec{C}^1\parallel\bm k$ and the flow leaves the plane. 
As a second example, observe that the vector field $\bm k C^3\in\spannR{\bm i,\bm j}$ evolves by $\bm k C^3_1=-u_{01}\bm k C^3_0u_{01}^{-1}$ which is a reflection along the perpendicular bisector of each edge. As such, it is a parallel normal field to the curve and, in particular, has constant length.

\begin{figure}[h!]
  \centering
  \includegraphics[width=0.36\textwidth]{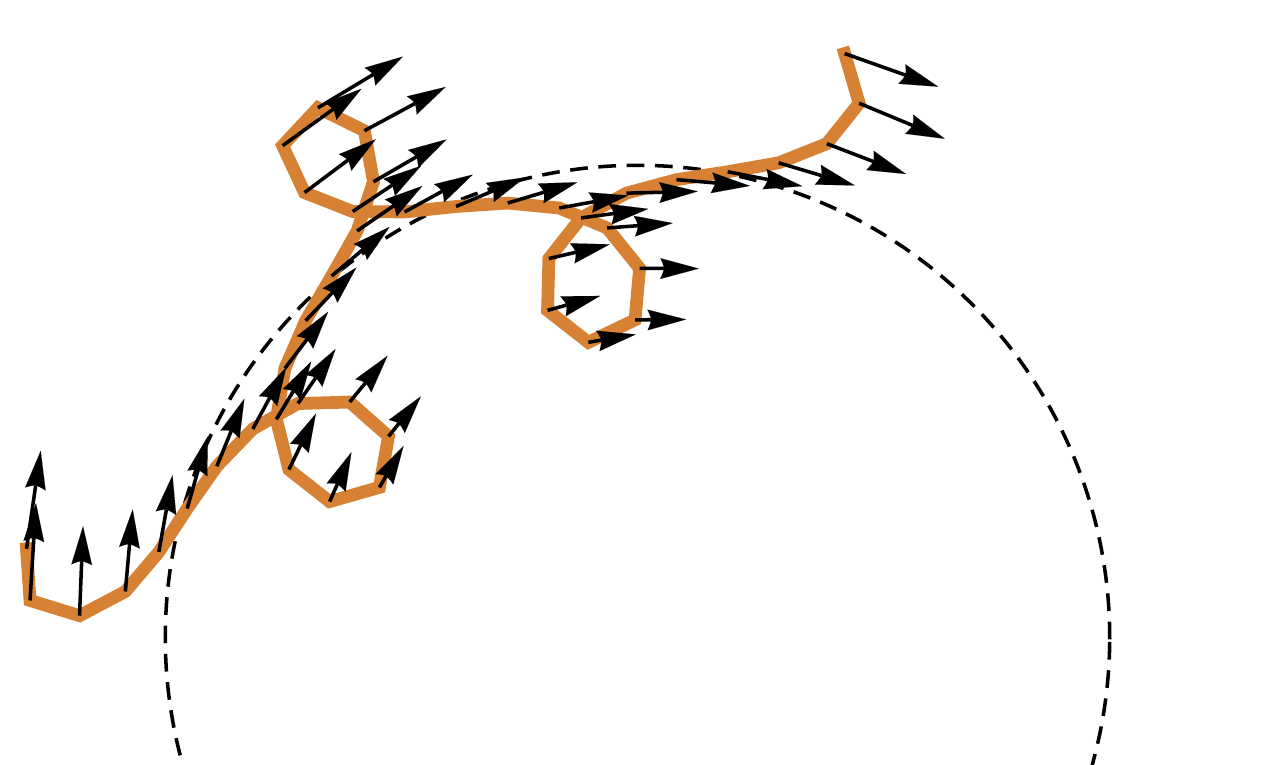}
  \includegraphics[width=0.25\textwidth]{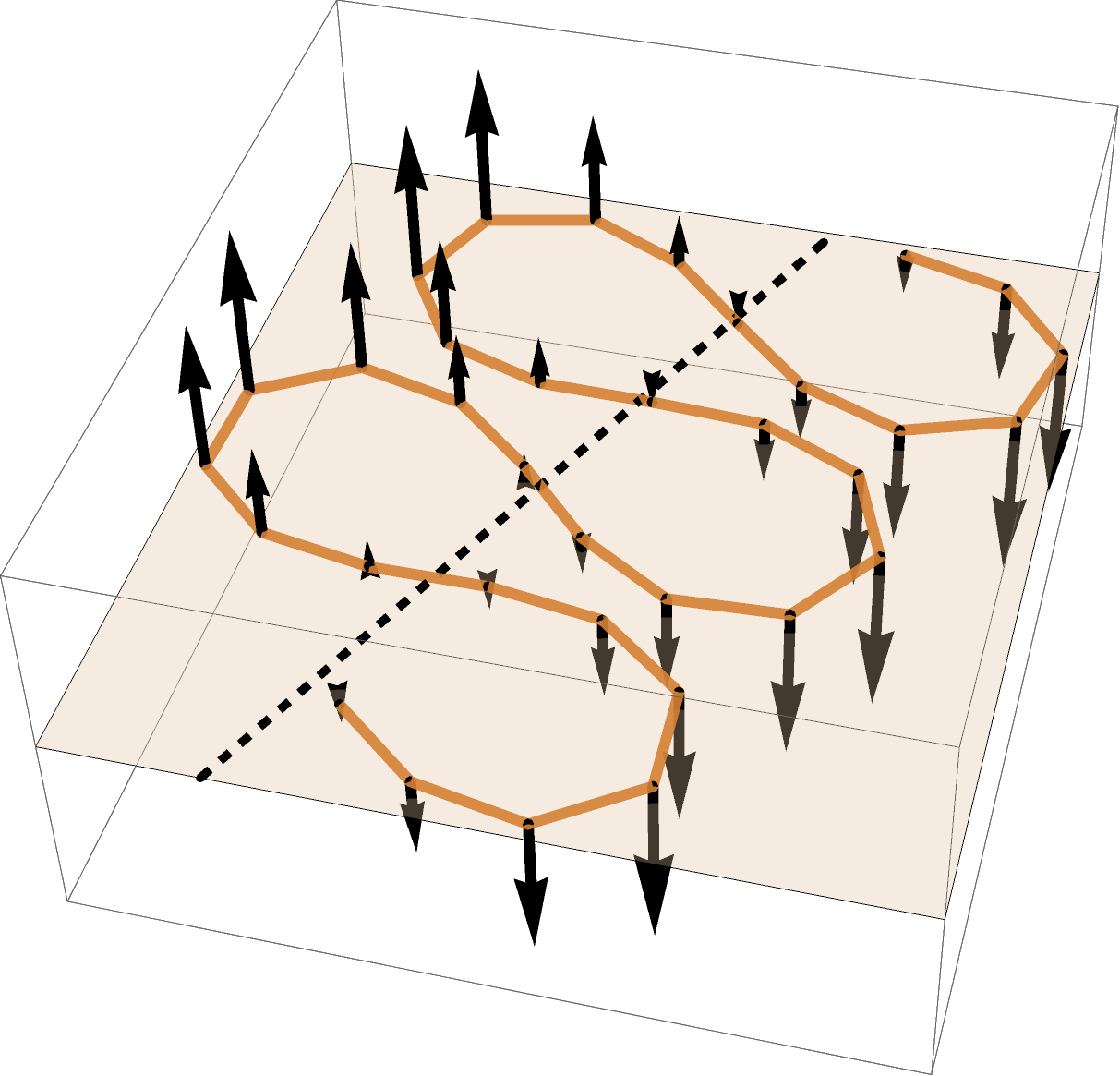}
  \includegraphics[width=0.36\textwidth]{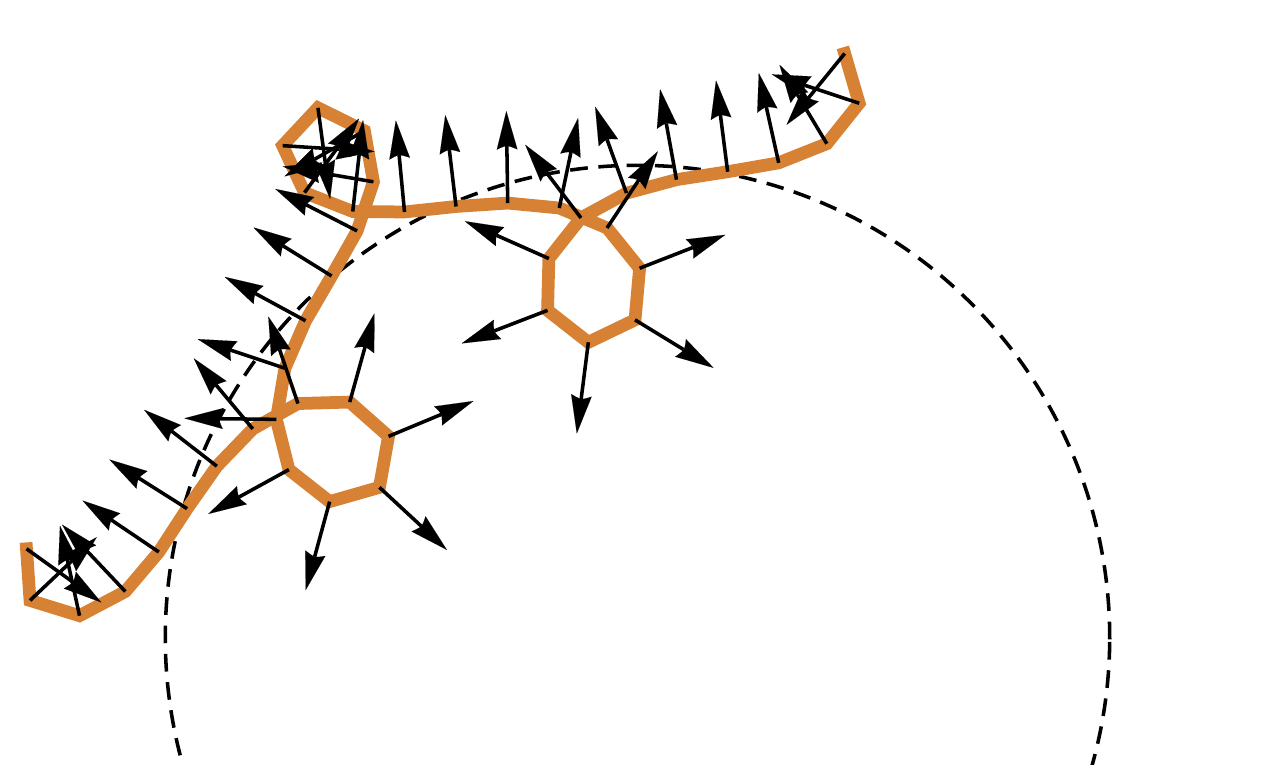}
  \caption{Flow vector fields for $n$-invariant curves: 
  The vector field $\vec C^1$ induces a rotation which stays inside the plane for odd $n$ (left) and leaves the plane for even $n$ (middle). The vector field~$\bm k C^3$ constitutes an edge-parallel normal map (right). 
  }
  \label{fig:flowVF}
\end{figure}

\begin{rem}
Using Lemma \ref{lem:fixedPointRelations} we can construct an $n$-invariant curve iteratively from $E\in\mathbb{H}\setminus\{0\}$ and initial points $f(t_0)=f^{(0)}(t_0),...,f^{(n)}(t_0)$ in transformation direction. Indeed, this initial data determines an initial polynomial $\mathcal{P}(t_0)$ which determines $T(t_0),H(t_0)$ as stated in the lemma and, due to \eqref{eq:HDef}, it determines $u(t_0,t_1)$ for a neighboring vertex $t_1$ as well. Then, the points $f^{(i)}(t_1)$ are determined and we can iterate the process to obtain an $n$-invariant curve. By interpreting $f^{(i)}(t_0)$ as an initial polygon with a monodromy given by $E$ we see that our construction is analogous to the construction of periodic B\"acklund transformations \cite{hoffmannSmokeRingFlow,pinkallSmokeRingFlow} (for more details, see also \cite{skewParNetsThesis}). In our planar case, this is known as discrete bicycle transformation \cite{discrete_bicycle} and constitutes a special case of cross-ratio dynamics \cite{affolter2023integrable,AFIT_dynamics,HMNP_periodicConformalMaps} on the twisted polygon $f^{(i)}(t_0)$. Note that if the curve itself and its transformations are periodic we find this dynamic in both lattice directions.
\end{rem}

We have investigated the polynomial \eqref{eq:poly} determined by the B\"acklund transformations corresponding to an $n$-invariant curve. 
For the converse direction, we will now reconstruct the B\"acklund transformations from a given polynomial. We will need some results about the factorization of \emph{quaternionic polynomials} which are polynomials $\sum_{i=0}^{n}\lambda^iC^{i}$ in a scalar variable $\lambda$ with quaternions $C^i$ as coefficients. 

\begin{prop}
\label{prop:izosimovFact}
\begin{enumerate}
\item Any quaternionic polynomial can be factorized into linear factors.
\item Let $\bm 1+\lambda q$ be right (or left) divisor of a quaternionic polynomial and denote the set of quaternions with the same trace and determinant as $q$ by $[q]$. Then, either
\begin{enumerate}
\item any right (or left) divisor $\bm 1+\lambda \tilde q$ with $\tilde{q}\in[q]$ agrees with $\bm1+\lambda q$ or
\item $\bm 1+\lambda \tilde{q}$ is a right/left divisor for all $\tilde{q}\in[q]$. This is the case if and only if $1+\lambda\tr q+\lambda^2\det q$ divides the polynomial.
\end{enumerate}
\item A quaternionic polynomial is called \emph{special} if $C^i\in\spannR{\bm i,\bm j}$ for odd $i$ and $C^i\in\spannR{\bm 1,\bm k}$ for even $i$. A special quaternionic polynomial can be factorized into special linear factors if and only if all roots of the determinant polynomial are imaginary.
\end{enumerate}
\end{prop}

Proof and more context can be found in \cite[Section 4.1]{izosimov2023recutting}. Such factorization results are also currently studied for other algebras in the context of kinematics (see, e.g., \cite{li2019factorization}). Note, that for an $n$-invariant curve $E^{-1}\mathcal{P}$ is a special polynomial. 

\begin{lem}
\label{lem:polyToInv}
A curve $f\in\CE(\EucDelta)$ is $n$-invariant if and only if at every vertex there exists a quaternionic %degree $n$ 
polynomial of the form $\mathcal{P}=C^0+\lambda C^1+...+\lambda^n C^n$ fulfilling
\begin{enumerate}
\item $C^n,C^{n-2},...\in\spannR{\bm i,\bm j}$ and $C^{n-1},C^{n-3},...\in\spannR{\bm 1,\bm k}$ and $C^0\neq0$,
\item equation \eqref{eq:polySim} and
\item $\det\mathcal{P}$ has only imaginary roots distinct from $\frac{i}{\EucDelta}$.
\end{enumerate}
\end{lem}

\begin{proof}
If the curve is $n$-invariant the polynomial $\mathcal{P}$ given in \eqref{eq:poly} fulfills (1)-(3). Conversely, consider a polynomial $\mathcal{P}$ fulfilling (1)-(3). We will iteratively construct a factorization at every point. Initially, with $E:=C^0$ consider $E^{-1}\mathcal{P}(t_0)$ which is a special polynomial with the same roots as~$\mathcal{P}(t_0)$ and, therefore, Proposition \ref{prop:izosimovFact} (3) implies the existence of a factorization
\begin{align*}
E^{-1}\mathcal{P}(t_0)=(\bm 1+\lambda v^{(n-1)}(t_0))\cdots(\bm 1+\lambda v^{(0)}(t_0))
\end{align*}
with $v^{(i)}(t_0)\in\spannR{\bm i,\bm j}$. These transport matrices are initial data for a sequence of B\"acklund transformations which we denote by $f=f^{(0)},...,f^{(n)}=\tilde f$. Note that if $\deg \mathcal{P}<n$ the sequence is completed by $n-\deg\mathcal{P}$ identity transformations (for which $v^{(i)}=0$). 

Now, we show that the equality $\mathcal{P}=E(\bm 1+\lambda v^{(n-1)})\cdots(\bm 1+\lambda v^{(0)})$ (which holds at $t_0$) gets preserved along the curve. If it holds at some point $f_0$ we have
\begin{align*}
(\bm 1+\lambda \tilde{u}_{\bar 10})(\bm 1+\lambda v^{(n-1)}_{\bar 1})\cdots(\bm 1+\lambda v^{(0)}_{\bar 1})=(\bm 1+\lambda v^{(n-1)}_0)\cdots(\bm 1+\lambda v^{(0)}_0)(\bm 1+\lambda {u}_{\bar 10})\\
=E^{-1}\mathcal{P}_0(\bm 1+\lambda {u}_{\bar 10})=E^{-1}(\bm 1+\lambda {u}_{\bar 10})\mathcal{P}_{\bar 1}=(\bm 1+\lambda E^{-1}{u}_{\bar 10}E)E^{-1}\mathcal{P}_{\bar 1}=:\mathcal{Q}
\end{align*}
The determinant of $\mathcal{Q}$ is not divided by $(1+\lambda^2\EucDelta^2)^2$ since, by (3), $1+\lambda^2\EucDelta^2$ does not divide $\det\mathcal{P}$. Therefore, $1+\lambda^2\EucDelta^2$ does not divide $\mathcal{Q}$ and by Proposition \ref{prop:izosimovFact} (2) the left divisor is unique implying $\tilde{u}_{\bar 10}=E^{-1}u_{\bar 10}E$ and $\mathcal{P}_{\bar 1}=E(\bm 1+\lambda v^{(n-1)}_{\bar 1})\cdots(\bm 1+\lambda v^{(0)}_{\bar 1})$. By symmetry of the construction, the same argument works in the other curve direction. 
Now, $\tilde{u}_{01}=E^{-1}u_{01}E$ implies that the curve is $n$-invariant with isometry given by $E$. 
\end{proof}

Note that addition of a real polynomial to $\mathcal{P}$ does not affect \eqref{eq:polySim}. Thus, by changing the real part of $\mathcal{P}$ (in a manner that preserves condition (1) and (3) as well) we will be able to find different sequences of B\"acklund transformations that make the curve $n$-invariant. 

Now, we have the tools to characterize $n$-invariant curves geometrically. 

\begin{ex}
One can check that for an equally sampled circle with center $X$ the polynomial given by the coefficients
    \begin{align*}
%\label{eq:inv1poly}
    C^0_0&:=E:=r_0+\bm k,\\
    C^1_0&:=2(F_0-X)\bm k
\end{align*}
fulfills the conditions of Lemma \ref{lem:polyToInv} for almost all $r_0\in\mathbb{R}$.
\end{ex}

The following statement is a version for plane curves of a known result for elastic rods \cite{hoffmannSmokeRingFlow,skewParallelogramNets}.

\begin{prop}
\label{prop:elastic2Inv}
A curve $f\in\CE$ is $2$-invariant if and only if it is an elastic curve.
\end{prop}

\begin{proof}
Let $f\in\CE(\EucDelta)$ be $2$-invariant and let $f=f^{(0)},f^{(1)},f^{(2)}=\tilde{f}\in\CE$ be a corresponding sequence of B\"acklund transformations with $\tilde{u}_{01}=E^{-1}u_{01}E$ for some $E\in\spannR{\bm i,\bm j}$ with $\det E=1$. 
We will find the directrix as the line of reflection belonging to the orientation-reversing isometry: All points $\frac{1}{2}(F+\tilde F)$ lie on a line with direction vector $E$ since
\begin{align*}
\frac{1}{2}(F_1+\tilde F_1)-\frac{1}{2}(F_0+\tilde F_0)=\frac{1}{2}(\tilde{u}_{01}+u_{01})=\frac{E^{-1}}{2}(u_{01}E+Eu_{01})=-\innerM{E,u_{01}}E^{-1}=\innerM{E,u_{01}}E.
\end{align*}
Now, the signed distance $d_0$ of a point $F_0$ to this line is given by
\begin{align*}
d_0=\innerM{\bm k E,F_0- \frac{1}{2}(F_0+\tilde{F}_0)}=-\frac12\innerM{\bm k E, v_0^{(0)}+v_0^{(1)}}=\frac14\left(\tr\bm k E( v_0^{(0)}+v_0^{(1)})\right)=\frac14\tr(\bm k B_0)=\frac{\beta\EucDelta}{4}\kappa_0
\end{align*}
where we used Lemma \ref{lem:fixedPointRelations} and, in particular, $B=C^1=E(v^{(0)}+v^{(1)})$. We conclude that $f$ is elastic by Corollary~\ref{cor:elasticProp}.

Conversely, consider an elastic curve $f\in\CE(\EucDelta)$ and its directrix with normalized direction vector $E$ such that the signed distance of each curve point to the directrix fulfills $d_0=\frac{\beta\EucDelta}{4}\kappa_0$ for some $\beta\neq0$. We define a polynomial~$\mathcal{P}_0$ of degree $2$ at each vertex by the coefficients
\begin{align}
\label{eq:inv2poly}
C^0_0&:=E,\nonumber\\
C^1_0&:=r_1-\beta \vec H_0,\nonumber\\
    C^2_0&:=\beta T_0+\EucDelta^2 E%\beta H_0^*u_{01}+\EucDelta^2 E=\beta u_{0\bar1}H_0^*+\EucDelta^2 E.
\end{align}
for some constant $r_1\in\mathbb{R}$. We need to check the conditions (1)-(3) of Lemma \ref{lem:polyToInv}. Condition (1) is clearly fulfilled. A calculation (see Appendix) shows that condition (2) holds as well and that condition (3) holds for some values $r_1$. Hence, by Lemma \ref{lem:polyToInv}, the curve is $2$-invariant.
\end{proof}

\begin{figure}[h!]
  \centering
  \includegraphics[width=0.3\textwidth]{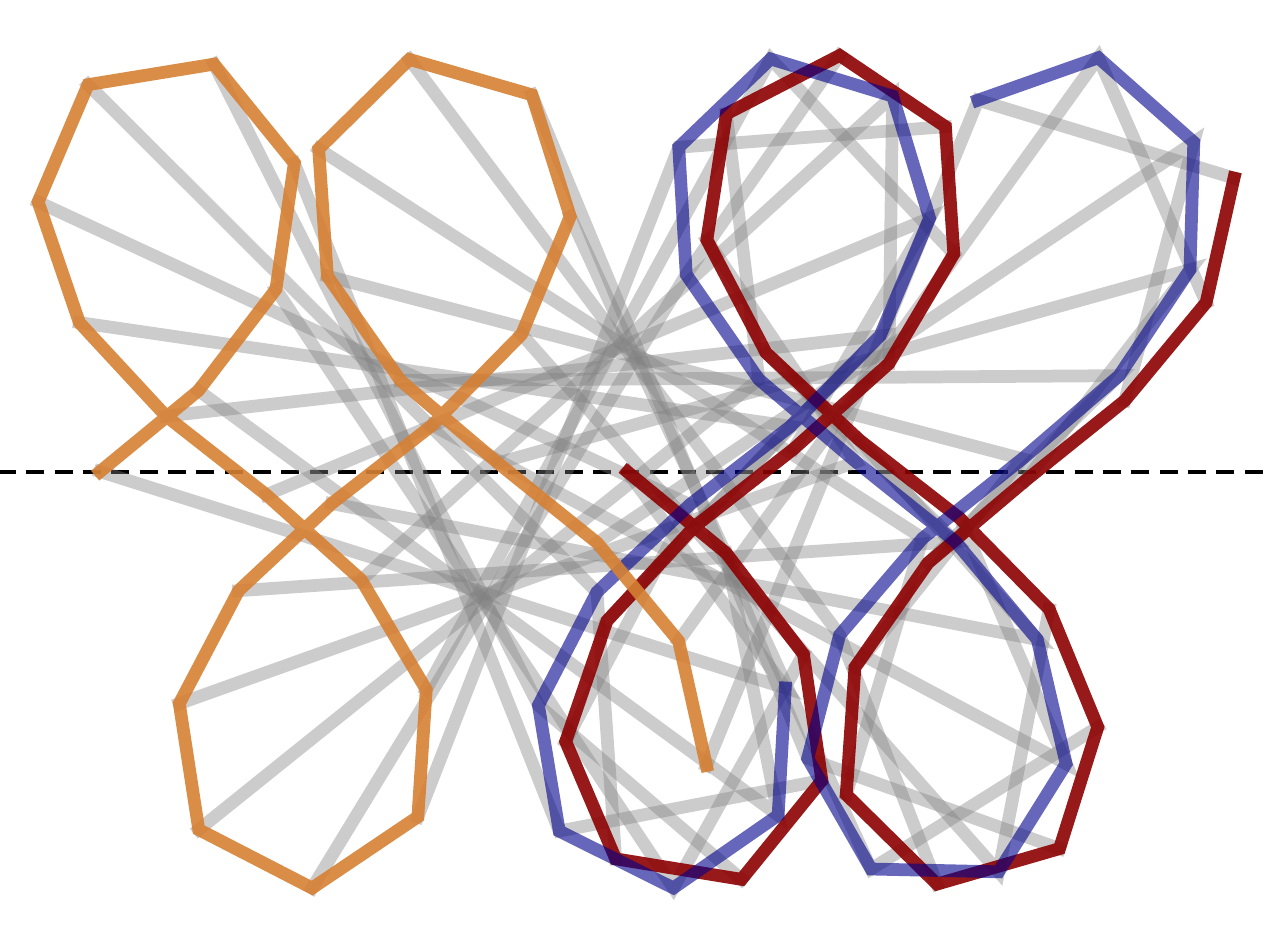}
  \includegraphics[width=0.3\textwidth]{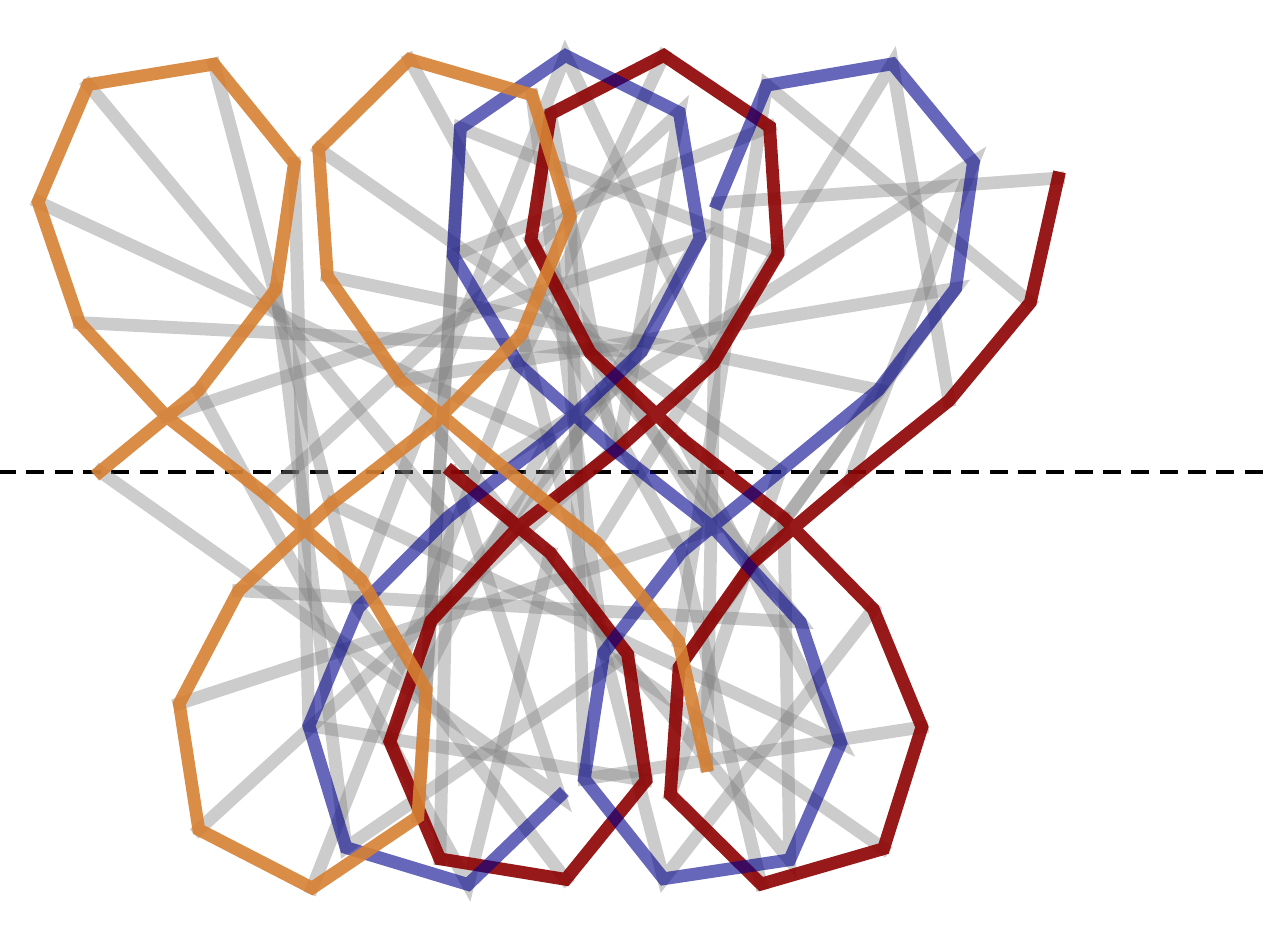}
  \includegraphics[width=0.3\textwidth]{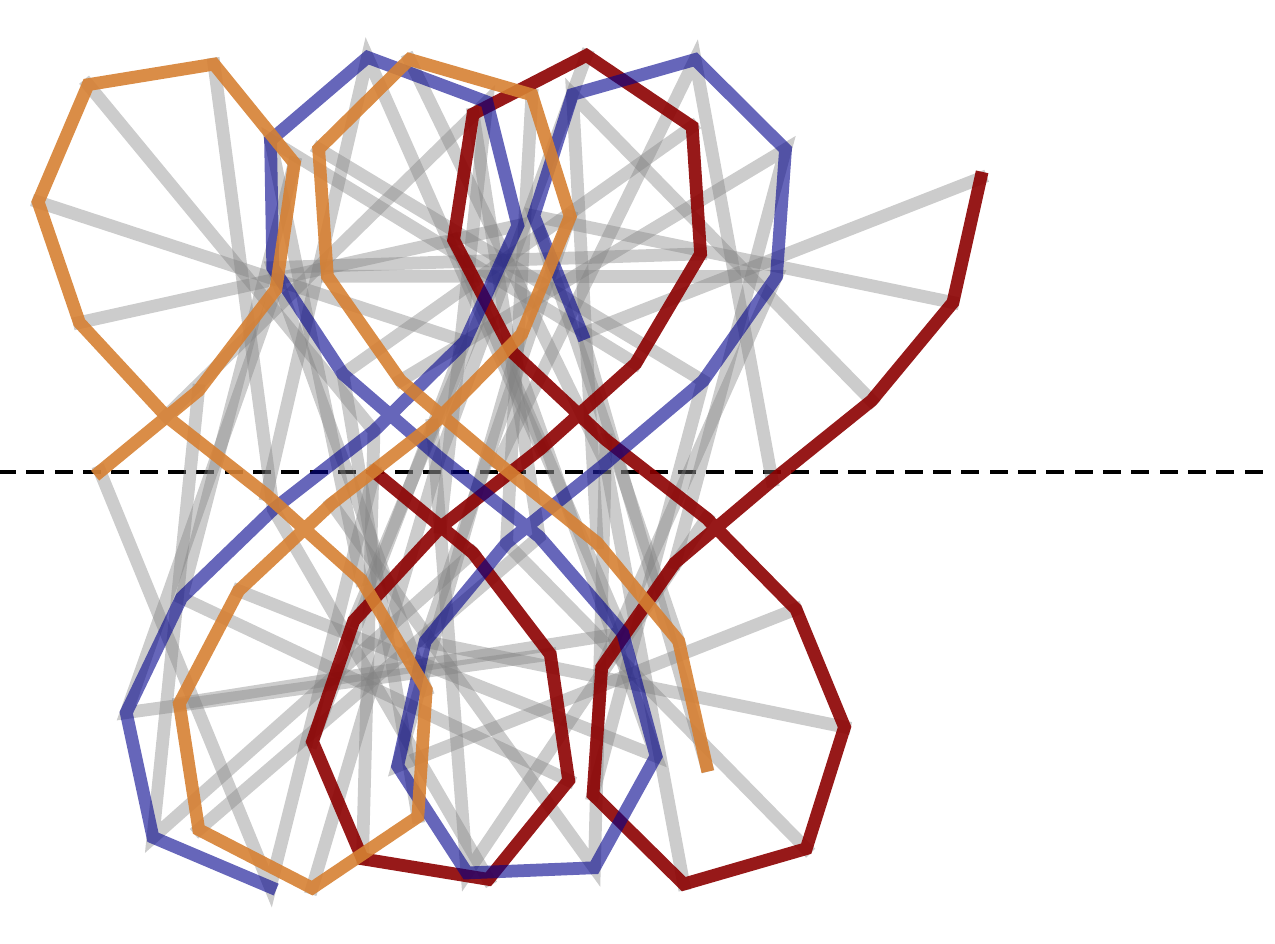}
  \caption{Different sequences of B\"acklund transformations for an elastic curve (orange). The last curve (red) can be obtained from the first curve (orange) by reflection in (and translation along) the directrix.}
  \label{fig:elasticaC1}
\end{figure}

As visible in Figure \ref{fig:flowVF} in the middle, the flow vector field $\vec C^1=\vec B=-\beta\vec H$ for a $2$-invariant curve induces a motion in Euclidean $3$-space. In fact $\vec H$ is the semi-discrete Hashimoto flow defined in \cite{lagrangeTop} and, hence, (as discussed in \cite{lagrangeTop,hoffmannSmokeRingFlow,skewParallelogramNets}) discrete elastica are the invariant curves of the Hashimoto flow.

\begin{prop}
\label{prop:ConstElastic3Inv}
A curve $f\in\CE$ is $3$-invariant if and only if it is a constrained elastic curve.
\end{prop}

\begin{proof}
We will distinguish the case where there isometry is a rotation ($\vec E\neq 0$) from the special case where it is a translation ($\vec E=0$). We will show that the case of a rotation corresponds to an area-constrained elastic curve while the case of a translation corresponds to an elastic curve or a circle. 

First, if the curve is $3$-invariant with a rotation as isometry we can assume normalized $\vec E=\bm k$. Consider the point $X=F-\frac{1}{2}\bm k C^1$. This point is in fact constant and will turn out to be the center of the directrix:
\begin{align}
\label{eq:centerConst}
X_1-X_0=u_{01}-\frac{1}{2}\bm k(C^1_1-C^1_0)
=u_{01}-\frac{1}{2}\bm k(u_{01}\bm k-\bm ku_{01})=0.
\end{align}
We calculate
\begin{align*}
\|F-X\|^2&=\det(F-X)=\frac{1}{4}\det C^1=\frac{1}{4}(\ninv_2-2\innerM{\vec C^0,\vec C^2})=\frac{1}{4}(\ninv_2-2\innerM{\bm k,\vec B+\EucDelta^2\bm k})\\
&=\frac{1}{4}(\ninv_2-2\EucDelta^2+\beta\EucDelta\kappa).
\end{align*}
Now, with $r^2:=\frac{1}{4}(\ninv_2-2\EucDelta^2)$, the squared tangential distance of $F$ to the circle with center $X$ and radius $r$ is
\begin{align}
\label{eq:constrElasticDist}
x^2=\|F-X\|^2-r^2=\frac{\beta\EucDelta}{4}\kappa.
\end{align}
Note, that $x^2$ and $r^2$ can be negative. The curve is area-constrained elastic by Corollary \ref{cor:ConstrElasticProp}.

For the converse, let $f\in\CE$ be an area-constrained elastic curve fulfilling \eqref{eq:constrElasticDist} for some circle with center $X$, radius $r$ and some constant $\beta\in\mathbb{R}$. We will define a polynomial $\mathcal{P}$ by
\begin{align}
\label{eq:inv3poly}
    C^0_0&:=E:=r_0+\bm k,\nonumber\\
    C^1_0&:=2(F_0-X)\bm k,\nonumber\\
    C^2_0&:=r_2-\beta\vec H_0+\EucDelta^2\bm k,\nonumber\\
    C^3_0&:=\beta T_0+\EucDelta^2 C^1_0. %\beta u_{01}H_0+\EucDelta^2 C^1_0=\beta H_0u_{\bar10}+\EucDelta^2 C^1_0.
\end{align}
for some constants $r_0,r_2\in\mathbb{R}$. We need to check the conditions (1)-(3) of Lemma \ref{lem:polyToInv}. Analogous to the $2$-invariant case, condition (1) is clearly fulfilled and a calculation (see Appendix) shows that condition (2) holds as well and that condition (3) holds for some $r_0,r_2$. Hence, the curve is $3$-invariant.

Now, consider an elastic curve. We know it is $2$-invariant with corresponding polynomial $\mathcal{P}$ given by \eqref{eq:inv2poly}. To show that the curve is $3$-invariant consider the polynomial $\hat{\mathcal P}:= r_0+r_2\lambda^2+\lambda \overrightarrow{\mathcal{P}}$ which fulfills conditions (1) and (2) of Lemma~\ref{lem:polyToInv} for any $r_0,r_2\in\mathbb{R}$ with $r_0\neq0$. Applying Lemma~\ref{lem:polyRoots}~(2) (in the Appendix) to the constant real polynomial $\det \hat{\mathcal P}$ yields that $\hat{\mathcal{P}}$ also fulfills condition (3) for some $r_0,r_2$. Hence, the elastic curve is $3$-invariant as well. 

Conversely, consider a curve which is $3$-invariant with a corresponding polynomial $\mathcal{P}$, isometry given by $E\in\mathbb{R}$ (i.e., the isometry is a translation) and $C^1\neq0$. To show that the curve is $2$-invariant consider the polynomial $\hat{\mathcal{P}}:=r_1\lambda+\frac{1}{\lambda}\overrightarrow{\mathcal{P}}$ which fulfills conditions (1) and (2) of Lemma \ref{lem:polyToInv} for all $r_1\in\mathbb{R}$. Again, applying Lemma \ref{lem:polyRoots} (1) (in the Appendix) to the constant real polynomial $\det \hat{\mathcal P}$ yields that $\hat{\mathcal{P}}$ also fulfills condition (3) for some $r_1$. Hence, the curve is $2$-invariant and, thus, elastic. 

Finally, in the special case of a $3$-invariant curve with $E\in\mathbb{R}$ and $C^1=0$ analogous argumentation to above for the polynomial $\hat{\mathcal{P}}:=r_0+\frac{1}{\lambda^2}\overrightarrow{\mathcal{P}}$ shows that the curve is $1$-invariant and, thus, a circle. 
\end{proof}

\begin{figure}[h!]
  \centering
  \includegraphics[width=0.32\textwidth]{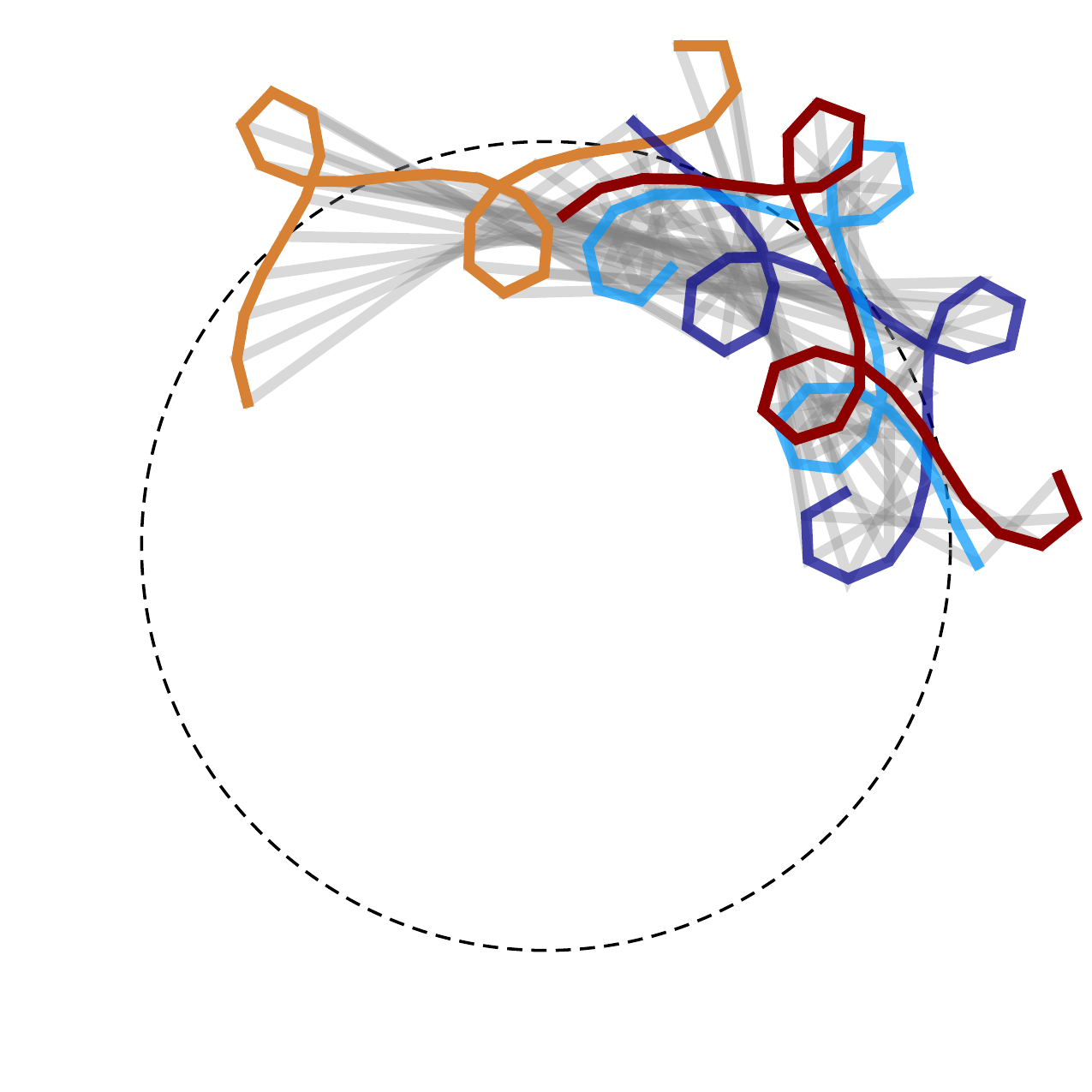}
  \includegraphics[width=0.32\textwidth]{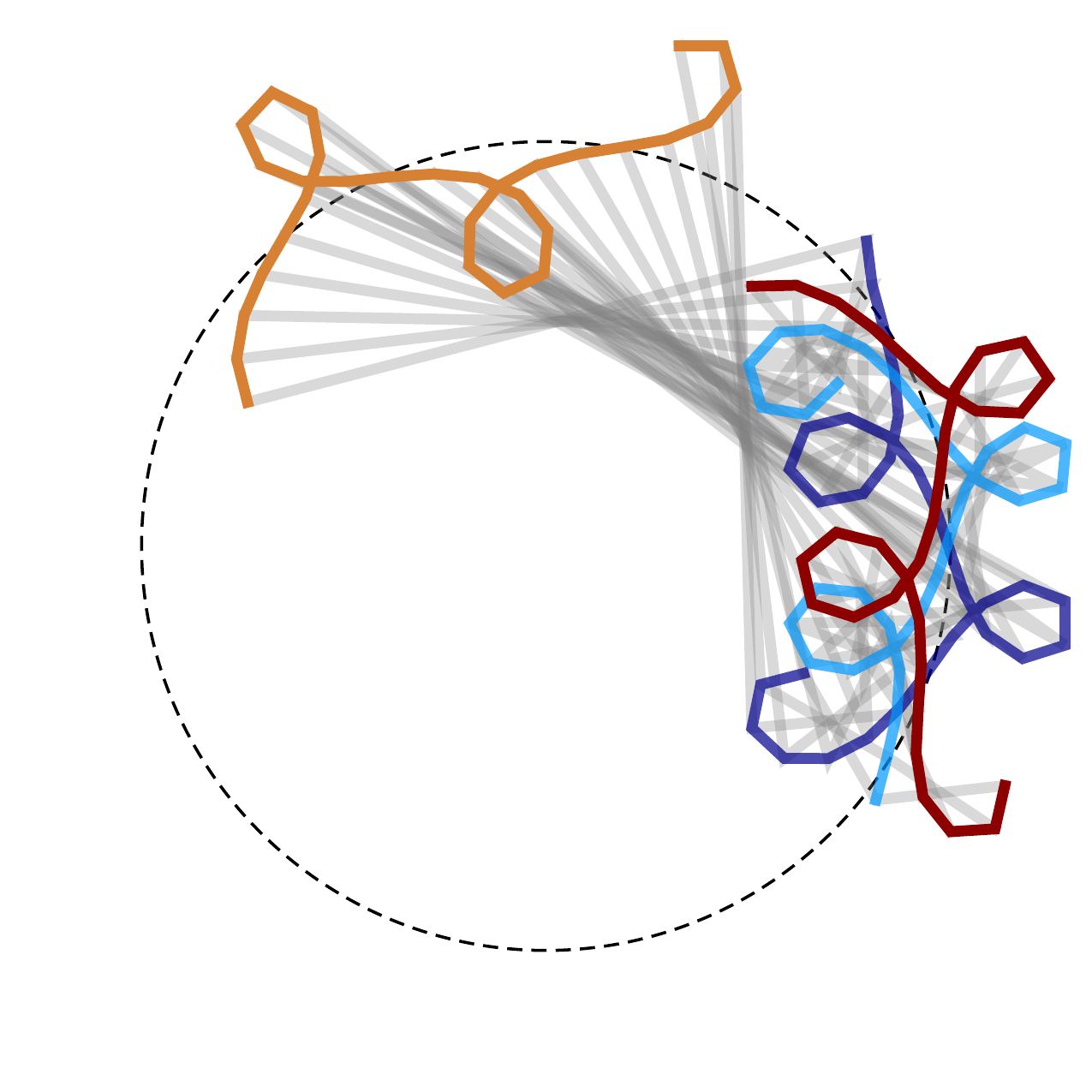}
  \includegraphics[width=0.32\textwidth]{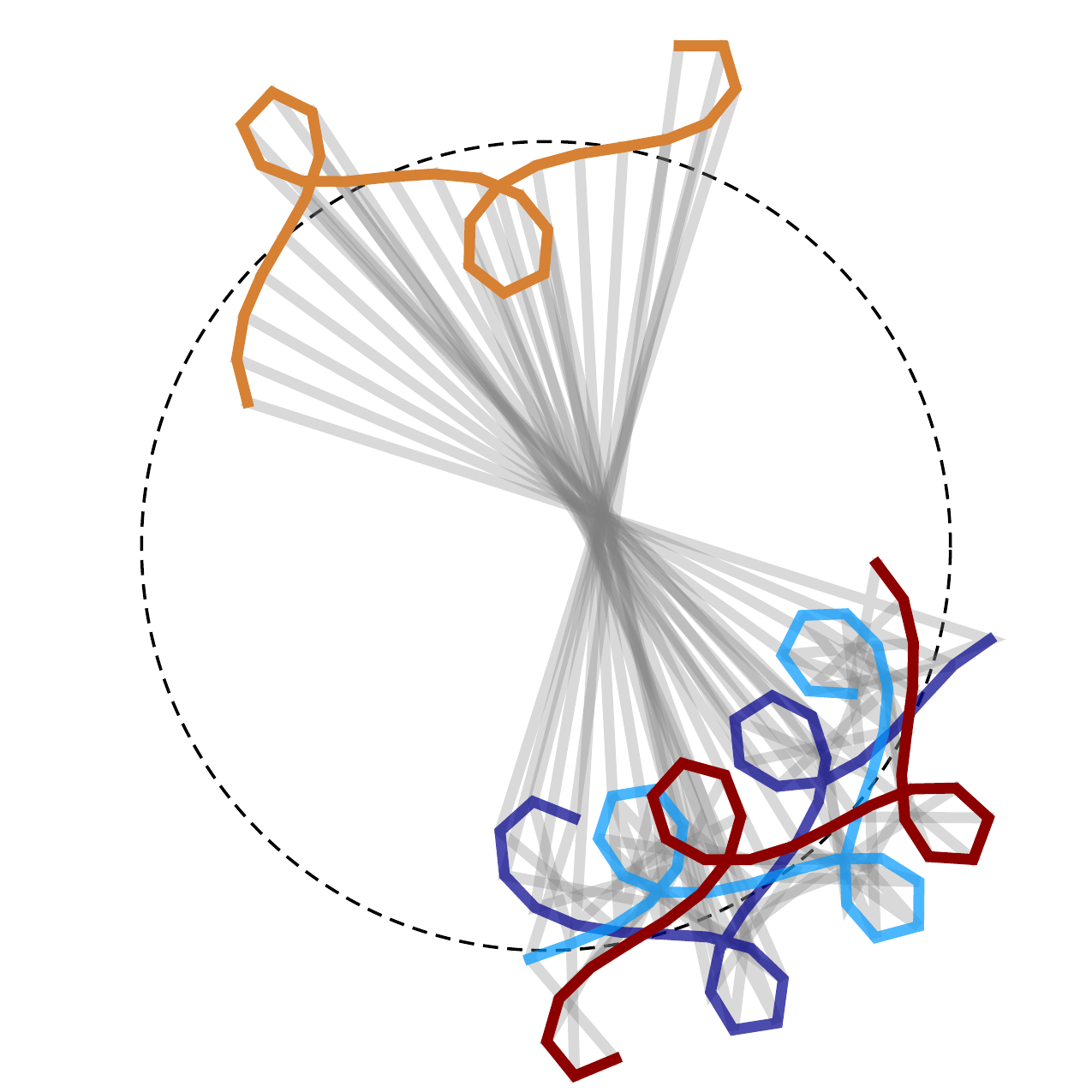}
  \caption{Different sequences of B\"acklund transformations for an area-constrained elastic curve (orange). The last curve (red) can be obtained from the first curve (orange) by rotation around the center of the directrix.}
\end{figure}

In \cite{HKtodaLattice}, constrained elastic curves are introduced as invariant curves of a semi-discrete mKdV flow. 
Indeed, the vector field $C^1$ induces a Euclidean motion on the curve and we will relate it to the mKdV flow given in \cite{HKtodaLattice}. For this consider the discrete Schwarzian derivative
\begin{align*}
S_{01}:=H_1H_0^*=(\bm 1+\frac{\EucDelta}{2}\kappa_1\bm k)(\bm 1-\frac{\EucDelta}{2}\kappa_0\bm k)=(1+\frac{\EucDelta^2}{4}\kappa_1\kappa_0)\bm 1+\frac{\EucDelta}{2}(\kappa_1-\kappa_0)\bm k.
\end{align*}
The semi-discrete tangent flow and mKdV flow are then given by
\begin{align*}
F^T_0:=T_0,\qquad F^{mKdV}_0:=\frac{1}{2}(S_{\bar10}+S_{01})T_0.
\end{align*}

\begin{prop}
\label{prop:mKdV}
An area-constrained elastic curve $f\in\CE$ is invariant under a linear combination of mKdV flow and tangent flow.
\end{prop}

\begin{proof}
The curve is $3$-invariant and the coefficient $C^1$ of the corresponding polynomial is the flow vector field which induces a rotation around the center of the directrix (see Figure \ref{fig:flowVF} on the left). A long calculation (see the Appendix) shows that
\begin{align*}
C_0^1=\frac{-\beta}{2\eta^2}F_0^{mKdV}+\frac{1}{2\beta\EucDelta^4}(\ninv_6-\EucDelta^4\ninv_2+2\EucDelta^6\theta_0)F_0^T
\end{align*}
which is a linear combination of mKdV and tangent flow with constant coefficients.
\end{proof}

%%%%%%%%%%%%%%%%%%%%%%%%%%%%%%%%%%%%
%
%
%%%%%%%%%%%%%%%%%%%%%%%%%%%%%%%%%%%%
\subsection{Associated family and $n$-invariant curves}
\label{sec:nonEucFlows}

We will now show that the associated family preserves the property of a curve of being $n$-invariant allowing us to complete the proof of Theorem \ref{thm:invariantCurvesConstElastica}. 
For this, we extend the associated family to sequences of B\"acklund transformations $f=f^{(0)},...,f^{(n)}$. As in Section \ref{sec:assoFamily}, we will assume a fixed initial point $F^{(0)}(t_0)=\bm 0$ for curves in Euclidean space and $F^{(0)}(t_0)=\bm k$ for curves in non-Euclidean space. Thus, the curve $f$ and its transformations $f^{(i)}$ are completely determined by the transport matrices $u$ and $v$. 

Recall \eqref{eq:compCond}, stating that multiplication of linear factors is path-independent for $\lambda\in\mathbb{C}$. We consider the value $\lambda$ to be \emph{admissible} if it fulfills the assumptions of Proposition \ref{associatedFamily} and if $\bm1\pm\lambda v^{(i)}$ is invertible. 
For admissible $\lambda\in\mathbb{C}$ the map $\Phi^\lambda$ given by 
\begin{align}
\label{eq:frameBL}
    \Phi^{(0),\lambda}(t_0)=\bm 1,\qquad
    \Phi_1^{(i),\lambda}=(\bm1+\lambda u^{(i)}_{01})\Phi_0^{(i),\lambda},\qquad
    \Phi_0^{(i+1),\lambda}=(\bm1+\lambda v^{(i)}_{0})\Phi_0^{(i),\lambda}
\end{align}
is well-defined. It extends the map $\Phi^\lambda:\V_n\to\matC$ given in Section \ref{sec:assoFamily} in transformation direction. Analogous to \eqref{eq:assoEdge}, we also define
\begin{align}
\label{eq:assoEdgeV}
v^{(i),\lambda}_{0}:=(\Phi_0^{(i+1),\lambda})^{-1}(\bm 1-\lambda v^{(i)}_{0})\Phi^{(i),\lambda}_0.
\end{align}
The pair $p^\lambda=(u^\lambda,v^\lambda)$ defines a map $\E_n\to\matC$.

\begin{lem}
\label{lem:familyPreservesSP}
Let $f^{(0)},...,f^{(n)}$ be a sequence of B\"acklund transformations. Then, the pair~$p^\lambda=(u^\lambda,v^\lambda)$ is a skew parallelogram net for admissible $\lambda\in\mathbb{C}$.
\end{lem}

\begin{proof}
We will use that $p=(u,v)$ is a skew parallelogram net. The multiplicative condition holds since
\begin{align*}
v^{(i),\lambda}_{1}u^{(i),\lambda}_{01}&=(\Phi_1^{(i+1),\lambda})^{-1}(\bm 1-\lambda v^{(i)}_{1})(\bm 1-\lambda u^{(i)}_{01})\Phi^{(i),\lambda}_0\\
&=(\Phi_1^{(i+1),\lambda})^{-1}(\bm 1-\lambda u^{(i+1)}_{01})(\bm 1-\lambda v^{(i)}_{0})\Phi^{(i),\lambda}_0=u^{(i+1),\lambda}_{01}v^{(i),\lambda}_{0}
\end{align*}
and the additive condition holds since
\begin{align*}
u^{(i),\lambda}_{01}+v^{(i),\lambda}_{1}&=(\Phi^{(i+1),\lambda}_1)^{-1}((\bm 1+\lambda v^{(i)}_{1})(\bm 1-\lambda u^{(i)}_{01})+(\bm 1-\lambda v^{(i)}_{1})(\bm 1+\lambda u^{(i)}_{01}))\Phi^{(i),\lambda}\\
&=2(\Phi^{(i+1),\lambda}_1)^{-1}(\bm 1-\lambda^2 v_1^{(i)}u^{(i)}_{01})\Phi^{(i),\lambda}
\end{align*}
coincides with
\begin{align*}
v^{(i),\lambda}_{0}+u^{(i+1),\lambda}_{01}&=(\Phi^{(i+1),\lambda}_1)^{-1}((\bm 1+\lambda u^{(i+1)}_{01})(\bm 1-\lambda v^{(i)}_{0})+(\bm 1-\lambda u^{(i+1)}_{01})(\bm 1+\lambda v^{(i)}_{0}))\Phi^{(i),\lambda}\\
&=2(\Phi^{(i+1),\lambda}_1)^{-1}(\bm 1-\lambda^2 u^{(i+1)}_{01}v_0^{(i)})\Phi^{(i),\lambda}.
\end{align*}
\end{proof}

Now, we extend the associated family to B\"acklund transformations.

\begin{prop}
Let $f^{(0)},...,f^{(n)}\in\C$ be a sequence of B\"acklund transformations of the curve $f=f^{(0)}$. The transport matrices
\begin{enumerate}
\item $u^\lambda, v^{\lambda}$ for $\Q=\QE$ and admissible $\lambda$,
\item $u^1, v^{1}$ for $\Q=\QS$,
\item $iu^1,iv^{1}$ for $\Q=\QH$
\end{enumerate}
define a sequence of curves $\asso^\lambda f^{(0)},...,\asso^\lambda f^{(n)}\in\mathcal{C}_{\mathcal{Q}_2}$ (with $\mathcal{C}_{\mathcal{Q}_2}$ as in Proposition \ref{associatedFamily}) which is again a sequence of B\"acklund transformations of $\asso^\lambda f=\asso^\lambda f^{(0)}$.
\end{prop}

\begin{proof}
The transport is closed around every quad since $p^\lambda=(u^\lambda, v^{\lambda})$ is a skew parallelogram net (Lemma~\ref{lem:familyPreservesSP}). 
Therefore, $\asso^\lambda f^{(i)}$ is well-defined. The same arguments as in the proof of Proposition~\ref{associatedFamily} show that $\asso^\lambda f^{(i)}\in\mathcal{C}_{\mathcal{Q}_2}$. Now, the property of $p^\lambda$ being a skew parallelogram net also implies that the curves constitute a sequence of B\"acklund transformations. One can easily check that our regularity assumptions stay preserved as well.
\end{proof}

\begin{prop}
\label{prop:familyInvariance}
The associated family preserves $n$-invariance.
\end{prop}

\begin{proof}
Let $f\in\C$ be $n$-invariant with a sequence of B\"acklund transformations $f=f^{(1)},...,f^{(n)}=\tilde{f}$ and a (split-)quaternion $E$ fulfilling $\tilde{u}_{01}=E^{-1}u_{01}E$. We will show that the sequence of B\"acklund transformations $\asso^\lambda f^{(1)},...,\asso^\lambda f^{(n)}$ preserves the shape of $\asso^\lambda f$ and, thus, causes $\asso^\lambda f$ to be $n$-invariant. 
For this, observe that the matrix $E^\lambda:=(\Phi^\lambda)^{-1}E\tilde{\Phi}^\lambda$ (with $\Phi^\lambda=\Phi^{(0),\lambda}$ and $\tilde\Phi^\lambda=\Phi^{(n),\lambda}$) is constant along the curve since
\begin{align*}
(\Phi_1^\lambda)^{-1}E\tilde{\Phi}^\lambda_1=(\Phi^\lambda_0)^{-1}(\bm 1+\lambda u_{01})^{-1}E(\bm 1+\lambda \tilde{u}_{01})\tilde{\Phi}_0^\lambda=(\Phi_0^\lambda)^{-1}E\tilde{\Phi}_0^\lambda.
\end{align*}
First, if $\Q=\mathbb{E}$ we have
\begin{align*}
\asso^\lambda\tilde{F}= (\tilde{\Phi}^\lambda)^{-1}\bm k \tilde{\Phi}^\lambda=\pm (\tilde{\Phi}^\lambda)^{-1}E^{-1}\bm k E\tilde{\Phi}^\lambda=\pm (E^\lambda)^{-1}\asso^\lambda F E^\lambda.
\end{align*}
The sign depends on whether $E$ commutes with $\bm k$ (orientation preserving $E\in \spannR{\bm1,\bm k}$) or anti-commutes with $\bm k$ (orientation reversing $E\in \spannR{\bm i,\bm j}$). 
Now, $\asso^\lambda f\in\mathcal{C}_{\mathcal{Q}_2}$ is $n$-invariant and the new isometry is given by
\begin{enumerate}
\item the quaternion $E^\lambda$ if $\lambda\in\mathbb{R}$ (and, hence, $\Q_2=\QS$).
\item the split-quaternion $E^\lambda$ if $\lambda\in i\mathbb{R}$ (and, hence, $\Q_2=\QH$) and $E\in \spannR{\bm1,\bm k}$.
\item the split-quaternion $i E^\lambda$ if $\lambda\in i\mathbb{R}$ (and, hence, $\Q_2=\QH$) and $E\in \spannR{\bm i,\bm j}$.
\end{enumerate}

On the other hand, if $\Q=\mathbb{S}$ or $\Q=\mathbb{H}$ we have $\asso^1f\in\CE$ and
\begin{align*}
 \tilde u^1_{01}=(\tilde{\Phi}^1_1)^{-1}(1-\tilde{u}_{01})\tilde{\Phi}_0^1=(\tilde{\Phi}^1_1)^{-1}E^{-1}(1-u_{01})E\tilde{\Phi}^1_0=(E^1)^{-1}u^1_{01} E^1
\end{align*}
and it only remains to show that $E^1$ (or a scalar multiple) is a quaternion in $\spannR{\bm i,\bm j}$ or $\spannR{\bm 1,\bm k}$. 
Indeed, using $\Phi^{(i),1}\bm k=F^{(i)}\Phi^{(i),1}$ and $\tilde F=\pm E^{-1}FE$, we have
\begin{align*}
E^\lambda \bm k=(\Phi^1)^{-1}E \tilde F\tilde{\Phi}^1=\pm (\Phi^1)^{-1}FE \tilde{\Phi}^1=\pm \bm k E^\lambda
\end{align*}
and, hence, $\asso^1f\in\CE$ is $n$-invariant with the isometry given by
\begin{enumerate}
\item the quaternion $E^1\in\spannR{\bm 1,\bm k}$ for orientation preserving $E$ and $\Q=\mathbb{S}$ or $\Q=\mathbb{H}$.
\item the quaternion $E^1\in\spannR{\bm i,\bm j}$ for orientation reversing $E$ and $\Q=\mathbb{S}$.
\item the quaternion $iE^1\in\spannR{\bm i,\bm j}$ for orientation reversing $E$ and $\Q=\mathbb{H}$.
\end{enumerate}
In all cases, one can also check that regularity of the sequence of transformations is preserved.
\end{proof}

Thus, $n$-invariant curves and their B\"acklund transformations in space forms can be obtained from corresponding $n$-invariant curves in Euclidean space. For example, the $n$-invariant curves in different space forms visible in Figure \ref{fig:nInvExamples} are related by the associated family. 

Finally, we combine our observations to prove our main Theorem.

\begin{proof}[Proof of Theorem \ref{thm:invariantCurvesConstElastica}.]
In Proposition \ref{prop:elastic2Inv} and \ref{prop:ConstElastic3Inv} we have shown the statement for curves in Euclidean space. 
Note that the associated family preserves the curvature equation \ref{eq:elasticaCurvature} of constrained elastic curves since $\asso^\lambda\dDelta\asso^\lambda \kappa=\dDelta\kappa$ (Proposition \ref{prop:familyCurvature}). 

Now, if a curve $f$ in a non-Euclidean space form is $2$-invariant then the associated Euclidean curve~$\asso^1f$ is also $2$-invariant (Proposition \ref{prop:familyInvariance}) and, therefore, elastic (Proposition \ref{prop:elastic2Inv}). Then, every curve $\asso^\lambda\circ \asso^1f$ is elastic and, since the associated family is reversible (Proposition \ref{prop:FamilyIdentity}), this includes~$f$. 

For the converse, we can use the same results in the other direction: If $f$ is elastic the associated Euclidean curve $\asso^1 f$ must be elastic and, therefore, it is $2$-invariant. All curves $\asso^\lambda\circ \asso^1 f$ are $2$-invariant and this includes the original curve $f$.

Analogously, using Proposition \ref{prop:ConstElastic3Inv}, we can show that $3$-invariant curves are exactly constrained elastic curves.
\end{proof}

%%%%%%%%%%%%%%%%%%%%%%%%%%%%%%%%%%%%%%%%
%
%
%%%%%%%%%%%%%%%%%%%%%%%%%%%%%%%%%%%%%%%%
\section{Appendix: Polynomials and Euclidean elastica}

We provide some calculations for elastic and area-constrained elastic curves in $\QE$. We start with two statements that we will use to show condition (3) of Lemma \ref{lem:polyToInv}.

\begin{lem}
\label{lem:polyRoots}
    \begin{enumerate}
        \item For any $\ninv_0,\ninv_2,\ninv_4\in\mathbb{R}$ with $\theta_0,\theta_4>0$ there exist infinitely many $r_1\in\mathbb{R}$ such that the real polynomial
        \begin{align*}
            \ninv_0+\lambda^2(r_1^2+\ninv_2)+\lambda^4\ninv_4
        \end{align*}
        has only imaginary roots.
        \item For any $\ninv_0,\ninv_2,\ninv_4,\ninv_6\in\mathbb{R}$ with $\theta_6>0$ there exist infinitely many combinations $r_0,r_2\in\mathbb{R}$ such that the real polynomial
        \begin{align*}
            r_0^2+\ninv_0+\lambda^2(2r_0r_2+\ninv_2)+\lambda^4(r_2^2+\ninv_4)+\lambda^6\ninv_6
        \end{align*}
        has only imaginary roots.
    \end{enumerate}
\end{lem}

\begin{proof}
For the first statement, consider the quadratic polynomial $\mathcal{Q}:=\ninv_0+\mu(r_1^2+\ninv_2)+\mu^2\ninv_4$. We will show that for sufficiently large $r_1^2$ it has only real negative roots, which proves the claim of the lemma. In fact, the quadratic discriminant
\begin{align*}
(r_1^2+\ninv_2)^2-4\ninv_0\ninv_4
\end{align*}
is positive for large $r_1^2$ and, hence, the roots of $\mathcal{Q}$ are real. Now, since the coefficients of $\mathcal{Q}$ are positive its roots are necessarily negative. 

For the second statement, consider the cubic polynomial $\mathcal{Q}:=r_0^2+\ninv_0+\mu(2r_0r_2+\ninv_2)+\mu^2(r_2^2+\ninv_4)+\mu^3\ninv_6$ and its cubic discriminant
\begin{align*}
18\ninv_6(r_2^2+\ninv_4)(2r_0r_2+\ninv_2)(r_0^2+\ninv_0)-4(r_2^2+\ninv_4)^3(r_0^2+\ninv_0)\\
+(r_2^2+\ninv_4)^2(2r_0r_2+\ninv_2)^2-4\ninv_6(2r_0r_2+\ninv_2)^3-27\ninv_6^2(r_0^2+\ninv_0)^2.
\end{align*}
We write $r_0=Rr_2$ and will show that for sufficiently large $R$ and $r_2$ the polynomial $\mathcal{Q}$ has only real negative roots which, again, proves the claim of the lemma. With $r_0=Rr_2$ the discriminant becomes a polynomial of degree six in $r_2$ with leading coefficient
\begin{align*}
-4\ninv_0+4\ninv_2R-4{\ninv_4}R^2+4\ninv_6R^3
\end{align*}
which is positive for sufficiently large $R$. Thus, if $r_2$ is sufficiently large as well, the discriminant is positive and $\mathcal{Q}$ has only real roots. Now, since the coefficients of $\mathcal{Q}$ are positive its roots are necessarily negative. 
\end{proof}

\begin{proof}[Calculations for polynomial \eqref{eq:inv2poly} for an elastic curve.]
Consider an elastic curve and the polynomial $\mathcal{P}$ with coefficients given by \eqref{eq:inv2poly}. 
We will check \eqref{eq:polySim} using its expanded version \eqref{eq:polyEvol}. Clearly, $C^0$ is constant. Also, we have
\begin{align*}
C^1_1-C^1_0&=-\beta(\vec H_1-\vec H_0)=-\frac{\beta\EucDelta}{2}(\kappa_1-\kappa_0)\bm k\\
&=-2(d_1-d_0)\bm k=-2\innerM{u_{01},\bm k E}\bm k=(u_{01}\bm kE+\bm k Eu_{01})\bm k=u_{01}E-Eu_{01}
\end{align*}
and, using \eqref{eq:HDef}, we have
\begin{align*}
C^2_1-C^2_0=\beta(T_1-T_0)=\beta(H_1u_{01}-u_{01}H_0)=u_{01}C^1_0-C^1_1u_{01}
\end{align*}
and
\begin{align*}
C^2_1u_{01}=\EucDelta^2(-\beta H_1+Eu_{01})=\EucDelta^2(-\beta H_0+u_{01}E)=u_{01}C^2_0.
\end{align*}
Hence, the polynomial $\mathcal{P}$ fulfills \eqref{eq:polySim} and, therefore, condition (2) of Lemma \ref{lem:polyToInv}. 
Then, $\det\vec P$ is constant along the curve and by denoting (as in \eqref{eq:theta}) the coefficients of $\det\vec P$ by $\theta_i$ it is a polynomial as discussed in Lemma \ref{lem:polyRoots} (1). Thus, it has only imaginary roots (also, $\det \mathcal P(\frac{i}{\EucDelta})\neq0$ is true for almost all $r_1$). We conclude that there exists $r_1$ such that for our polynomial $\mathcal{P}$ all conditions of Lemma \ref{lem:polyToInv} are met.
\end{proof}

\begin{proof}[Calculations for polynomial \eqref{eq:inv3poly} for an area-constrained elastic curve.]
Consider an area-constrained elastic curve and the polynomial $\mathcal{P}$ with coefficients given by \eqref{eq:inv3poly}. Again, we will check \eqref{eq:polySim} using its expanded version \eqref{eq:polyEvol}. Clearly, $C^0$ is constant and we have $C^1_1-C^1_0=2u_{01}\bm k=u_{01}E-Eu_{01}$. From this, by writing $C^1_0u_{01}=:a_{0}+b_{0}\bm k$ for $a_{0},b_{0}\in\mathbb{R}$ we obtain $C^1_1u_{01}=a_0+(b_0+2\EucDelta^2)\bm k$ and, thus, we can compute 
\begin{align*}
C^2_1-C^2_0&=-\beta(\vec{H}_1-\vec{H}_0)=-\frac{\beta\EucDelta}{2}(\kappa_1-\kappa_0)\bm k=-2(x_1^2-x_0^2)\bm k=-2(\det(F_1-X)-\det(F_0-X))\bm k\\
&=-\frac{1}{2}(\det C^1_1-\det C^1_0)\bm k=-\frac{1}{2\EucDelta^2}(a_0^2+(b_0+2\EucDelta^2)^2-a_0^2-b_0^2)\bm k=-2(b_0+\EucDelta^2)\bm k\\
&=a_0-b_0\bm k-a_0-(b_0+2\eta^2)\bm k=u_{01}C^1_0-C^1_1u_{01}.
\end{align*}
Next, using \eqref{eq:HDef}, we calculate
\begin{align*}
C^3_1-C^3_0&=\beta(T_1-T_0)+\EucDelta^2(C^1_1-C^1_0)=\beta\vec H_1u_{01}-\beta u_{01}\vec H_0+2\EucDelta^2u_{01}\bm k\\
&=(\beta\vec H_1-\eta^2\bm k)u_{01}- u_{01}(\beta\vec H_0-\eta^2\bm k)=u_{01}C^2_0-C^2_1u_{01}
\end{align*}
and
\begin{align*}
C^3_1u_{01}&=\EucDelta^2(-\beta H_1+C^1_1u_{01})=\EucDelta^2(C^2_1+C^1_1u_{01}-r_2-\beta-\EucDelta^2\bm k)\\
&=\EucDelta^2(C^2_0+u_{01}C^1_0-r_2-\beta-\EucDelta^2\bm k)=\EucDelta^2(-\beta H_0+u_{01}C^1_0)=u_{01}C^3_0.
\end{align*}
Hence, the polynomial $\mathcal{P}$ fulfills \eqref{eq:polySim} and, therefore, condition (2) of Lemma \ref{lem:polyToInv}. 
Analogous to the $2$-invariant case, $\det\vec P$ is constant along the curve and by denoting (as in \eqref{eq:theta}) the coefficients of $\det\vec P$ by $\theta_i$ it is a polynomial as discussed in Lemma \ref{lem:polyRoots} (2). Thus, it has only imaginary roots (also, again, $\det \mathcal P(\frac{i}{\EucDelta})\neq0$ is true for almost all $r_0,r_2$). We conclude that there exists $r_0,r_2$ such that for our polynomial $\mathcal{P}$ all conditions of Lemma \ref{lem:polyToInv} are met.
\end{proof}

\begin{proof}[Calculation of the coefficients of the flow $C^1$ written as linear combination of $F^{mKdV}$ and $F^T$.]

In this calculation we will use \eqref{eq:HDef}, Proposition \ref{prop:Hcurv} and the relations in Lemma \ref{lem:fixedPointRelations} for $3$-invariant curves multiple times. 

First, we have
\begin{align*}
F^{mKdV}_0=\frac12(H^*_{\bar1}u_{01}^{-1}+H_1u^{-1}_{\bar10})T_0^2=\frac{\det T_0}{2\eta^2}(u_{01}H_{\bar1}+H_1u_{\bar10}).
\end{align*}
Now, the curve is $3$-invariant and we consider a corresponding polynomial $\mathcal{P}$. % with $E=\bm k$. 
Using \eqref{eq:polyEvol}, we get
\begin{align*}
-\beta \vec H_1&=\vec B_1=\vec B_0+u_{01}C^1_0-C^1_1u_{01}=-\beta\vec H_0+u_{01}C^1_0-C^1_0u_{01}-2\EucDelta^2\vec E\\
-\beta \vec H_{\bar{1}}&=\vec B_{\bar{1}}=\vec B_0-u_{\bar10}C^1_0+C^1_{\bar{1}}u_{\bar10}=-\beta\vec H_0-u_{\bar10}C^1_0+C^1_0u_{\bar10}-2\EucDelta^2\vec E
\end{align*}
We combine 
\begin{align*}
\beta F^{mKdV}_0&=-\frac{\det T_0}{2\eta^2}(u_{01}(-\beta\vec H_{\bar1})+(-\beta\vec H_1)u_{\bar10}-\beta(u_{01}+u_{\bar10}))\\
&=-\frac{\det T_0}{2\eta^2}(-\beta u_{01}H_0- \beta H_0u_{\bar10}+2u_{01}C^1_0u_{\bar10}-(u_{01}u_{\bar10}+u_{\bar10}u_{01})C^1_0+2\EucDelta^2\vec E(u_{01}-u_{\bar10}))\\
&=-\frac{\det T_0}{\eta^2}(-\beta T_0+(u_{01}C^1_0u_{\bar10}-\EucDelta^2 C^1_0)+(\EucDelta^2+\innerM{u_{01},u_{\bar10}})C^1_0+\EucDelta^2\vec E(u_{01}-u_{\bar10})).
\end{align*}
Each term is a multiple of $C^1$ or $T$ since
\begin{align*}
u_{01}C^1_0u_{\bar10}-\EucDelta^2C^1_0=(T_0H_0^{-1}C^1_0+C^1_0H_0^{-1}T_0)u_{\bar10}=(H_0^{-1})^*(T_0C^1_0+C^1_0T_0)u_{\bar10}%=\frac{-2\innerM{ T_0,C^1_0}}{\det H_0}H_0u_{\bar10}
=\frac{-2\innerM{ T_0,C^1_0}}{\det H_0}T_0
\end{align*}
and, with $\vec E=:e\bm k$,
\begin{align*}
\vec E(u_{01}-u_{\bar10})=e\bm k(T_0H_0^{-1}-H_0^{-1}T_0)=\frac{e}{\det H_0}\bm k(H_0^{}-H_0^{*})T_0=-\frac{e\eta\kappa}{\det H_0}T_0.
\end{align*}
We have $\eta^2\det H_0=\det T_0=\frac{2\eta^4}{\eta^2+\innerM{u_{01},u_{\bar10}}}$ 
and, hence, we can combine 
\begin{align*}
\beta F^{mKdV}_0=\frac{\beta\det T_0}{\eta^2} T_0+2\innerM{ T_0,C^1_0}T_0-2\eta^2C^1_0+e\eta^3\kappa T_0=-2\eta^2C^1_0+\frac{a_0}{\beta\eta^2}F^T_0
\end{align*}
with coefficient $a_0:={\beta^2\det T_0}+2\beta\eta^2\innerM{ T_0,C^1_0}+e\beta\eta^5\kappa$. We need to show that $a$ is constant. For this, we express it in terms of our invariants $\ninv_j$ as
\begin{align*}
    a_0&=\det A_0+2\innerM{A_0,\eta^2C^1_0}-2\innerM{\eta^4\vec E,\vec B_0}\\&=\innerM{C^3_0-\EucDelta^2C^1_0,C_0^3+\EucDelta^2C_0^1}-2\EucDelta^4\innerM{\vec E,\vec C_0^2}+2\EucDelta^6\innerM{\vec E,\vec E}\\
&=\det C_0^3-\EucDelta^4\det C_0^1-2\EucDelta^4\innerM{\vec C^0,\vec C^2}+2\EucDelta^6\det\vec C^0=\ninv_6-\EucDelta^4\ninv_2+2\EucDelta^6\ninv_0.
\end{align*}
Thus, we showed that the flow vector field $C^1$ (which only induces a Euclidean motion) can be written as constant linear combination of mKdV- and tangent flow.
\end{proof}

%%%%%%%%%%%%%%%%%%%%%%%%%%%%%%%%%%%%
\bibliography{elasticspaceformbib.bib}
%%%%%%%%%%%%%%%%%%%%%%%%%%%%%%%%%%%%

\bigskip

\noindent\\[12pt]\begin{minipage}{14cm}
\textbf{Tim Hoffmann}, \href{mailto:tim.hoffmann@ma.tum.de}{tim.hoffmann@ma.tum.de}
\\Department of Mathematics,
Technical University of Munich, Germany
\end{minipage}
\\[12pt]\begin{minipage}{14cm}
\textbf{Jannik Steinmeier}, \href{mailto:jannik.steinmeier@tum.de}{jannik.steinmeier@tum.de}
\\Department of Mathematics,
Technical University of Munich, Germany
\end{minipage}
\\[12pt]\begin{minipage}{14cm}
\textbf{Gudrun Szewieczek}, \href{mailto:gudrun.szewieczek@uibk.ac.at}{gudrun.szewieczek@uibk.ac.at}
\\Department of Basic Sciences in Engineering, 
University of Innsbruck, Austria
\end{minipage}
\end{document}